\documentclass[11pt]{amsart}
\usepackage{amsfonts}
\usepackage{amsmath}
\usepackage{color}
\usepackage[mathscr]{eucal}
\usepackage{amscd, latexsym, graphicx, mathrsfs, enumerate}
\usepackage{amssymb}

\makeindex

\newtheorem{thm}{{{Theorem}}}[section]
\newtheorem{prop}[thm]{{Proposition}}
\newtheorem{lem}[thm]{{Lemma}}
\newtheorem{cor}[thm]{{Corollary}}

\newtheorem{remark}[thm]{Remark}
\numberwithin{equation}{section}
\newtheorem{Def}[equation]{Definition}

\def\N{\mathbb{N}}
\def\Z{\mathbb{Z}}
\def\Q{\mathbb{Q}}
\def\R{\mathbb{R}}
\def\C{\mathbb{C}}
\def\A{\mathbb{A}}

\def\diag{{\mathop{\mathrm{diag}}}}
\def\sgn{{\mathop{\mathrm{sgn}}}}
 
\def\vol{{\mathop{\mathrm{vol}}}}

\def\d{{\mathrm{d}}}
\def\ve{\varepsilon}
\def\tr{{\rm tr}}

\def\bsl{\backslash}
\def\inf{\infty}
\def\fin{{\mathrm{fin}}}

\def\reg{{\mathrm{reg}}}
\def\min{{\mathrm{min}}}

\def\bK{{\mathbf{K}}}

\def\cL{{\mathcal{L}}}

\def\bs{{\backslash}}
\def\ds{\displaystyle}
\def\ord{{\rm ord}}
\def\lra{{\longrightarrow}}
\def\G{{\Gamma}}

\def\trep{{\mathbf{1}}}

\numberwithin{equation}{section}

\title[An equidistribution theorem for holomorphic Siegel modular forms for $GSp_4$]
{An equidistribution theorem for holomorphic Siegel modular forms for $GSp_4$ and its applications}
\author{Henry H. Kim, Satoshi Wakatsuki and Takuya Yamauchi}
\thanks{The first author is partially supported by NSERC. The second author is partially supported by JSPS Grant-in-Aid for Scientific Research (No. 26800006, 25247001, 15K04795). The third author is partially supported by JSPS Grant-in-Aid for Scientific Research (C) No.15K04787.}
\subjclass[2010]{}
\address{Henry H. Kim \\
Department of mathematics \\
 University of Toronto \\
Toronto, Ontario M5S 2E4, CANADA \\
and Korea Institute for Advanced Study, Seoul, KOREA}
\email{henrykim@math.toronto.edu}

\address{Satoshi Wakatsuki \\
Faculty of Mathematics and Physics, Institute of Science and Engineering\\
Kanazawa University\\
Kakumamachi, Kanazawa, Ishikawa, 920-1192, JAPAN}
\email{wakatsuk@staff.kanazawa-u.ac.jp}

\address{Takuya Yamauchi \\ 
Mathematical Inst. Tohoku Univ.\\
 6-3,Aoba, Aramaki, Aoba-Ku, Sendai 980-8578, JAPAN}
\email{tyamauchi@m.tohoku.ac.jp}

\begin{document}
\begin{abstract}
We prove an equidistribution theorem for a family of holomorphic Siegel cusp forms for $GSp_4/\Q$ in  
various aspects. A main tool is Arthur's invariant trace formula. 
While Shin \cite{Shin} and Shin-Templier \cite{ST} used Euler-Poincar\'e functions at infinity in the formula, we 
use a pseudo-coefficient of a holomorphic discrete series to extract holomorphic Siegel cusp forms. Then the non-semisimple contributions arise from the geometric side, and this provides new second main terms $A, B_1$ in Theorem \ref{main} which have not been studied and 
a mysterious second term $B_2$ also appears in the second main term coming from the semisimple elements. 
Furthermore our explicit study enables us to treat more general aspects in the weight.  
We also give several applications 
including the vertical Sato-Tate theorem, the unboundedness of Hecke fields and low-lying zeros for degree 4 spinor $L$-functions and degree 5 standard $L$-functions of holomorphic Siegel cusp forms.  
\end{abstract}
\maketitle

\tableofcontents 

\section{Introduction}
Recently equidistribution theorems for a family of automorphic forms or automorphic representations of a 
reductive group $G$ over a number field have been studied in various aspects. The basic conceptual studies have been proposed by 
Sarnak and Serre in the case $G=GL_2$ (cf. \cite{Sarnak}, \cite{Serre}) though 
the case when $G$ has the compact symmetric space has been studied thoroughly. Since then, the trace formula for $G$ has become one of most powerful tools to analyze 
equidistribution theorems related to a distribution of eigenvalues for a fixed operator acting on a family of 
automorphic forms for $G$. 
After Selberg's celebrated work, the trace formula for Hecke operators on the 
space of automorphic representations for $G$ whose symmetric space is non-compact, has been developed by many people   
and it took much time and several stages to reach the current form which is invariant under conjugation  
(see \cite{Arthur1} and the references therein), which is called Arthur's invariant trace formula.  

In \cite{Shin}, Shin made good use of Sauvageot's important results \cite{sau} to show that the limit of an automorphic counting measure is the Plancherel measure. It implies the 
equidistribution of Hecke eigenvalues of automorphic forms on $G$. His key idea is to relate the automorphic counting measure with the spectral side of Arthur's invariant trace formula and then estimate its geometric side. After that, in \cite{ST}, he and Templier tackled a 
difficult problem of making it explicit to obtain a power-saving error term for the purpose of an application to low-lying zeros of a 
family of 
automorphic $L$-functions. To do that, as in \cite{Shin}, they used the geometric expansion of the error term 
mainly consisting of global coefficients, invariant distributions, and orbital integrals.   
Then they estimated each invariant in a uniform way.   
One of main difficulties seems to be a uniform boundedness of the orbital integrals. However since they chose the Euler-Poincar\'e function at infinity, they had only to consider the semisimple contributions in the geometric side. This would be a usual way to get around the non-semisimple contributions. Their results are fully general as much as possible within current knowledge under certain hypotheses such as Langlands functoriality conjecture but they work on all automorphic forms or automorphic representations 
which exhaust $L$-packets of the discrete series representations at infinity. 

It is quite natural to consider the automorphic counting measure on automorphic representations with a fixed discrete series representation at infinity. In \cite{Shin} Shin did not address this problem but he claimed that it may be possible to do so. In this paper we carry it out, namely,
we study equidistributions of Hecke eigenvalues of holomorphic Siegel modular forms of degree two, i.e., cuspidal representations which have a holomorphic discrete series at infinity. 
The strategy is similar to Shin \cite{Shin} and Shin-Templier \cite{ST}, but in addition to 
the semisimple contributions, 
we also have to estimate the non-semisimple contributions which have not been understood well in general. 
This makes the situation more difficult but as a payoff we will be able to observe the meaning of the second main terms $A, B_1$ coming from 
the non-semisimple contribution in 
comparison with the spectral side and a mysterious contribution $B_2$ also in the second main term from the semisimple part. 
Furthermore our results are unconditional in contrary to \cite{ST}.       
To explain our main results, we fix our notation. 

Let $G=GSp_4$ and $S'$ be a finite set of rational primes. Note that the symbol $S$ will be used for a finite set of places 
including $\infty$ in Section \ref{s3} and Section \ref{sec6}.  
Let $\A$ (resp. $\A_f$) be the ring of (resp. finite) adeles of $\Q$, $\Q_{S'}=\prod_{p\in S'}\Q_p$, and $\A^{S',\infty}=\prod'_{p\notin S'\cup \{\infty\}} \Bbb Q_p$. 
Let $\widehat{\Z}$ be the profinite completion of $\Z$. 
We denote by $\widehat{G(\Q_{S'})}$ the unitary dual of $G(\Q_{S'})=\prod_{p\in S'}G(\Q_p)$ equipped with Fell topology. 
Put $A_{G,\infty}=Z_{G}(\R)^\circ \simeq \R_{>0}$. Fix a Haar measure $\mu^{S',\infty}$ of $G(\A^{S',\infty})$ so that 
$\mu^{S',\infty}(G(\widehat{\Z}^{S'}))=1$, and let $U$ be a compact open subgroup of $G(\Bbb A^{S',\infty})$. 

Consider the 
algebraic representation 
$\xi=\xi_{\underline{k}}$ for $\underline{k}=(k_1,k_2),\ k_1\ge k_2 \ge 3$ as in (\ref{algebraic-char}), and let $D_{l_1,l_2}^{\rm hol}$ be the holomorphic discrete series of $G(\Bbb R)$ with the Harish-Chandra parameter $(l_1,l_2)=(k_1-1,k_2-2)$, and whose central character equal to $\chi_{\xi^{\vee}}$ on $A_{G,\infty}$. We choose the test function $f_{S'}=f_\xi f_U$ such that $f_\xi$ is a pseudo-coefficient of $D_{l_1,l_2}^{\rm hol}$. Then we define a measure on $\widehat{G(\Q_{S'})}$ by 

\begin{equation}\label{mu}
\widehat{\mu}_{U,\xi_{\underline{k}},D_{l_1,l_2}^{\rm hol}}:
=\frac{1}{{\rm vol}(G(\Q)A_{G,\infty}\bs G(\A))\cdot {\rm dim}\, \xi_{\underline{k}}} \sum_{\pi^0_{S'}\in \widehat{G(\Q_{S'})}}
\mu^{{S'},\infty}(U) m_{\rm cusp}(\pi^0_{S'}; U,\xi_{\underline{k}},D_{l_1,l_2}^{\rm hol})\delta_{\pi^0_{S'},\xi},
\end{equation}
where $\delta_{\pi^0_{S'},\xi}$ is a normalized Dirac delta measure supported on $\pi^0_{S'}$ with respect to the Plancherel measure 
$\widehat{\mu}^{{\rm pl}}_{S'}$ on $\widehat{G(\Q_{S'})}$ (see (\ref{n-dirac})), 
and for a given unitary representation $\pi^0_{S'}$ of $G(\Q_{S'})$, 

\begin{equation}\label{mult-more}
m_{\rm cusp}(\pi^0_{S'}; U,\xi_{\underline{k}},D_{l_1,l_2}^{\rm hol})= \sum_{\pi\in \Pi(G(\A))\atop \pi_{S'}\simeq \pi^0_{S'},\,\pi_\infty\simeq D_{l_1,l_2}^{\rm hol}}m_{\rm cusp}(\pi){\rm tr}(\pi^{{S'},\infty}(f_U))\cdot \tr(\pi_\infty(f_\xi)).
\end{equation} 

To state the equidistribution theorem, we need to introduce the Hecke algebra $C^\infty_c(G(\Q_{S'}))$  which is dense 
under the map $f_{S'}\mapsto \widehat{f}_{S'}=[\widehat{f}_{S'}:\pi_{S'}\mapsto {\rm tr} \pi_{S'}(f_{S'})]$ in 
a reasonable space $\mathcal{F}(\widehat{G(\Q_{S'})})$ consisting of suitable $\widehat{\mu}^{{\rm pl}}_{S'}$-measurable functions 
on $\widehat{G(\Q_{S'})}$ (see Section 2.3 of \cite{Shin} for that space). Put $K_p=G(\Z_p)$. The spherical Hecke algebra  
$H^{{\rm ur}}(G(\Q_p))$ is defined by the subalgebra consisting of $K_p$-bi-invariant functions in $C^\infty_c(G(\Q_p))$.  
It is well-known that 
$$H^{{\rm ur}}(G(\Q_p))=\C[h_{a_1,a_2,a_3} \mid a_1,a_2,a_3\in\Z , \quad a_3\geq a_1\geq  a_2\geq 0 ]$$
where  
$h_{a_1,a_2,a_3}$ is the characteristic function of $K_p\diag(p^{-a_1},p^{-a_2},p^{a_1-a_3},p^{a_2-a_3})K_p$. 
For any $\kappa \in \Z_{\ge 0}$ we denote by $H^{{\rm ur}}(G(\Q_p))^{\kappa}$ the $\C$-span of 
the functions $h_{a_1,a_2,a_3}$ satisfying $a_3\leq \kappa$.

In this paper, we restrict ourselves to $U=K(N)$ for $N\in \Z_{>0}$ which is the kernel of the natural quotient map from 
$G(\widehat{\Z})$ to $G(\widehat{\Z}/N\widehat{\Z})$. Put $\G(N)=Sp_4(\Q)\cap K(N)$. Let ${S'}=\{p\}$ for $p\nmid N$. 
Let $\widehat{\mu}^{{\rm pl}}_p=\widehat{\mu}^{{\rm pl}}_{\{p\}}$.
Then our main theorem is as follows: 

\begin{thm}\label{main}
Let $\{(K(N),\xi_{\underline{k}})\}$ be a family of weights and levels so that $p\nmid N$, $k_1\ge k_2\ge 3$ and 
$N+k_{1}+k_{2}\lra \infty$. For any $f\in H^{{\rm ur}}(G(\Q_p))^\kappa$, 
$$\lim_{k_1+k_2+N\to \infty} \widehat{\mu}_{K(N),\xi_{\underline{k}},D_{l_1,l_2}^{\rm hol}}(\widehat{f})=
\widehat{\mu}^{{\rm pl}}_p(\widehat{f}).
$$
More precisely, there exist constants $a,b,a'$ and $b'$ depending only on $G$ such that 

\begin{enumerate}
\item (level-aspect) Fix $k_1,k_2$. Then for $N\gg p^{10\kappa}$, 
$$\widehat{\mu}_{K(N),\xi_{\underline{k}},D_{l_1,l_2}^{\rm hol}}(\widehat{f})=\widehat{\mu}^{{\rm pl}}_p(\widehat{f})+
A+O(p^{a\kappa+b} \varphi(N)N^{-3}),\quad A=O(p^\kappa \varphi(N)N^{-2}),
$$
where $\varphi$ stands for Euler's phi function:
\item (weight-aspect) Fix $N$. Then as $k_1+k_2\to\infty$,
$$\widehat{\mu}_{K(N),\xi_{\underline{k}},D_{l_1,l_2}^{\rm hol}}(\widehat{f})=\widehat{\mu}^{{\rm pl}}_p(\widehat{f})+
B_1+B_2+O(\frac{p^{a'\kappa+b'}}{(k_1-k_2+1)(k_1-1)(k_2-2)}),$$
$$B_1=O(\frac{p^\kappa}{(k_1-1)(k_2-2)}),\ B_2=O(\frac{p^\kappa}{(k_1-k_2+1)(k_1+k_2-3)}).$$
\end{enumerate}
\end{thm}
\begin{remark} Here the second main terms $A,B_1$ come from non-semisimple contributions ($I_2(f)$ in Proposition \ref{level-est} and \ref{weight-est}), 
while 
$B_2$ comes from semisimple contributions ($I_3(f)$ in Proposition \ref{weight-est}).
\end{remark}

We explain several applications of Theorem \ref{main}. First, we consider 
the vertical Sato-Tate theorem which is formulated in terms of the classical setting. 
Let $S_{\underline{k}}(\G(N),\chi)$ be the space of classical holomorphic Siegel cusp forms of level $\G(N)$ 
with a central character $\chi:(\Z/N\Z)^\times\lra \C^\times$ and weight $\underline{k}=(k_1,k_2), k_1\ge k_2\ge 3$ 
(see (\ref{mfc}) of Section \ref{class}). 
For a prime $p\nmid N$, let $T(p^n)$ be the Hecke operator with the similitude $p^n$ (see Section \ref{HP}). 
Any eigenform with respect to $T(p^n)$ for any non-negative integer $n$ and any prime $p\nmid N$ is called 
a Hecke eigen cusp form.  
Let $HE_{\underline{k}}(\G(N),\chi)$ be a basis of $S_{\underline{k}}(\G(N),\chi)$ consisting of Hecke eigen forms outside $N$. 
Let $S_{\underline{k}}(\G(N),\chi)^{{\rm tm}}$ be the subspace of $S_{\underline{k}}(\G(N),\chi)$ 
generated by any Hecke eigen form $F$ outside $N$ so that $\pi_{F,p}$ is tempered for any $p\nmid N$. 
Put $HE_{\underline{k}}(\G(N),\chi)^{{\rm tm}}=S_{\underline{k}}(\G(N),\chi)^{{\rm tm}}\cap HE_{\underline{k}}(\G(N),\chi)$. 
For each $p\nmid N$, we fix a square root $\chi(p)^{\frac{1}{2}}$ in $\C$ and we write 
$\chi(p)^{-\frac{1}{2}}$ for $(\chi(p)^{\frac{1}{2}})^{-1}$. 
For a Hecke eigen form $F\in S_{\underline{k}}(\G(N),\chi)^{{\rm tm}}$, the Satake parameter of $\pi_{F,p}$ at $p\nmid N$ is given by
$\{\alpha_{0p}, \alpha_{0p}\alpha_{1p}, \alpha_{0p}\alpha_{1p},\alpha_{0p}\alpha_{1p}\alpha_{2p}\}$. Since $\alpha_{0p}^2\alpha_{1p}\alpha_{2p}=\chi(p)^2$, if we let $\alpha_{F,p}=\alpha_{0p}$ and $\beta_{F,p}=\alpha_{0p}\alpha_{1p}$, it can be written as
$\{\alpha^{\pm}_{F,p},\beta^{\pm}_{F,p}\}$ 
Then it follows from the temperedness that if we set
$$a_{F,p}:=\alpha_{F,p}\chi(p)^{-\frac{1}{2}}+\alpha^{-1}_{F,p}\chi(p)^{\frac{1}{2}},\ 
b_{F,p}:=\beta_{F,p}\chi(p)^{-\frac{1}{2}}+\beta^{-1}_{F,p}\chi(p)^{\frac{1}{2}},
$$
then 
$$a_{F,p}, b_{F,p}\in [-2,2].
$$ 
We introduce a suitable measure $$\mu_p=f_p(x,y)g^+_p(x,y)g^-_p(x,y)\cdot \mu^{{\rm ST}}_{\infty}$$ on $\Omega:=[-2,2]\times [-2,2]$, where  
$$f_p(x,y)=\frac{1}{\left(\left(\sqrt{p}+\frac{1}{\sqrt{p}}\right)^2-x^2\right)
\left(\left(\sqrt{p}+\frac{1}{\sqrt{p}}\right)^2-y^2\right)},\ 
 \mu^{{\rm ST}}_{\infty}=\frac{(x-y)^2}{\pi^2}\sqrt{1-\frac{x^2}{4}}\sqrt{1-\frac{y^2}{4}},$$
$$g^{\pm}_p(x,y)=\frac{1}{\left(\sqrt{p}+\frac{1}{\sqrt{p}}\right)^2-2\left(1+\frac{xy}{4}\pm\sqrt{1-\frac{x^2}{4}}\sqrt{1-\frac{y^2}{4}}\right)}.
$$
Note that the denominator of $g^{+}(x,y)g^-(x,y)$ is $x^2+y^2-xy(p+p^{-1})-4+(p+p^{-1})^2$. 
Let $C^0(\Omega,\R)$ be the space of $\R$-valued continuous functions on $\Omega$. To control non-tempered part of 
$S_{\underline{k}}(\G(N))$ we need to assume that $(N,11!)=1$ which might be unnecessary. 
Then we have  
\begin{thm}\label{Sato-Tate}
Let  $p\nmid N$, $k_1\geq k_2\geq 3$, and  
$N+k_{1}+k_{2}\lra \infty$ satisfying $(N,11!)=1$. 
Put $d^{{\rm tm}}_{\underline{k},N}(\chi)={\rm dim}S_{\underline{k}}(\G(N),\chi)^{{\rm tm}}$. Then 
the set 
$$\{ (a_{F,p},b_{F,p})\in \Omega\ |\  F\in HE_{\underline{k}}(\G(N),\chi)^{{\rm tm}}\}
$$
is $\mu_p$-equidistributed in $\Omega$, namely, for any $f\in C^0(\Omega,\R)$,

$$\lim_{N+k_1+k_2\to \infty \atop  p\nmid N,\,(N,11!)=1}\frac{1}{d^{{\rm tm}}_{\underline{k},N}(\chi)}\sum_{F\in 
HE_{\underline{k}}(\G(N),\chi)^{{\rm tm}}}f(a_{F,p},b_{F,p})=
\int_{\Omega}f(x,y)\mu_p.
$$
\end{thm}
It follows immediately from this that 
\begin{cor}\label{hecke1}
Let the notation be in Theorem \ref{Sato-Tate}. 
Then $$\sup_{N+k_1+k_2\to \infty\atop  (N,11!)=1}\{[\Q_F:\Q]\ |\ F\in 
HE_{\underline{k}}(\G(N),\chi)^{{\rm tm}}\}=\infty$$
where $\Q_F$ is the Hecke field of $F$ in Definition \ref{Hecke-f}.
\end{cor} 
The space $S_{\underline{k}}(\G(N),\chi)^{{\rm tm}}$ contains endoscopic lifts from elliptic modular forms which are  
called Yoshida lifts. 
The above corollary is still true even if we restrict $F$ to be a non-endoscopic lift but as mentioned before we set a condition on the level so that $N$ is 
coprime to $11!$ to control the conductor under the functoriality:
\begin{cor}\label{hecke2} Let the notation be as above. Then 
$$\sup_{N+k_1+k_2\to \infty \atop (N,11!)=1}\{[\Q_F:\Q]\ |\ F\in 
HE_{\underline{k}}(\G(N),\chi)^{{\rm tm}}:\text{non-endoscopic}\}=\infty.
$$ 
\end{cor} 
In the above corollary it is also interesting to remove among $HE_{\underline{k}}(\G(N),\chi)^{{\rm tm}}$ the lift from Hilbert modular forms over real quadratic fields and the symmetric cubic lift other than endoscopic lifts. 
We do not treat these two cases in this paper. It is possible to study the former one by using Base change but the latter one is difficult to analyze and need something beyond endoscopy. 

Next, we consider the distribution of the low-lying zeros of either spinor or standard $L$-functions for our family.  
For simplicity, denote $S_{\underline{k}}(\G(N),1)$, $HE_{\underline{k}}(\Gamma(N),1)$ by  $S_{\underline{k}}(N)$, 
$HE_{\underline{k}}(N)$, resp.
For $F\in S_{\underline{k}}(N)$ we denote the non-trivial zeros of $L(s,\pi_F, \ast),\ \ast\in\{{\rm Spin},{\rm St}\}$ by $\sigma_{F}=\frac{1}{2}+\sqrt{-1} \gamma_{F}$. We do not assume GRH, and hence 
$\gamma_F$ can be a complex number. Let $\phi$ be a Schwartz function which is even and whose transform has a compact support (so $\phi$ extends to an entire function). 
Define 
\begin{equation*}
D(\pi_F,\phi,\ast) = \sum_{\gamma_{F}}\phi\left( \frac{\gamma_{F}}{2\pi} \log c_{\underline{k},N}\right),
\end{equation*}
where $\log c_{\underline{k},N}=\ds\frac 1{{\rm dim}S_{\underline{k}}(N)}\ds\sum_{F\in HE_{\underline{k}}(\G(N))} \log c(F,*)$ and 
$c(F,*)$ stands for the analytic conductor (see Section \ref{L-function}).
Then we prove 
\begin{thm} \label{spin-st}
Let the notations be as in Theorem \ref{Sato-Tate}. Put $d_{\underline{k},N}={\rm dim}S_{\underline{k}}(N)$. Let $\phi$ be a Schwartz function which is even and whose Fourier transform has a support sufficiently smaller than 
$(-1,1)$. Then 
\begin{enumerate}
\item $\ds\lim_{N+k_1+k_2\to\infty \atop (N,11!)=1} \ds\frac 1{d_{\underline{k},N}} \ds\sum_{F\in  HE_{\underline{k}}(N)} D(\pi_F,\phi,{\rm Spin})=\hat\phi(0)+\frac 12 \phi(0)=
\int_\Bbb R \phi(x)W(G)(x)\, dx,$
\newline where $G=SO({\rm even})$, $SO({\rm odd})$, or $O$ type.
\item $\ds\lim_{N+k_1+k_2\to\infty \atop (N,11!)=1} \frac 1{d_{\underline{k},N}} \sum_{F\in HE_{\underline{k}}(N)} D(\pi_F,\phi,{\rm St})=\hat\phi(0)-\frac 12 \phi(0)=
\int_\Bbb R \phi(x)W(Sp)(x)\, dx,$
\newline where the corresponding density functions $W(G)$ are 
$$W(SO({\rm even}))(x)= 1+\frac{\sin 2\pi x}{ 2\pi x},\ W(SO({\rm odd}))(x)=1-\frac{\sin 2\pi x}{ 2\pi x} + \delta_0(x),$$
$W(O)(x)= 1+\ds\frac 12\delta_0(x)$, and 
$W(Sp)(x)=1-\ds\frac{\sin 2\pi x}{ 2\pi x}.$ 
\end{enumerate}

\end{thm}

\begin{remark} We stated our result only for $\Gamma(N)$ for simplicity. In weight aspect we expect that our result holds for other congruence subgroups such as $\Gamma_0(N)=\left\{\begin{pmatrix} A&B\\C&D\end{pmatrix} |\, C\equiv 0\ (N)\right\}$.
\end{remark}

\begin{remark} Kowalski-Saha-Tsimerman \cite{KST} (in level one case) and M. Dickson \cite{D} considered weighted one-level density of spinor $L$-functions of scalar-valued Siegel cusp forms, namely, let $\mathcal F_k(N)$ be a basis  of the space of Siegel eigen cusp forms of weight $k$ with respect to $\Gamma_0(N)$. Then
$$\lim_{N+k\to\infty} \frac 1{\ds\sum_{F\in \mathcal F_k(N)} \omega_{F,N,k}} \sum_{F\in \mathcal F_k(N)} \omega_{F,N,k} D(\pi_F,\phi, {\rm Spin})=\hat\phi(0)-\frac 12\phi(0)=\int_\Bbb R \phi(x)W(Sp)(x)\, dx.
$$
where 
$$\omega_{F,N,k}=\frac {\sqrt{\pi}(4\pi)^{3-2k}\Gamma(k-\frac 32)\Gamma(k-2) |A(F;E_2)|^2}{{\rm vol}(\Gamma_0(N)\backslash \Bbb H_2) 4\langle F,F\rangle},
$$
and $F(Z)=\ds\sum_{T>0} A(F;T) e^{2\pi i Tr(TZ)}$. So the symmetry type is $Sp$. Notice that the symmetry type is changed due to the weighted sum.
\end{remark}

\begin{remark}\label{Shin-T} If we apply the main results in 
Shin-Templier \cite {ST} for $G=GSp_4$, one-level density is for a family of all cuspidal representations whose infinity types are 
in the local $L$-packet consisting of both holomorphic discrete series $D^{{\rm hol}}_{l_1,l_2}$ and the large discrete series $D^{{\rm large}}_{l_1,l_2}$.
Hence their family consists of both holomorphic Siegel cusp forms and non-holomorphic forms. The global $L$-packet of a holomorphic Siegel cusp form always contains a non-holomorphic form. 
However, there are cuspidal representations with the large discrete series at the infinity whose global L-packet does not contain 
any cuspidal representation with a holomorphic discrete series at infinity. 

Let $\widetilde S_{\underline{k}}(N)$ be the set of isomorphic classes of cuspidal representations with the given infinity type in the local $L$-packet
$\{D^{{\rm hol}}_{l_1,l_2}, D^{{\rm large}}_{l_1,l_2}\}$.
 Then Shin-Templier showed that

$$\lim_{k_1+k_2+N\to\infty} \frac 1{|\widetilde S_{\underline{k}}(N)|} \sum_{\pi \in \widetilde S_{\underline{k}}(N)} D(\pi,\phi, \ast)=\hat\phi(0)\pm \frac 12\phi(0)=
\int_\Bbb R \phi(x)W(G)(x)\, dx,
$$
where $\pm$ is according to $*=$ Spin or St, resp. and $G$ is as in Theorem \ref{spin-st}. The symmetry type is the same because the 
$L$-functions of representations in the same $L$-packet remain the same. Also we can see that the contribution from endoscopic non-holomorphic forms is negligible, and in fact, we expect that it matches with non-semisimple contributions from the geometric side.
We elaborate it more in Section \ref{Shin}.
\end{remark}

Some experts in the trace formula may be wondering why we did not use the stable trace formula for proving the main theorem. A reason is that even if we use the stable trace formula for $G = GSp4$, we would have to take care of the fundamental lemma for Hecke elements under several transfers which might be feasible, but this gives rise to a conditional result in contrast to our main theorem. However, the method of the stabilization is still important in our proof. In fact, our argument of estimating non-semisimple terms is essentially the same as the stabilization. 

This paper is organized as follows. In Section 2, we recall the correspondence between classical holomorphic Siegel cusp forms and their adelic forms and their associated cuspidal automorphic representations of $GSp_4$. We also recall various subspaces and their dimensions. In Section 3 we quickly recall the classification of the algebraic representations of $GSp_4(\R)$ and 
adjust a central character to match with the holomorphic discrete series in question. 

In Section 4 we give a precise description of the spectral decomposition and residual spectrum and classification of CAP forms.
Even though we consider a pseudo-coefficient of holomorphic discrete series, there are non-trivial contributions from the residual spectrum and CAP forms. We also define a normalized 
(automorphic) counting measures which will be related to the Plancherel measure with error terms.  
 In Section 5, we estimate the global coefficients, 
 invariant distributions which are (limit) character formulas of holomorphic discrete series, and 
orbital integrals for spherical elements. 
The geometric side will be decomposed into seven main terms according to the shape of conjugacy classes. 
In Section 6, by using results in previous section we estimate those seven terms.  
In Section 7, we give a proof of the main theorem. Their interpretation in terms 
of classical Siegel modular forms will be given in Section 8. 
In Section 9, we review the spinor L-function and the standard L-function of automorphic representations for $GSp_4$ whose 
infinite component is the holomorphic discrete series and 
estimate the conductor which shows up in the functional equation. Then using the main theorem, we can estimate the sum of the coefficients of automorphic L-functions.  
In Sections 10, we apply results of Section 9 to obtain the one-level density result for degree 4 spin $L$-functions and 
degree 5 standard $L$-functions. 

In Section 11 we compare Shin's work with ours to explain the meaning of the second main terms in terms of 
automorphic representations. An analytic property of an infinite sum which is necessary in Section 10 will be 
proved in the appendix. 

\medskip

\textbf{Acknowledgments.} We would like to thank J. Arthur, P.-S. Chan ,  M. Miyauchi, R. Schmidt, S-W. Shin and N. Templier for helpful discussions. 
This work started when the authors visited RIMS in Kyoto in February 2015. We also 
discussed at KIAS in June 2015 and university of Toronto in September 2015. We thank these institutes for their incredible hospitality.  
   
\section{Preliminaries for holomorphic Siegel modular forms}\label{smf}
In this section, we recall holomorphic Siegel modular forms of genus 2.  
We refer to \cite{Taylor-thesis}, \cite{vG} for the classical setting and \cite{borel&jacquet} for the adelic setting. 
First we fix our notation. 
For any commutative ring $R$ with a unit, let $M_n(R)$ be the algebra consisting of all square matrices over $R$ with size $n$ and denote by $E_n$ (resp. $0_n$) the identity matrix (resp. zero matrix). We write $\diag(a_1,\dots,a_n)\in M_n(R)$ for the diagonal matrix whose entries are $a_1$, $\dots$, $a_n\in R$. We also write ${}^tX$ for the transpose of $X\in M_n(R)$.  

We define the generalized symplectic group by 
$$G=GSp_4=\Bigg\{X\in GL_4\ \Bigg|\ {}^tX\begin{pmatrix}
0_2 & E_2 \\
-E_2 & 0_2
\end{pmatrix}X=\nu(X)\begin{pmatrix}
0_2 & E_2 \\
-E_2 & 0_2
\end{pmatrix},\ \nu(X)\in GL_1 \Bigg\}$$
which is a smooth algebraic group scheme over $\Z$. The similitude $\nu$ defines the character 
$\nu:G\lra GL_1$ and defines the symplectic group $Sp_4:={\rm Ker}(\nu)$ which is of rank 2.  

Let $B$ be the standard Borel subgroup consisting of upper triangular matrices and $T$ be the split diagonal torus in $B$. 
For $i=1,2$, let $P_i=M_iN_i$ be the parabolic subgroup of $G=GSp_4$ defined by 
$$
M_1=\left\{\begin{pmatrix} A&0\\0& u {}^t A^{-1}\end{pmatrix}\ \Bigg|\  A\in GL_2,\, u\in GL_1\right\}\simeq GL_2\times GL_1,\,
N_1=\left\{ \begin{pmatrix} E_2&S\\0& E_2\end{pmatrix} \ \Bigg|\   S=\begin{pmatrix} a&b\\b&c\end{pmatrix} \right\},
$$
$$N_2=\left\{ \begin{pmatrix} E_2& A\\ 0&E_2\end{pmatrix} \begin{pmatrix} C&0\\0& C'\end{pmatrix} \Bigg| \, A=\begin{pmatrix} a&b\\b&0\end{pmatrix}, 
C=\begin{pmatrix} 1&d\\0&1\end{pmatrix}, C'=\begin{pmatrix} 1&0\\-d&1\end{pmatrix}\right\},
$$
\begin{eqnarray*}
M_2=\left\{ \begin{pmatrix} t&{}&{}&{}\\{}&a&{}&b\\{}&{}&\det(g)/t&{}\\{}&c&{}&d\end{pmatrix}\ \Bigg|\ t\in GL_1,\   
g=\begin{pmatrix} a&b\\c&d\end{pmatrix}\in GL_2 \right\}.
\end{eqnarray*}

We sometimes use the theory of elliptic modular forms on the upper half-plane $\mathbb{H}_1$ with respect to the following congruence subgroups of $SL_2(\Z)$: 

$$\G^1(N)=\{g\in SL_2(\Z)\ |\ g\equiv E_2\ {\rm mod}\ N \},$$
$$
\G^1_1(N)=\Bigg\{g=\begin{pmatrix}
a & b \\
c & d
\end{pmatrix}\in SL_2(\Z)\ \Bigg|\ a-1\equiv c\equiv 0\ {\rm mod}\ N \Bigg\},
$$
$$
\G^1_0(N)=\Bigg\{g=\begin{pmatrix}
a & b \\
c & d
\end{pmatrix}\in SL_2(\Z)\ \Bigg|\  c\equiv 0\ {\rm mod}\ N \Bigg\}.
$$

\subsection{Classical Siegel modular forms}\label{class}
Let $\mathcal{H}_2=\{Z\in M_2(\C)|\ {}^tZ=Z,\  {\rm Im}(Z)>0\}$ be the Siegel 
upper half-plane of degree 2. 
For a pair of non-negative integers $\underline{k}=(k_1,k_2)$, $k_1\ge k_2\ge 0$, we define the 
algebraic representation $\lambda_{\underline{k}}$ of $GL_2$ with the highest weight $\underline{k}$ by 
$$V_{\underline{k}}={\rm Sym}^{k_1-k_2}{\rm St}_2\otimes {\rm det}^{k_2} {\rm St}_2,
$$ 
where ${\rm St}_2$ is the standard representation of dimension 2 with the basis $\{e_1,e_2\}$. More explicitly, if $R$ is any ring, 
then $V_{\underline{k}}(R)=\ds\bigoplus_{i=0}^{k_1-k_2}Re^{k_1-k_2-i}_1\cdot e^i_2$ and for 
$g=\begin{pmatrix}
a & b \\
c & d
\end{pmatrix}
\in GL_2(R)$,  $\lambda_{\underline{k}}(g)$ acts on $V_{\underline{k}}(R)$ by 
$$g\cdot e^{k_1-k_2-i}_1\cdot e^i_2:=\det(g)^{k_2}(ae_1+ce_2)^{k_1-k_2-i}\cdot (be_1+de_2)^i.$$
We identify $V_{\underline{k}}(R)$ with $R^{\oplus (k_1-k_2+1)}$,
and
$\lambda_{\underline{k}}(g)$ with the representation matrix of $\lambda_{\underline{k}}(g)$ 
with respect to the above basis.

We have the action and the automorphy factor $J(\gamma,Z)$ by
\begin{equation}\label{sp4-action}
\gamma Z=(AZ+B)(CZ+D)^{-1}, \quad J(\gamma,Z)=CZ+D\in GL_2(\C),
\end{equation}
for
$\gamma=\begin{pmatrix}
A& B\\
C& D
\end{pmatrix}
\in GSp_4(\R)^+$ and $Z\in \mathcal{H}_2$.

For an integer $N\ge 1$, we define a principal congruence subgroup $\Gamma(N)$ to be 
the group consisting of the elements $g\in Sp_4(\Z)$ such that $g\equiv 1 \ {\rm mod}\ N$. 
For a $V_{\underline{k}}(\C)$-valued function $F$ on $\mathcal{H}_2$, the action of $\gamma \in GSp_4(\R)^\circ$ is defined by 
\begin{equation}\label{transformation}
F(Z)|[\gamma]_{\underline{k}}:=\lambda_{\underline{k}}(\nu(\gamma)J(\gamma,z)^{-1})F(\gamma Z).
\end{equation} 
The algebra of all $Sp_4(\R)$-invariant differential operators on $\mathcal{H}_2$ is 
isomorphic to $\C[\Omega,\Delta]$, the commutative polynomial ring of two variables, where
$\Omega$ is the degree 2 Casimir element, and $\Delta$ is the degree 4 element. (see Section 5 of \cite{KY} for the details 
and $\Omega=\Delta_1,\Delta=\Delta_2$ in the notation there.) 
It is easy to see  that 
\begin{equation}\label{Omega}
\Omega F=\frac{1}{12}((k_1-1)^2+(k_2-2)^2-5)F,\quad
\Delta F=
((k_1-1)(k_2-2))^2F,
\end{equation}
for all $V_{\underline{k}}(\C)$-valued holomorphic function $F$ on $\mathcal{H}_2$ when $k_1\ge k_2\ge 0$. 
When $(k_1,k_2)=(3,3),\ (3,1)$, or $(0,0)$, the two eigenvalues are 0 and 4. It is similar for $(k_1,3)$ and $(k_1,1)$. 
However if $k_2>3$, then there is no weight other than $(k_1,k_2)$ itself with the same eigenvalues. 
This observation would be related to the sets in (\ref{hol-coeff}). 

Suppose that $k_2\ge 3$ and let us introduce  
the Harish-Chandra parameter  $(l_1,l_2)$ for the holomorphic discrete series generated by (non-zero) $F$ of weight $(k_1,k_2)$. It has the relation  
$(k_1,k_2)=(l_1+1,l_2+2)$. 
Then in terms of the Harish-Chandra parameter we rewrite (\ref{Omega}) as follows: 
\begin{equation}\label{Omega1}
\Omega F=\frac{1}{12}(l^2_1+l^2_2-5)F,\quad
\Delta F=(l_1 l_2)^2F.
\end{equation}

Let us turn to the general case of $k_1\ge k_2\ge 0$. 
For a principal congruence subgroup $\G(N),\ N\ge 1$,  we say that a holomorphic function $F:\mathcal{H}_2\lra V_{\underline{k}}(\C)$ is 
a holomorphic Siegel modular form of weight $(k_1,k_2)$ with respect to $\G(N)$ if it satisfies
$F|[\gamma]_{\underline{k}}=F$ for any $\gamma\in\G(N)$. If we further impose the following condition,  
we call $F$ a holomorphic Siegel cusp form: 
$$\lim_{Z\to \partial\mathcal{H}_2}F|[\gamma]_{\underline{k}}(Z)=0\ {\rm for\ any\ }\gamma\in Sp_4(\Q)$$
where   $\partial\mathcal{H}_2$ stands for the boundary of the Satake compactification of $\mathcal{H}_2$. 
We denote by $M_{\underline{k}}(\G(N))$ (resp. $S_{\underline{k}}(\G(N))$) the space of such holomorphic Siegel modular 
(resp. cusp) forms. The space $E_{\underline{k}}(\G(N))$ of Eisenstein series is defined by the orthogonal complement of 
$S_{\underline{k}}(\G(N))$ in $M_{\underline{k}}(\G(N))$ with respect to Petersson  inner product. 
Hence we have 
$$M_{\underline{k}}(\G(N))=S_{\underline{k}}(\G(N))\oplus E_{\underline{k}}(\G(N)).$$
By Harish-Chandra (see Theorem 1.7 of \cite{borel&jacquet}.) the space $M_{\underline{k}}(\G(N))$ is finite dimensional. 
We will give an estimation of the dimension of $S_{\underline{k}}(\G(N))$ and of its specific subspace later on. 

The group $\G(N)$ contains the subgroup consisting of 
$\begin{pmatrix}
E_2& NS \\
0 & E_2
\end{pmatrix}$, 
$S={}^tS\in M_2(\Z)$. Hence for a given $F\in M_{\underline{k}}(\G(N))$, we have the Fourier expansion 
\begin{equation}\label{fourier}
F(Z)=\ds\sum_{T\in {\rm Sym}^2(\Z)_{\ge 0}}A_F(T,Y)e^{\frac{2\pi\sqrt{-1}}{N}{\rm tr}(TX)}, \quad Z=X+Y \sqrt{-1}\in \mathcal H_2,
\end{equation}
 where ${\rm Sym}^2(\Z)_{\ge 0}$ is the subset of $M_2(\Q)$ consisting of all symmetric matrices 
$\begin{pmatrix}
a& \frac{b}{2} \\
\frac{b}{2} & c
\end{pmatrix}$,
$a,b,c\in \Z$, which are semi-positive definite. 

\subsection{Hecke operators}\label{HP}
We define the Hecke operators on $M_{\underline{k}}(\Gamma(N))$ as in \cite{ev}:     
For any positive integer $n$ coprime to $N$, let 
$$\Delta_n(N):=\Bigg\{g\in M_4(\Z)\cap GSp_4(\Q)\ \Bigg|\ 
g\equiv \begin{pmatrix}
E_2 & 0   \\
0 & \nu(g)E_2 
\end{pmatrix}\, {\rm mod}\ N ,\ \nu(g)=n \Bigg\}.
$$
For $m\in \Delta_n(N)$, we introduce the action of the Hecke operators on $M_{\underline{k}}(\Gamma(N))$ 
 by 
\begin{equation}\label{hecke-ope}
T_m F(Z):=\nu(m)^{\frac{k_1+k_2}{2}-3}\ds\sum_{\alpha\in \Gamma(N)\backslash\Gamma(N) m
\Gamma(N)} F(Z)|[(\nu(m)^{-\frac{1}{2}}\alpha]_{\underline{k}},
\end{equation}
and for any positive integer $n$, put  
$$T(n):=\sum_{m\in \G(N)\backslash \Delta_n(N)}T_m.
$$
These actions preserve the space $S_{\underline{k}}(\G(N))$. 
For $t_1={\rm diag}(1,1,p,p),\ t_2={\rm diag}(1,p,p^2,p)$ for a prime $p$, put $T_{j,p^k}:=T_{t^k_j}\ j=1,2$, $k\in\Z_{>0}$ and 
fix $\widetilde{S}_{p^i,1}, \widetilde{S}_{p^i,p^i}\in Sp_4(\Z)$ so that 
$\widetilde{S}_{p^i,1}\equiv {\rm diag}(p^{-i},1,1,p^i)\ {\rm mod}\ N$ and 
$\widetilde{S}_{p^i,p^i}\equiv {\rm diag}(p^{-i},p^{-i},p^i,p^i)\ {\rm mod}\ N$ for each $i\in\Z_{>0}$. Put $R_{p^i}:=\widetilde{S}_{p^i,p^i}T_{p^i E_4}=p^{i(k_1+k_2-6)}\widetilde{S}_{p^i,p^i}$ and note that 
it commutes with any Hecke operator. 
Then we see that 
\begin{equation}\label{hecke-operators}
\begin{split}
& T(p)=T_{1,p},\quad T(p^2)=T_{1,p^2}+T_{2,p}+R_p, \\
& T^2_{1,p}-T(p^2)-p^2 R_p =p (T_{2,p}+(1+p^2)R_p), \\
& T^2_{2,p} =T_{\diag(1,p^2,p^4,p^2)}+ (p+1)T_{\diag(p,p,p^3,p^3)}+(p^2-1)T_{\diag(p,p^2,p^3,p^2)}+(p^4+p^3+p+1) R_{p^2}. 
\end{split}
\end{equation}
The last relation is obtained as follows: First, we note that $M\in T_{2,p}$ if and only if $r_p(M)=1$, where $r_p(M)$ is the rank of 
$M$ mod $p$. Then 
$$T_{2,p}^2=\sum_{M} t(K M K) KMK,
$$
where $M$ runs over all double coset representatives of $K\backslash \Delta_{p^4}(N)/K$, and $t(KMK)$ is the number of left coset representatives $A$ of $T_{2,p}/K$ such that $A^{-1}M\in T_{2,p}$, i.e., $r_p(A^{-1}M)=1$. Now $M$ runs over the following elements; 
\begin{eqnarray*}
&& \diag(1,1,p^4,p^4),\, \diag(1,p,p^4,p^3),\, \diag(1,p^2,p^4,p^2), \\
&& \diag(p,p,p^3,p^3),\, \diag(p,p^2,p^3,p^2), \, \diag(p^2,p^2,p^2,p^2).
\end{eqnarray*}

Then we may use the explicit left coset representatives in p.189-190 of 
\cite{RS}.  
  
\medskip

By (1.15), p.435 of \cite{ev} we have the following relation:
\begin{equation}\label{zeta-hecke}
\begin{split}
&\sum_{n=0}^\infty T(p^n)t^n=\frac{P_{p}(t)}{Q_{p}(t)},\ P_{p}(t)=1-p^2 R_p t^2,\\
&Q_{p}(t)=1-T(p)t+\{T(p)^2-T(p^2)-p^2 R_p\}t^2-p^3 R_p T(p)t^3+p^{6}R_{p^2}t^4.
\end{split}
\end{equation}
 
The finite group $Sp_4(\Z/N\Z)\simeq Sp_4(\Z)/\G(N)$ acts on  
$M_{\underline{k}}(\G(N))$ by $F\mapsto F|[\tilde{\gamma}]_{\underline{k}}$ if 
we fix a lift $\tilde{\gamma}$ of $\gamma\in Sp_4(\Z/N\Z)$. 
We denote this action by the same notation $F|[\gamma]_{\underline{k}}$. 
This action does not depend on the choice of lifts of $\gamma$. 
The diagonal subgroup of $Sp_4(\Z/N\Z)$ is isomorphic to $(\Z/N\Z)^\times\times (\Z/N\Z)^\times$ by sending 
$S_{a,b}:={\rm diag}(a^{-1},b^{-1},a,b)$ to 
$(a,b)$ and it also acts on $M_{\underline{k}}(\G(N))$, factoring through the action of $Sp_4(\Z/N\Z)$. Then we have the character decomposition   
\begin{equation}\label{character}
M_{\underline{k}}(\G(N))=\bigoplus_{\chi_1,\chi_2:(\Z/N\Z)^\times\lra\C^\times}M_{\underline{k}}(\G(N),\chi_1,\chi_2),
\end{equation}
where  
$M_{\underline{k}}(\G(N),\chi_1,\chi_2)=\{F\in M_{\underline{k}}(\G(N))\ |\ F|[S_{a,1}]_{\underline{k}}=\chi_1(a)F \ {\rm and}\  F|[S_{a,a}]_{\underline{k}}=\chi_2(a)F \}.
$
It is easy to see that the Hecke operators preserve $M_{\underline{k}}(\G(N),\chi_1,\chi_2)$ (cf. \cite{manni&top}). 
We should remark that in order that $M_{\underline{k}}(\G(N),\chi_1,\chi_2)\ne 0$,
 the weight $(k_1,k_2)$ has to satisfy the parity condition 
\begin{equation}\label{parity}
\chi_2(-1)=(-1)^{k_1+k_2}. 
\end{equation}
In particular, if $N=1$ or $2$, then $\chi_2$ has to be trivial and therefore  \begin{equation}\label{parity1}
k_1+k_2\equiv 0\ {\rm mod}\ 2. 
\end{equation}
Put 
\begin{equation}\label{mfc}
\begin{split}
M_{\underline{k}}(\G(N),\chi):=\bigoplus_{\chi_1:(\Z/N\Z)^\times\lra\C^\times}M_{\underline{k}}(\G(N),\chi_1,\chi)\\
S_{\underline{k}}(\G(N),\chi):=M_{\underline{k}}(\G(N),\chi)\cap  S_{\underline{k}}(\G(N)).
\end{split}
\end{equation}

Throughout this paper, we assume this parity condition (\ref{parity}) for $F$. Let 
$$F(Z)=\ds\sum_{T\in {\rm Sym}^2(\Z)_{\ge 0}}A_F(T,Y)e^{\frac{2\pi\sqrt{-1}}{N}{\rm tr}(TX)}\in 
M_{\underline{k}}(\G(N),\chi),\quad Z=X+Y \sqrt{-1},
$$ 
be an 
eigenform for all $T(p^i),\ p\nmid N,\ i\in\N$ with eigenvalues $\lambda_F(p^i)$, i.e.,
$$
T(p^i)F=\lambda_F(p^i)F.
$$
By (\ref{zeta-hecke}) we have the following relation:
\begin{equation}\begin{split}
&\sum_{n=0}^\infty\lambda_F(p^n)t^n=\frac{P_{F,p}(t)}{Q_{F,p}(t)},\ P_{F,p}(t)=1-p^{\mu-1}\chi(p)t^2,\\
&Q_{F,p}(t)=1-\lambda_F(p)T+\{\lambda_F(p)^2-\lambda_F(p^2)-p^{\mu-1}\chi(p)\}t^2-\chi(p)p^{\mu}\lambda_F(p)t^3+\chi(p)^2p^{2\mu}t^4
\end{split}
\end{equation}
where $\mu=k_1+k_2-3$.
Then we define the partial spinor L-function of $F$ by 
$$L^{N}(s,{\rm spin},F):=\ds\prod_{p\nmid N} Q_{F,p}(p^{-s})^{-1}.
$$

\begin{Def}\label{Hecke-f}
We define the Hecke field $\Q_F$ of $F$ 
by 
$$\Q_F=\Q( \lambda_F(p^i), \, \chi_j(p), \, \text{$j=1,2$ for $p\nmid N$ and $i\ge 0$}).
$$
\end{Def}

It is well-known (cf. \cite{Taylor-thesis}) that $\Q_F$ has a finite degree over $\Q$ since 
$M_{\underline{k}}(\G(N))$ has a $\Q$-structure $L_\Q$ which is preserved by Hecke actions and the 
Hecke algebra inside ${\rm End}_\Q(L_\Q)$ is finitely generated.    

\subsection{Adelic forms}\label{adelic}
For a positive integer $N$, let $K(N)$ be the group consisting of the elements 
$g\in GSp_4(\widehat{\Z})$ such that $g\equiv E_4\ {\rm mod}\ N$. Then we see that  $\G(N)=Sp_4(\Q)\cap K(N)$ and $\nu(K(N))=1+N\widehat{\Z}$. 
Then it follows from the strong approximation theorem for $Sp_4$ that 
\begin{eqnarray}\label{strong}
 G(\A) &=& \coprod_{1\le a < N \atop (a,N)=1}G(\Q)G(\R)^+ d_a K(N)=\coprod_{1\le a < N \atop (a,N)=1}G(\Q)Z_G(\R)^+Sp_4(\R)d_a K(N) \label,
\end{eqnarray}
where $d_a$ is the diagonal matrix such that $(d_a)_p={\rm diag}(a,a,1,1)$ if $p|N$, $(d_a)_p=E_4$ otherwise. 

Let $I:=E_2\sqrt{-1}\in\mathcal{H}_2$ and $U(2)={\rm Stab}_{Sp_4(\R)}(I)$. 
For any open compact subgroup $U$ of $GSp_4(\widehat{\Z})$  
we let $\mathcal{A}_{\underline{k}}(U)^\circ$ denote the subspace of functions 
$\phi:GSp_4(\Q)\backslash GSp_4(\A)\lra V_{\underline{k}}(\C)$ such that 

\begin{enumerate}
\item $\phi(gu u_\infty)=\lambda_{\underline{k}}(J(u_\infty,I)^{-1})\phi(g)$ for all $g\in G(\A)$, $u\in U$, and $u_\infty\in 
U(2)A_{G,\infty}$, 
\item for $h\in G(\A_f)$, the function 
$$\phi_h:\mathcal{H}_2\lra V_{\underline{k}}(\C),\ \phi_h(Z)=\phi_h(g_\infty I):=\lambda_{\underline{k}}(J(g_\infty,I))\phi(hg_\infty)$$
is a holomorphic function where $Z=g_\infty I,\ g_\infty\in Sp_4(\R)$ (note that this definition is independent of the choice of $g_\infty$),
\item\label{cusp} for $g\in G(\A)$, $\ds\int_{N_R(\Q)\backslash N_R(\A)}\phi(ng)dn=0$ for any parabolic $\Q$-subgroup $R$ and $dn$ 
is the Haar measure on $N_R(\Q)\backslash N_R(\A)$. 
\end{enumerate}

We define similarly $\mathcal{A}_{\underline{k}}(U)$ by omitting the last condition (\ref{cusp}). 
Let $\G(N)_a:=Sp_4(\Q)\cap d_a^{-1}K(N)d_a$. Note that $\G(N)_a=\G(N)$ for each $a$. 
Then we have the isomorphism 
\begin{equation}\label{isom}
\mathcal{A}_{\underline{k}}(K(N))\stackrel{\sim}{\lra} \bigoplus_{1\le a < N\atop (a,N)=1}
M_{\underline{k}}(\G(N)_a), \quad \phi\mapsto (\phi_{d_a})_a.
\end{equation} 
We also have the isomorphism 
\begin{equation}\label{isom1}
\mathcal{A}_{\underline{k}}(K(N))^\circ\simeq \bigoplus_{1\le a < N\atop (a,N)=1}S_{\underline{k}}(\G(N)_a)
\simeq \bigoplus_{1\le a < N\atop (a,N)=1}S_{\underline{k}}(\G(N)),
\end{equation}
as well (cf. \cite{borel&jacquet} for checking the cuspidality). 
We should note that it follows from the condition (1) to be an automorphic form that 
\begin{equation}\label{central-action}
\phi(g z_\infty )=z^{-(k_1+k_2)}_\infty\phi(g),\ g\in G(\A),\ z_\infty\in A_{G,\infty}. 
\end{equation}

Now we restrict the isomorphism (\ref{isom}) to specific subspaces, using the character decomposition (\ref{character}). 
Given two Dirichlet characters $\chi_i : (\Z/N\Z)^\times \lra \C^\times$, $i = 1, 2$, associate 
the characters $\chi'_i:\A_f^{\times}\lra \C^\times$ via the natural map $\A_f^\times\lra \widehat
\A_f^\times/\Q_{>0}={\widehat{\Bbb Z}}^\times\lra (\Z/N\Z)^\times$. 

Define
$\widetilde{\chi} : T(\A_f) \lra \C^\times$ by 
$$\tilde\chi'(\diag(*,*, c, d)) =\chi'_1(d^{-1}c)\chi'_2(d).$$
Choose $F = (F_a)$ from RHS of (\ref{isom}) which satisfies 
$F|[S_{z,z}]_{\underline{k}} =(F_a|[S_{z,z}]_{\underline{k}}) = (\chi_2(z)F_a) = \chi_2(z)F$ and
$F|[S_{z,1}]_{\underline{k}} = \chi_1(z)F$. If we write $g\in G(\A)$ as $g = rz_\infty d_a g_\infty k \in G(\A)$ and take $z_f \in T(\A_f )$, then define the automorphic function attached to $F$ by
\begin{equation}\label{const-auto}
\phi_F(g z_f) = \lambda_{\underline{k}}(J(g, I))^{-1} F_a(g_\infty I)\tilde\chi(z_f).
\end{equation}
Then this gives rise to the isomorphism of the subspaces 
\begin{equation*}
\mathcal{A}_{\underline{k}}(K(N), \tilde\chi)\stackrel{\sim}{\lra} \bigoplus_{1\le a < N\atop (a,N)=1} 
M_{\underline{k}}(\G(N)_a, \chi_1,\chi_2).
\end{equation*}

The central character of any element in $\mathcal{A}_{\underline{k}}(K(N), \tilde\chi)$ is given by 
$\chi_2$ and hence $\mathcal{A}_{\underline{k}}(K(N), 1)\simeq 
\ds\bigoplus_{\chi_1\in \widehat{(\Z/N\Z)^\times}}\bigoplus_{1\le a < N\atop (a,N)=1} 
M_{\underline{k}}(\G(N)_a, \chi_1,1)$. 

\begin{remark}
Note that $M_{\underline{k}}(\Gamma(N))$ is embedded diagonally into 
$\ds\bigoplus_{1\le a < N\atop (a,N)=1} M_{\underline{k}}(\G(N)_a)$. 
So given a cusp form $F\in M_{\underline{k}}(\Gamma(N))$, we obtain $\phi_F\in \mathcal{A}_{\underline{k}}(K(N))$ which under the isomorphism (\ref{isom}), corresponds to $(F,...,F)$, and $\phi_F$ gives rise to a cuspidal representation $\pi_F$.
Conversely, given a cuspidal representation $\pi$ of $GSp_4(\A)$, there exists $N>0$ and $\phi\in \mathcal{A}_{\underline{k}}(K(N))$
which spans $\pi$. Under the isomorphism (\ref{isom}), $\phi$ corresponds to $(F_a)_{1\le a < N\atop (a,N)=1}$. 
For any $a$, let $\pi_{F_a}$ be the cuspidal representation associated to $F_a$. Then $\pi$ and $\pi_{F_a}$ have the same Hecke eigenvalues for $p\nmid N$, and hence they are in the same $L$-packet.
\end{remark}

We now study the Hecke operators on $\mathcal{A}_{\underline{k}}(K(N))$ and its relation to classical Hecke operators. 
Let $\phi$ be an element of  $\mathcal{A}_{\underline{k}}(K(N))$ and 
$F=(F_a)_a$ be the corresponding element of RHS via the above isomorphism (\ref{isom}). 
For any prime $p\nmid N$ and  $\alpha \in G(\Q)\cap T(\Q_p)$, define the Hecke action with respect to $\alpha$ 
\begin{equation}\label{HLL}
\widetilde{T}_\alpha \phi(g):= \ds\int_{G(\A_f)}([K(N)_p\alpha K(N)_p]
\otimes 1_{K(N)^p})(g_f)\phi(g g_f) dg_f,
\end{equation}
where $dg_f$ is the Haar measure on $G(\A_f)$ so that vol$(K)=1$. Here $K(N)_p$ is the $p$-component of $K(N)$ and  
 $K(N)^p$ is the group consisting of all elements of $K(N)$ with trivial $p$-component. 
Here $[K(N)_p\alpha K(N)_p]$ stands for the characteristic function of $K(N)_p\alpha K(N)_p$.   
Then by using (\ref{strong}) and (\ref{HLL}), we can easily see that 
 \begin{equation}\label{local-global}
\nu(\alpha)^{-\frac{k_1+k_2-3}{2}}T_\alpha F(Z)=\nu(\alpha)^{-\frac{3}{2}}\widetilde{T}_{\alpha^{-1}} \phi(g),
\end{equation}
where $g=rz_\infty g_a g_\infty k$ as above and $Z=g_\infty I$ (cf. Section 8 of \cite{m&y}).  
From this relation, up to the factor of $\nu(\alpha)^{\frac{k_1+k_2}{2}-3}$, the isomorphism (\ref{isom}) 
preserves Hecke eigenforms in both sides. 

Conversely, let $F\in S_{\underline{k}}(\G(N))$ 
be a Siegel cusp form which is a Hecke eigenform. 
Then it is easy to see that $\phi=(F)_a$ is an eigenform in $\ds\bigoplus_{1\le a < N\atop (a,N)=1}S_{\underline{k}}(\G(N)_a)$. Hence we have the 
Hecke eigenform $\phi$ in $\mathcal{A}_{\underline{k}}(K(N))$ corresponding to $F$. 
We denote by $\pi_F$ the automorphic cuspidal representation associated to $F$ via $\phi$.  

To end this subsection we give an estimation of the dimension of each space which immediately follows from 
\cite{Tsushima}, \cite{Wakatsuki}:
\begin{prop} \label{dimension}
For any $\underline{k}=(k_1,k_2), k_1\ge k_2\ge 3$ and $N$, as $N+k_1+k_2\to\infty$, 
\begin{enumerate}
\item ${\rm dim} S_{\underline{k}}(\G(N))\sim C\cdot N^{10}(k_1-1)(k_2-2)(k_1-k_2+1)(k_1+k_2-3)$,  
\item ${\rm dim} \mathcal{A}_{\underline{k}}(K(N))^\circ\sim C\cdot \varphi(N)N^{10}(k_1-1)(k_2-2)(k_1-k_2+1)(k_1+k_2-3)$,
\item ${\rm dim} S_{\underline{k}}(\G(N),\chi)\sim C\cdot \frac{N^{10}}{\varphi(N)}(k_1-1)(k_2-2)(k_1-k_2+1)(k_1+k_2-3)$,  
\item ${\rm dim} \mathcal{A}_{\underline{k}}(K(N),\chi)^\circ\sim C\cdot N^{10}(k_1-1)(k_2-2)(k_1-k_2+1)(k_1+k_2-3)$,  
\end{enumerate}
where $C$ is a positive constant which is independent of $\underline{k}$ and $N$. 
\end{prop}
\begin{proof}First we assume $k_2\ge 5$ and treat the cases (1) and (2). 
Then one can apply Theorem 3, 4 of \cite{Tsushima} in case $N=1,2$, 
and Theorem 7.3 of \cite{Wakatsuki} in case $N\ge 3$  for $(j,k)=(k_1-k_2,k_2)$. 
If $k_2=3$ or 4, the argument in Section 5 of \cite{Ibuk} shows the above dimension formulas is still validity for 
$k_2\ge 3$ up to the difference comes from $\dim M_{k_1-1,1}(\G(N))$ which occurs in case $k_2=3$. By Proposition \ref{cap-k-order} and \ref{K-Eisen}, 
we will see that $\dim M_{k_1-1,1}(\G(N))=o({\rm dim}S_{\underline{k}}(\G(N)))$ as $k_1+k_2+N\to \infty$. Hence we have the claim. 
The second claim follows from (\ref{isom1}).  

For $(3),(4)$, consider a normal subgroup of $Sp_4(\Z)$ defined by  
$$\Gamma'(N)=\{\gamma\in Sp_4(\Z)\ |\ (\gamma\ {\rm mod}\ N)\in (\Z/N\Z)^\times E_4\}.$$
By Theorem 3.2 of \cite{Wakatsuki} the main term contributing to the dimension is 
$$[Sp_4(\Z):\G'(N)](k_1-1)(k_2-2)(k_1-k_2+1)(k_1+k_2-3)$$
up to an absolute constant. 
Hence we have the assertion since 
$$[Sp_4(\Z):\G'(N)]=\varphi(N)^{-1}[Sp_4(\Z):\G(N)].$$
\end{proof}

\subsection{The infinity component of $\pi_F$}\label{inf}

Let $F$ be a Hecke eigenform in $S_{\underline{k}}(\G(N))$ with the associated 
cuspidal representation $\pi_F=\pi_{F,\infty}\otimes \otimes_p' \pi_{F,p}$ of $GSp_4(\A)$. 
We assume that $k_1\ge k_2\ge 3$. Then $\pi_{F,\infty}$ is a unitary tempered representation in the discrete spectrum and 
its minimal $K$-type is $(k_1,k_2)=(l_1+1,l_2+2)$, where $(l_1,l_2)$ is the 
Harish-Chandra parameter in the holomorphic discrete series. 
Under the condition $l_1>l_2>0$, the discrete series with 
the Harish-Chandra parameter $(l_1,l_2)$ is unique up to isomorphism as $(\frak g,K)$-cohomology where 
$\frak g$ stands for the complexification of Lie($Sp_4(\R)$) and $K=U(2)$. 
Let $D^{{\rm hol}}_{l_1,l_2}$ be the  holomorphic discrete series of $GSp_4(\R)$ with Harish-Chandra parameter $(l_1,l_2)$ as above 
and its central character is given by $z\mapsto z^{-k_1-k_2}=z^{-l_1-l_2-3}$ on $A_{G,\infty}\simeq \R_{>0}$. 
If we want to insist on the minimal $K$-type $(k_1,k_2)$ instead of the Harish-Chandra parameter $(l_1,l_2)$, 
the holomorphic discrete series in question 
would be denoted by $D^{{\rm hol}}_{k_1,k_2}$.  
Hence we have 
$$\pi_{F,\infty}\simeq D^{{\rm hol}}_{l_1,l_2}.
$$
Note that both of central characters coincide on $A_{G,\infty}$ because of (\ref{central-action}).

\section{Algebraic representations of $GSp_4(\R)$}\label{alg-rep}
In this section we quickly recall algebraic representations of $GSp_4(\R)$. This is necessary to fix 
the central characters of the representations associated to classical Siegel modular forms. 

Recall $G=GSp_4$ and put $G_0=Sp_4$. 
Let $T_0$ be the split maximal diagonal torus of $G_0$ and $K_0$ be the maximal compact subgroup of $G_0(\R)$. 
Let $\xi=(\xi,V)$ be an irreducible algebraic representation of $GSp_4(\R)$. 
We have a decomposition ${\rm Lie}(G(\R))=\frak z\oplus {\rm Lie}(G_0(\R))$ where $\frak z\simeq \R$ is the center of 
${\rm Lie}(G(\R))$. 
The infinitesimal action of ${\rm Lie}(G(\R))$ on $V$ uniquely determines $\xi$ since there exists a sufficiently small, open neighborhood of 
$GSp_4(\R)$  at the origin in Euclidean topology whose Zariski closure is $GSp_4(\R)$. It follows from this and 
the classification of all algebraic representations for $Sp_4$ (cf. Section 16.2 of \cite{FH}) that 
\begin{equation}\label{xi}
\xi\simeq \nu^c\otimes \rho_{a,b},\ c\in \Z
\end{equation}
where $\nu:GSp_4\lra GL_1$ is the similitude and $\rho_{a,b}$ is a unique irreducible algebraic representation of $G(\R)$ with highest weight $(a,b)\in\Z^2,a\ge b\ge 0$. 
It is clear that the compact torus $U(1)\times U(1)$ in $U(2)$ is Zariski dense in 
$T_0(\C)$ via $U(2)\simeq K$. Hence the algebraic character of $T_0$ is completely  
determined by the action of the diagonal compact torus $U(1)\times U(1)$. It follows from this that 
$\rho_{a,b}|_{K_0 }$ contains  
$V_{(a,b)}$ (see Section \ref{class} for the notation). 
The central character of $\xi$ is given by $\diag(z,z,z,z)\mapsto z^{2c+(a+b)}$ by (\ref{xi}). 
By Weyl's dimension formula (see (24.19) in p.406 of \cite{FH}), 
\begin{equation}\label{dimension-alg}
\dim \xi=\dim \rho_{a,b}=\frac{(a-b+1)(a+b+3)(a+2)(b+1)}{6}.
\end{equation}
Put $\xi_{c,a,b}=\nu^c\otimes \rho_{a,b}$. We usually consider 
\begin{equation}\label{algebraic-char}
\xi_{\underline{k}}=\xi_{3,k_1-3,k_2-3},\ \underline{k}=(k_1,k_2),\  k_1\ge k_2\ge 3
\end{equation}
and then the central character of its dual $\xi^\vee_{\underline{k}}$  
coincides on $A_{G,\infty}$ with one of $D^{{\rm hol}}_{k_1,k_2}$. 
Note that 
$$\dim \xi_{\underline{k}}=\ds\frac{(k_1-k_2+1)(k_1+k_2-3)(k_1-1)(k_2-2)}{6}=\ds\frac{(\ell_1-\ell_2)(\ell_1+\ell_2)\ell_1\ell_2}{6}=
\frac{1}{6}d(D^{{\rm hol}}_{k_1,k_2}),
$$ 
where $d(D^{{\rm hol}}_{k_1,k_2})=(\ell_1-\ell_2)(\ell_1+\ell_2)\ell_1\ell_2$ is the formal degree of $D^{{\rm hol}}_{k_1,k_2}$. 
  
\section{Spectral decomposition and Automorphic counting measures}\label{resi}

In this section as Shin did in \cite{Shin}, we introduce measures related to several automorphic forms to look carefully inside 
the spectral side. 
This is necessary to extract only holomorphic forms from Arthur's trace formula. It will 
be clear later on. 
However in order to do that, we have to use a single pseudo-coefficient in the trace formula and it causes defects 
which never appear in the setting of \cite{Shin}. Namely, non-semisimple orbit contributions arise from the geometric side.

Let us first fix a measure on $G(\A)$. For any finite prime $p$, let $\mu_p$ be the Haar measure on $G(\Q_p)$ 
so that $\mu_p(G(\Z_p))=1$. Let $\mu_\infty$ be the Euler-Poincar\'e measure (see Section 2 of \cite{Shin}). 
Then the product measure $\mu=\ds\prod_{p\le \infty}\mu_p$ on $G(\A)$ is compatible with 
the point counting measure on  $G(\Q)$ and the Lebesgue measure on $A_{G}$ and therefore it defines the quotient measure 
$\overline{\mu}$ on $G(\Q)A_{G,\infty}\bs G(\A)=G(\Q)\bs G(\A)^1$ where 
$G(\A)^1=\{g\in G(\A)\ |\ |\nu(g)|_{\A}=1\}$. 
It follows that 
\begin{equation}\label{identification}
G(\A)\simeq G(\A)^1\times A_{G,\infty},\ g=(g_f,g_\infty)\mapsto ((g_f,g_\infty |\nu(g)|^{-1}_{\A}),|\nu(g_\infty)|))
\end{equation} and clearly $G(\A)^1\supset G(\Q)$. 

\subsection{Spectral decomposition}
For a quasi-character $\chi$ (which is not necessarily unitary) on $A_{G,\infty}$ we define 
$L^2:=L^2(G(\Q)\bs G(\A),\chi)$ as the space of $\C$-valued functions on $G(\A)$ which are square integrable modulo $A_{G,\infty}$ 
with respect to the measure $\chi$, 
left $G(\Q)$-invariant, and transform under $A_{G,\infty}$ by $\chi$. 
Here ``square integrable modulo $A_{G,\infty}$" means that the integral over $G(\Q)\bs G(\A)^1$ is square integrable which makes sense 
because of (\ref{identification}).  
The regular action of $G(\A)$ decomposes the Hilbert space $L^2(G(\Q)\bs G(\A),\chi)$ into 
the discrete spectrum and the continuous one. Then 

$$L_{{\rm disc}}^2(G(\Bbb Q\backslash G(\Bbb A),\chi)=L_{{\rm cusp}}^2(G(\Bbb Q)\backslash G(\Bbb A),\chi)\oplus 
L_{{\rm res}}^2(G(\Bbb Q)\backslash G(\Bbb A),\chi).
$$ 
For $\ast\in\{{\rm disc},{\rm cusp},{\rm res}\}$, let

$$L^2_{*}:=L^2_{*}(G(\Bbb Q\backslash G(\Bbb A),\chi)=\bigoplus_\pi\, m_{*}(\pi)\, \pi,
$$
where $m_{*}(\pi)$ is the multiplicity of $\pi$ in $L^2_{*}(G(\Bbb Q\backslash G(\Bbb A),\chi)$.

We review the residual spectrum of $GSp_4$ in \cite{Kim}:
$$L_{{\rm res}}^2(G(\Bbb Q\backslash G(\Bbb A),\chi)=L^2(B)\oplus L^2(P_1)\oplus L^2(P_2).
$$

Then
$$L^2(P_1)=\bigoplus_{(\pi,\eta)} J(\frac 12, \pi\otimes\eta),
$$
where $\pi$ runs over cuspidal representations of $GL_2(\A_\Q)$ with the trivial central character such that $L(\frac 12,\pi)\ne 0$, and $\eta$ runs over gr\"ossencharacters of $\A^\times_\Q$ such that $\eta^2_\infty=\chi$, and $J(\frac 12,\pi\otimes\eta)$ is the unique quotient of 
${\rm Ind}_{P_1}^G \pi |\det|^{\frac 12} \otimes\eta$.

Similarly,
$$
L^2(P_2)=\bigoplus_{(\eta,\pi)} J(1, \eta\otimes\pi),
$$
where $\eta$ runs over nontrivial quadratic characters of $\A^\times_\Q$, and 
$\pi$ runs over monomial cuspidal representations of $GL_2(\A_\Q)$ such that 
$\pi\simeq \pi\otimes\eta$, and $\eta\omega_\pi=\chi$, and $J(1,\eta\otimes\pi)$ is the unique quotient of 
${\rm Ind}_{P_2}^G \eta|\cdot|\otimes\pi$.

Finally,
$$
L^2(B)=J(\rho_B, \chi(1,1,\mu))\oplus \bigoplus_\nu J(e_1,\chi(\nu,\nu,\mu)),
$$
where $\nu$ runs over nontrivial quadratic characters and $\mu^2=\chi$. Here $J(\rho_B,\chi(1,1,\mu))$ is the unique quotient of 
${\rm Ind}_B^G \chi(1,1,\mu)\otimes \exp(\langle \rho_B+\rho_B,H_B(\ )\rangle)$, and $J(e_1,\chi(\nu,\nu,\mu))$ is the unique quotient of
${\rm Ind}_B^G \chi(\nu,\nu,\mu)\otimes \exp(\langle e_1+\rho_B,H_B(\ )\rangle)$, which is the Langlands quotient of
${\rm Ind}_{P_2}^G |\cdot|\otimes\sigma$, where $\sigma={\rm Ind}_B^{M_2} \chi(\nu,\nu,\mu)$. 

Any local component of a representation in the residue spectrum is known to be non-tempered (cf. \cite{Wallach}). 
Recall that $D^{{\rm hol}}_{k_1,k_2}$ is tempered for $k_1\ge k_2\ge 3$.  
Therefore 
$${\rm Hom}_{G(\R)}(D^{{\rm hol}}_{k_1,k_2},L^2_{{\rm res}})=0.
$$

\subsection{Special cohomological representation of $GSp_4(\Bbb R)$}\label{special}

As we will see in Section \ref{pm}, 
the trace of a pseudo-coefficient of a holomorphic discrete series $D^{{\rm hol}}_{l_1,l_2}$ may be non-vanishing for $D^{{\rm hol}}_{l_1,l_2},\omega_{l_1}$, and $\bold{1}$, 
where $\bold{1}$ is the trivial representation of $G(\R)$, and $\omega_l$ is a certain unitary representation defined as follows:
For $l\geq 2$, it is the induced representation ${\rm Ind}_{Sp_4(\R)}^{G(\R)}\, \omega_l'$, where $\omega_l'$ is the Langlands quotient of 
${\rm Ind}^{Sp_4(\R)}_{P_2(\R)\cap Sp_4(\R)}| \cdot|\sgn \rtimes D^+_{l}$, where
 $D^+_{l}$ is the holomorphic discrete series of $SL_2(\R)$ with 
minimal $K$-type $l$. 

By the classification of the residual spectrum, $\bold{1}$ occurs as an infinity component of residual spectrum from the Borel subgroup:
\begin{equation}\label{residual}
{\rm Hom}_{G(\R)}(\mathbf{1},L^2)^{K(N)}={\rm Hom}_{G(\R)}(\mathbf{1},L^2(B))^{K(N)}=
\bigoplus_{\chi\in \widehat{(\Z/N\Z)^\times}\atop \chi^2=1}\C.
\end{equation}

One can see easily that there are no cuspidal representations with the infinity component $\bold{1}$: If $F$ is such a form, then clearly, $\Omega F=0$ for any differential operator. This is clearly not possible. (Or one can see from strong approximation theorem that $G(\Q)G(\Bbb R)$ is dense in $G(\Bbb A)$. Hence if $\bold{1}$ is the infinity component of an automorphic representation, the automorphic representation has to be
$\bold{1}$.)

The special unitary representation $\omega_l$ occurs as the infinity component of residual spectrum of Klingen parabolic subgroup, and CAP forms of weight $(l+1,1)$ from the Klingen parabolic subgroup (Proposition \ref{cap-k}). It is non-tempered. Hence
$${\rm Hom}_{G(\R)}(\omega_{l_1},L^2_{{\rm res}})^{K(N)}={\rm Hom}_{G(\R)}(\omega_{l_1},L^2(P_2))^{K(N)}\hookrightarrow 
\bigoplus_{a\in (\Z/N\Z)^\times}E^1_{(k_1,1)}(\G(N)_a),
$$
where $E^1_{(k_1,k_2)}(\G(N))$ is the space of Klingen Eisenstein series of weight $(k_1,k_2)$.

\begin{prop}\label{K-Eisen} For any $k_1\ge k_2\ge 1$, 
$$\dim E^1_{(k_1,k_2)}(\G(N))=O(k_1N^{3})\ {\rm as}\ k_1+N\to\infty.$$
\end{prop} 
\begin{proof}Put $P'_2=P_2\cap Sp_4$ (Recall that $P_2$ is Klingen parabolic in $GSp_4$). 
Put $\pi:P'_2\lra SL_2$ be the projection to $SL_2$-factor of the Levi subgroup of $P'_2$. 
Let $\G(1)=\coprod_{\lambda}\G(N)M_\lambda P'_2(\Z)$ be the double coset decomposition and $\Phi^1_\lambda$ be 
the associated Siegel $\Phi$-operator for each $M_\lambda$. 

First we assume that $k_2>4$. By a similar argument in the proof of Proposition 2.1 in \cite{Arakawa}, one has an injective linear map 
which is Hecke equivariant outside $N$:    
$$\Phi:E^1_{(k_1,k_2)}(\G(N))\lra \prod_{\lambda}S_{k_1}(\Gamma_\lambda),\ F\mapsto (\Phi^1_\lambda(F))_\lambda$$
where $\G_\lambda=\pi(M_\lambda P'_2 M^{-1}_\lambda\cap \G(N))\simeq 
\pi(M^{-1}_\lambda \G(N) M_\lambda\cap P'_2)=\pi(\G(N)\cap P'_2)=\G^1(N)$. 
By multiplicity one theorem for elliptic newforms, $\Phi^1_\lambda(F)$ is independent of $\lambda$ which means 
the composition of $\Phi$ with any projection to the $\lambda$-th component is still injective. 
Summing up, we have 
$$\dim E^1_{(k_1,k_2)}(\G(N))\le \dim S_{k_1}(\G^1(N)).$$
Since $\dim S_{k_1}(\G^1(N))=O(k_1N^3)$ as $k_1+N\to\infty$, we have $\dim E^1_{(k_1,k_2)}(\G(N))=O(k_1N^{3})$. 

Next we assume that $1\le k_2\le 4$. Let $E_4$ be the Siegel Eisenstein series of weight 4 with respect to $\G(1)=Sp_4(\Z)$. 
By multiplying $E_4$ we have an injective linear map 
$E^1_{(k_1,k_2)}(\G(N))\hookrightarrow E^1_{(k_1+4,k_2+4)}(\G(N))$. Then we may apply the previous argument. 
\end{proof}

\subsection{Classification of endoscopic forms}\label{class-endo} 
In this section we study endoscopic forms in $S_{\underline{k}}(\G(N))$ and give an estimation of the dimension of the space 
$S^{{\rm en}}_{\underline{k}}(\G(N))$ generated by such forms. 
For simplicity we work on adelic forms instead of classical forms. 
We write $\mathcal{A}^{{\rm en}}_{\underline{k}}(\G(N))^\circ$  for the adelic version of $S^{{\rm en}}_{\underline{k}}(\G(N))$ via the isomorphism (\ref{isom1}). 
We will freely use the results in \cite{Roberts}. 
The endoscopic lift in our situation is a functorial lift from the endoscopic group $H:=GSO(2,2)$ to $GSp_4$. Note that
$GSO(2,2)\simeq (GL_2\times GL_2)/\{(z,z^{-1}): z\in GL_1\}$, and 
$GSO(4)\simeq (D^{\times}\times D^{\times})/\{(z,z^{-1}): z\in GL_1\}$, where $D$ is a quaternion division algebra over $\Q$.
 Hence a cuspidal representation of $GSO(2,2)$ (resp. $GSO(4)$) can be written as 
$(\pi_1,\pi_2)$, where $\pi_1,\pi_2$ are cuspidal representations of $GL_2$ (resp. $D^\times$) with the same central characters.
When we lift a pair of elliptic new forms of weights $\ge 2$,  
it is known that 
\begin{enumerate}
\item if the lift factors through $GSO(4)$ via Jacquet-Langlands correspondence, then it has to be a holomorphic discrete series at 
infinity, and  
\item if the lift does not factor through $GSO(4)$, then the lift can not be a holomorphic discrete series at infinity (it is 
a large discrete series).  
\end{enumerate}
Since we are interested in holomorphic Siegel forms, only the first case can happen.

For any reductive group $G$ over $\Q$, 
we denote by $\Pi(G(\A))$ the set consisting of the isomorphism classes of cuspidal automorphic representations of $G(\A)$. 

Let $\Pi(\tau)$ be the global packet for $GSp_4$ constructed 
from a cuspidal representation $\tau$ of $H(\A)$ via theta lift due to Roberts \cite{Roberts}. Put $r_1=k_1+k_2-2$ and $r_2=k_1-k_2+2$. 
By adjusting the central character, we may assume that any element of $\Pi(\tau)$ is realized in the space 
$L^2_{{\rm cusp}}(G(\Q)\bs G(\A),\chi_{\xi^\vee})$ for 
$\xi=\xi_{\underline{k}}$. 
Then we see that 
\begin{equation}\label{dim-en}
{\rm dim} \mathcal{A}^{{\rm en}}_{\underline{k}}(K(N))^\circ={\rm vol}(K(N))^{-1}\sum_{\tau\in \Pi(H(\A))\atop \tau_\infty\simeq
D_{r_1}\otimes D_{r_2}}\sum_{\Pi\in \Pi(\tau)\atop \Pi_\infty \simeq D^{{\rm hol}}_{k_1,k_2}}m(\Pi){\rm tr}(\Pi_f(1_{K(N)})),\ 
m(\Pi)\in\{0,1\}
\end{equation}
where $1_{K(N)}={\rm char}_{K(N)}$. Note that $s\cdot \tau\not\simeq \tau$ ($s$ interchanges the two components) because of the infinite type and $r_1\not=r_2$.  
We shall describe the RHS more precisely. 
For $\Pi\in \Pi(\tau)$, let $T_{\Pi}$ be the set of every places $v$ of $\Q$ so that 
$\Pi_v\simeq \theta_{GSO(4)}(\tau^{\rm JL}_v)$ where $GSO(4)\simeq (D^\times_v\times D^\times_v)/\Q^\times_v$ 
for a unique quaternion division field over $\Q_v$ and $\theta$ is 
the local theta lift from $GSO(4)$ to $GSp_4$. Note that $\tau_v$ is necessarily square integrable. 
We write $\tau=(\pi_1,\pi_2)$ for $\tau\in \Pi(H)$ where each $\pi_i$ is a cuspidal automorphic representation of $GL_2(\A)$ 
such that $\omega_{\pi_1}=\omega_{\pi_2}$. For such a $\tau$, let $T_\tau$ be the set of all places of $\Q$ so that 
both of $\pi_{1,v}$ and $\pi_{2,v}$ are  square integrable. 

By Theorem 8.5 of \cite{Roberts} and Theorem 8.1,8.2 of \cite{GT1},  
if $m(\Pi)=1$, then for each $\Pi\in \Pi(\tau),\ \tau=(\pi_1,\pi_2)$, there exists a definite quaternion algebra $D$ over $\Q$ such that 
$\tau$ is in the image of Jacquet-Langlands correspondence from $GSO(D)(\A)$ and 
the set of ramified places of $D$ is given exactly by $T_\Pi=T_\tau$ and   

\begin{enumerate}
\item for each $v\in T_\tau$, 
\begin{enumerate}
\item if $\sigma^{JL}=\pi_{1,v}=\pi_{2,v}$, then $\Theta(\sigma\boxtimes \sigma)=\theta(\sigma\boxtimes \sigma)$ is the unique non-generic direct factor of 
$I_{Q(Z)}(1,\sigma^{JL})$ where $Q(Z)$ stands for the Klingen parabolic subgroup, 

\item if $\pi_{1,v}\not=\pi_{2,v}$, then $\Theta(\pi^{{\rm JL}}_{1,v}\boxtimes \pi^{{\rm JL}}_{2,v})=\theta(\pi^{{\rm JL}}_{1,v}\boxtimes \pi^{{\rm JL}}_{2,v})$ is a non-generic supercuspidal 
representation,    

\end{enumerate}

\item for each $v\not\in T_\tau$,  one of $\pi_{1,v}$ and $\pi_{2,v}$ has to be a principal series representation and 
\begin{enumerate}
\item if $\pi_{1,v}$ is square integrable and $\pi_{2,v}=\pi(\chi,\chi')$ for unitary characters $\chi,\chi'$ 
so that $\omega_{\pi_{1,v}}=\chi\chi'$, then 
$\theta(\pi_{1,v}\boxtimes \pi_{2,v})=J_{P(Y)}(\pi_{1,v}\otimes \chi^{-1},\chi)$ is a 
non generic Langlands quotient of $I_{P(Y)}(\pi_{1,v}\otimes \chi^{-1},\chi)$ where $P(Y)$ stands for the Siegel parabolic 
subgroup,  

\item if $\pi_{1,v}=\pi(\chi_1,\chi_1')$  and $\pi_{2,v}= \pi(\chi_2,\chi_2')$ for  unitary characters $\chi_i,\chi_i'$  
so that $\chi_1\chi'_1=\chi_2\chi'_2$, then 
$\theta(\pi_{1,v}\boxtimes \pi_{2,v})={\rm Ind}^G_B(\chi_2'/\chi_1,\chi_2/\chi_1;\chi_1)$. 
\end{enumerate}
\end{enumerate}
In each case the local L-packets are given as follows (see Section 8 of \cite{CG1} for (1) and Section 6.6 of \cite{CG1} for (2)): 
\begin{enumerate}
\item 
\begin{enumerate}
\item  $\{\pi^+,\pi^-\}$ where $\pi^+$ (resp. $\pi^-$) is the unique generic (resp. non-generic) direct factor of 
$I_{Q(Z)}(1,\sigma^{JL})$, 

\item $\{\pi^+:=\theta(\pi_{1,v}\boxtimes \pi_{2,v}),\pi^-:=\theta(\pi^{{\rm JL}}_{1,v}\boxtimes \pi^{{\rm JL}}_{2,v})\}$, 

\end{enumerate}

\item 
\begin{enumerate}
\item  $\{\pi^+,\pi^-:=J_{P(Y)}(\pi_{1,v}\otimes \chi^{-1},\chi)\}$ where $\pi^+$ is the unique generic direct factor of 
$I_{P(Y)}(\pi_{1,v}\otimes \chi^{-1},\chi)$,  

\item $\{{\rm Ind}^G_B(\chi_2'/\chi_1,\chi_2/\chi_1;\chi_1)\}$ (the packet is singleton). 
\end{enumerate}
\end{enumerate}
For a principal congruence subgroup $K(N)$ we put $K^H(N)=H(\A_f)\cap K(N)$. For a finite place $v=p_v$ of $\Q$, put 
$N_v=p^{{\rm ord}_v(N)}_v$. 
From now suppose that $(N,11!)=1$. Then by Th\'eor\`em 3.2.3 of \cite{Fer}, 
$(1_{K(N)_v},N^{-2}_v1_{K^H(N)_v})$ is a transfer pair for each finite place $v$.   
We now apply the (endoscopic) local character identities (Proposition 6.9 of \cite{CG1} for the L-packets in (2) and Proposition 8.2 for 
the L-packets in (1)) 
with the transfer pair $(1_{K(N)_v},N^{-2}_v1_{K^H(N)_v})$ for each finite place$v$, the RHS of (\ref{dim-en}) is bounded by 
\begin{eqnarray}
{\rm RHS} & \le & {\rm vol}(K(N))^{-1}\sum_{\tau\in \Pi(H(\A)) \atop \tau_\infty\simeq
D_{r_1}\otimes D_{r_2}}\sum_{\Pi\in \Pi(\tau)\atop \Pi_\infty \in 
\{D^{{\rm hol}}_{k_1,k_2},D^{{\rm large}}_{k_1,k_2} \} }m(\Pi){\rm tr}(\Pi_f(1_{K(N)})) \nonumber \\
&\le&   
{\rm vol}(K(N))^{-1}\sum_{\tau\in \Pi(H(\A)) \atop \tau_\infty\simeq
D_{r_1}\otimes D_{r_2}}2^{|T_\tau|}N^{-2}{\rm tr}(\tau_f(1_{K^H(N)}))\nonumber \\
&\ll &   
{\rm vol}(K(N))^{-1}N^{-2+\ve}{\rm vol}(K^H(N)){\rm dim}S_{r_1}(\G^1(N))\times {\rm dim}S_{r_2}(\G^1(N))  \nonumber \\
&\ll &  (r_1-1)(r_2-1)N^{9+\ve}  \nonumber 
\end{eqnarray} 
Since ${\rm dim} S^{{\rm en}}_{\underline{k}}(\G(N))=\varphi(N)^{-1}{\rm dim} \mathcal{A}^{{\rm en}}_{\underline{k}}(\G(N))^\circ$, we have 
\begin{thm}\label{dim-en1} 
${\rm dim} S^{{\rm en}}_{\underline{k}}(\G(N))=O((k_1-k_2+1)(k_1+k_2-3)N^{8+\ve})$  as $N+k_1+k_2\to \infty, (N,11!)=1$. 
\end{thm}

\subsection{Classification of CAP forms}\label{sccf} 
In this section, we classify CAP forms.
We say $F$ is a CAP form associated to a parabolic subgroup $P$ 
if $\pi_F$ is a CAP representation associated to $P$ in the sense of \cite{Gan}. Such a representation is completely classified by 
\cite{PiSha},\cite{Soudry}, \cite{Schmidt}.  

For holomorphic Siegel modular CAP forms of weight $k_1\ge k_2\ge 1$, it turns out that 
\begin{enumerate}
\item in the case when $(k_1,k_2)$, $k_2>2$, it has to be a CAP form associated to Siegel parabolic subgroup and $k_1=k_2$;
\item in the case when $k_2=2$, we must have $k_1=k_2=2$ and it can be a CAP form associated to any parabolic subgroup;
\item in the case when $(k_1,k_2)=(1,1)$, it can be a CAP form associated to Borel or Klingen parabolic subgroup, but not Siegel parabolic subgroup;
\item in the case when $k_1>k_2=1$, it has to be a CAP form associated to Klingen parabolic subgroup. 
\end{enumerate}
In \cite{manni&top}, one can see several examples regarding to the second case. In the third case 
Weissauer proved that any (Hecke eigen) Siegel modular forms of weight one with respect to $\G_0(N)$ is CAP (see \cite{wei-wt1}), but it is still open whether it is also the case for $\G(N)$. 
In the case (4), it will be proved in Proposition \ref{cap-k} that 
any (Hecke eigen) Siegel modular forms of  weight $(k_1,1)$ $k_1\ge 3$ is always a CAP form associated to Klingen. 
We expect that it holds even if $k_1=2$ but in this case we cannot use the geometric argument and thus 
we might have to rely on the classification of representations for $GSp_4$. Once  
Arthur's conjectural classification in \cite{ArthurGSp4} is completed, then our expectation would be true since it is known to be non-tempered at infinity 
by \cite{PSS}.   

Consider the first case $k_1\ge k_2\ge 2$. We first observe that $k_1=k_2$ by the argument in p.225 of \cite{Schmidt} and hence put 
$k:=k_1=k_2$.  
In this case the central character of $\pi_F$ should be a square of a character and by twisting we may 
assume that it is trivial. R. Schmidt completely characterized any holomorphic CAP form associated to $P_1$ (see Theorem 3.1 of \cite{Schmidt}.) 
by constructing a lift from a cuspidal automorphic representation $\pi$ of
${\rm PGL}_2(\A)$.  
Schmidt's construction can  be a functorial lift by the local Langlands conjecture established by \cite{GT} (see Remark 3.2-(a) 
of \cite{Schmidt}). 
To characterize $\pi_F$ by using the completed L-function (product over all places), we have to carefully look at the behavior at bad places since weak equivalence does not characterize $F$ except for the case of level one.  

Let $S^{P_1}_{k}(\G(N))$ be the space generated by Hecke eigen cusp forms of parallel weight $k$  
which are CAP associated to $P_1$. We now try to estimate the dimension of this space. 
Let $\mathcal{A}^{P_1}_{k}(\G(N))^\circ$ be the adelic version of $S^{P_1}_{k}(\G(N))$ via (\ref{isom1}). 
To do this we carefully check the behavior of the levels under the functorial lift constructed by Schmidt. 
As is done for endoscopic lifts we work on adelic forms. 

Since any CAP representation associated to the Siegel parabolic subgroup 
can be regarded as a cuspidal representation of $PGSp_4(\A)$ by Theorem 2.1 of \cite{PiSha}, we first assume that 
any representation in question has the trivial central character. 
 
Let $\mathcal{A}(\pi):=\mathcal{A}(\pi,1)$ be a candidate of the global A-packet for $PGSp_4$ constructed 
from a cuspidal representation $\pi$ of $PGL_2(\A)$ due to Schmidt \cite{Schmidt} and Section 4,5 of \cite{Gan} 
(they studied the same A-packets in conjunction with Waldspurger's local packers, cf. Section 3 of \cite{Schmidt}). 
We will use $\mathcal{A}(\pi)$ as a tool to describe the dimension of $\mathcal{A}^{P_1}_{k}(K(N),1)^\circ:=
\mathcal{A}^{P_1}_{k}(\G(N))^\circ\cap \mathcal{A}_{k}(\G(N),1)^\circ$ and do not care 
whether this packet satisfies desirable properties.  
Then we see that 
\begin{equation}\label{dim-cap}
{\rm dim} \mathcal{A}^{P_1}_{k}(K(N),1)^\circ ={\rm vol}(K(N))^{-1}\sum_{\pi\in \Pi(PGL_2(\A))\atop \tau_\infty\simeq
D_{2k-2}}\sum_{\Pi\in \mathcal{A}(\pi)\atop \Pi_\infty \simeq D^{{\rm hol}}_{k_1,k_2}}m(\Pi){\rm tr}(\Pi_f(1_{K(N)})), 
\end{equation}
where $m(\Pi)\in\{0,1\}$. 
We study the RHS more precisely. For $\pi\in \Pi(PGL_2(\A))$, let $S_\pi$ be the set of all places $v$ of $\Q$ such that 
$\pi_v$ is square integrable.  

By Theorem 3.1 of  \cite{Schmidt}, for each $\Pi\in \Pi(\tau)$ with $m(\Pi)=1$, there exists a subset $S\subset S_\pi$ such that 
$\Pi=\Pi(\pi\otimes \pi_S)$ and $(-1)^{|S|}=\ve(1/2,\pi)$ and 

\begin{enumerate}
\item if $v\not\in S$, then $\Pi_v$ is the unique, non-tempered irreducible 
quotient $Q(|\cdot|^{1/2}\pi, |\cdot|^{-1/2})$ (see Proposition 5.5.1 of \cite{RS}) of ${\rm Ind}^G_{P_1}(|\cdot|^{1/2}\pi\rtimes |\cdot|^{-1/2})$
where $P_1$ is the Siegel parabolic subgroup,    
\item if $v\in S\setminus\{\infty\}$, then $\Pi_v=SK(\pi^{{\rm JL}}_v)=\theta((\pi^{{\rm JL}}_v\boxtimes {\rm St}^{{\rm JL}})^+)$ is a 
non-generic (tempered) cuspidal representation, and  
\item if $v=\infty$, then $\Pi_\infty=D^{{\rm hol}}_{k,k}$. 
\end{enumerate}
Put $U=K(N)$ (resp. $U^1=K^1(N)$) and let $U_v$ (resp. $U^1_v$) be the $v$-component of $U$ (resp. $U^1$). 
In the first case, by using Iwahori decomposition with respect to $P_1=M_1N_1$ so that $U_v=(U_v\cap N^{-}_1)
(U_v\cap M_1)(U_v\cap N_1)$ we see that $(\Pi_v)^{U_v}\simeq ({\Pi_v}_{N_1})^{U_v\cap M_1}$, 
where ${\Pi_v}_{N_1}$ is the Jacquet module of $\Pi_v$ (see the argument of the proof of Theorem 2.1 of \cite{m&y} 
though it might be well-known for experts). By using explicit semisimple decomposition of ${\Pi_v}_{N_1}$ in 
the Table A.3, p. 273 of \cite{RS} we have that ${\rm tr}(\Pi_v(1_{U_v}))\le 3{\rm tr}(\pi_v(1_{U^1_v}))$
in the first case (1).     

In the second case the packet is L-packet and it consists of 
$\{\pi^-:=\theta((\pi^{{\rm JL}}_v\boxtimes {\rm St}^{{\rm JL}})^+), \pi^+:=\theta((\pi_v\boxtimes {\rm St})^+)\}$. 
Therefore one can apply the previous argument in the case of endoscopic lifts. 
To bound (\ref{dim-cap}) we additionally count admissible representations $\Pi'$ obtained from $\Pi$ in $\Pi(\tau)$ by switching the non-generic representation  
$\pi^-$ with the generic representation $\pi^+$ at the finite places $v$ in the case (2). 
Therefore we have 

\begin{eqnarray}
{\rm RHS} & = & {\rm vol}(K(N))^{-1}\sum_{\pi\in \Pi(PGL_2(\A))\atop \tau_\infty\simeq
D_{2k-2}}\sum_{\Pi\in \mathcal{A}(\pi)\atop \Pi_\infty \simeq D^{{\rm hol}}_{k_1,k_2}}
\sum_{S\subset S_\pi \atop \ve(1/2,\pi)=(-1)^{|S|}} {\rm tr}(\Pi_f(1_{K(N)}))  \nonumber  \\
&\le  & {\rm vol}(K(N))^{-1}\sum_{\pi\in \Pi(PGL_2(\A))\atop \tau_\infty\simeq
D_{2k-2}}\sum_{\Pi\in \mathcal{A}(\pi)\atop \Pi_\infty \simeq D^{{\rm hol}}_{k_1,k_2}}
\sum_{S\subset S_\pi \atop \ve(1/2,\pi)=(-1)^{|S|}} \{{\rm tr}(\Pi_f(1_{K(N)}))+ {\rm tr}(\Pi'_f(1_{K(N)})) \}  \nonumber  \\
&\le & {\rm vol}(K(N))^{-1} \sum_{\pi\in \Pi(PGL_2(\A))\atop \tau_\infty\simeq D_{2k-2}}
\sum_{S\subset S_\pi \atop \ve(1/2,\pi)=(-1)^{|S|}} \prod_{v\not\in S}3\cdot{\rm tr}(\pi_v(1_{U^1_v}))\\
&&\times\prod_{v\in S\setminus \{\infty\}}2N^{-2}_v{\rm tr}((\pi_v\boxtimes {\rm St})(1_{K^H(N)_v}))  \nonumber  \\
&\ll & {\rm vol}(K(N))^{-1} 2^{|\{p|N\}|} N^{-2} {\rm vol}(K^H(N))\varphi(N)^{-1}{\rm dim}S_{2k-2}(K^1(N))\times [\G^1_0(N):\G^1(N)] 
\nonumber \\
&\ll & kN^{7+\ve} \nonumber
\end{eqnarray} 
Note that $[\G^1_0(N):\G^1(N)] $ comes from 
$${\rm dim}{\rm St}^{K^1(N_v)}={\rm dim}({\rm St}^{K^1_0(N_v)})^{K^1_0(N_v)/K^1(N_v)}=
{\rm dim}(\C v_{{\rm new}})^{K^1_0(N_v)/K^1(N_v)}=[K^1_0(N_v):K^1(N_v)]$$
where $v_{{\rm new}}$ is a new vector of St. 
Summing up we have 
\begin{thm}\label{dim-cap} 
$${\rm dim} S^{P_1}_{\underline{k}}(\G(N))=O(kN^{7+\ve})$$  as $N+k\to \infty, (N,11)=1$. 
\end{thm}
\begin{proof}It follows from 
$${\rm dim} S^{P_1}_{\underline{k}}(\G(N))=O({\rm dim} S^{P_1}_{\underline{k}}(\G(N),1)\varphi(N))=
O({\rm dim} \mathcal{A}^{P_1}_{k}(K(N),1)^\circ)=O(kN^{7+\ve}).$$
\end{proof}

Next we consider the case $k_1\ge 3$ and $k_2=1$. First of all we prove the following: 
\begin{prop}\label{cap-k}Let $k\ge 3$ and $N\ge 1$ be integers. 
Then any Siegel cusp form $F$ of 
weight $(k,1)$ with respect to $\G(N)$ which is a Hecke eigenform
 is a CAP form
associated to the Klingen parabolic subgroup, and its associated cuspidal representation has the infinity type $\omega_{k-1}$.
\end{prop}
\begin{proof} 
By \cite{PSS}, $\pi_{F,\infty}|_{Sp_4(\R)}$ has the component which is isomorphic to $\omega_{k-1}$. (In \cite{PSS}, $\omega_{k-1}$ is denoted as $L(k-1,1)$.) Since it is 
the Langlands quotient of ${\rm Ind}^{Sp_4(\R)}_{P_2(\R)\cap Sp_4(\R)}| \cdot|\sgn \rtimes D^+_k,$ 
$H^2(({\rm Lie}Sp_4(\R),K),\omega_{k-1}\otimes\xi^\vee_{(k,3)})\not=0$ by a direct calculation with Proposition 3.1 of \cite[p.36]{BW}.

For $p\nmid N$, let ${\rm Sat}_p(\pi_F)$ be the set of all Satake parameters of $\pi_{F,p}$ which take the values in 
${}^L GSp_4=GSp_4(\C)$. 
By Theorem 24.1 of \cite{Laumon2} (see also Theorem 7.5-(4) of \cite{Laumon1}), 
there exists an elliptic new form $f$ of weight $k$ with level dividing a power of $N$ such that 
$${\rm Sat}_p(\pi_F)=\{p^{1/2}\alpha_p,p^{1/2}\beta_p,p^{-1/2}\alpha^{-1}_p,p^{-1/2}\beta^{-1}_p  \}$$
for all $p\nmid N$, where $\{\alpha_p,\beta_p\}$ is the Satake parameters of $\pi_{f,p}$. 
This means that $\pi_F$ is weakly equivalent to 
${\rm Ind}^{Sp_4(\A)}_{P_2(\A)\cap Sp_4(\A)}| \cdot|\sgn \rtimes \pi_f$ and hence $F$ is a CAP form associated to the Klingen parabolic subgroup. 
\end{proof}

By Proposition \ref{cap-k}, if $F$ is any Siegel cusp form of 
weight $(k,1),\ k\ge 3$ with respect to $\G(N)$ which is a Hecke eigenform, then 
$\pi_F$ is a CAP representation associated to $P_2$ and then it follows from \cite{Soudry} that 
$\pi_F$ can be obtained by a theta lifting 
from a Hecke character for an imaginary quadratic field.    
In \cite{Wei1}, Weissauer computed the dimension of $S_{(3,1)}(\G(N))$ but his calculation works for any weight 
$(k,1)$ with $k\ge 3$  
(just replace $\delta(K,2,\frak a N_1)$ with $\delta(K,k-1,\frak a N_1)$ in his notation). 

For an imaginary quadratic field $K$ let $\mathcal{O}_K$ be the ring of integers, 
$w(K)$ the cardinality of the units in $\mathcal{O}_K$, $D_K$ the fundamental discriminant of $K$, 
and $\theta_K$ the different of $K/\Q$. 
For an ideal $\frak f\subset \mathcal{O}_K$, let $E_{K,\frak f}$ be the units in $\mathcal{O}_K$ congruent to $1$ modulo $\frak f$. 
Clearly $|E_{K,f}|$ divides $6$. 
For an integer $k\ge 2$ and an ideal $\frak f\subset \mathcal{O}_K$, 
define $\delta(K,k,\frak f)$ to be 1 if $(E_{K,\frak f})^k=\{1\}$ and 0 otherwise. For a fixed $K$ and a positive integer $N$, put 
$$N_1(K)=\prod_{p|N\atop {\rm split}\ {\rm in}\ K}p^{{\rm ord}_p(N)}.$$
For an ideal $\frak a\subset  \mathcal{O}_K$ put $N(\frak a)=| \mathcal{O}_K/\frak a|$.  
For $K,\ N$,  and $\frak a$, 
put $$p(\frak a,N,K):=
N(\frak a)^2\prod_{\frak p|\frak a}(1-N(\frak p)^{-2})N_1(K)^4\prod_{p|N_1(K)}(1-p^{-4}).$$
Then we have the following 
 \begin{prop}For $k\ge 3$, the following equality hold:
$${\rm dim} S_{(k,1)}(\G(N))=\sum_{K}\frac{h(K)}{w(K)}\sum_{M}\sum_{\frak a}\delta(K,k-1,\frak a N_1(K))
|E_{K,\frak a N_1(K)}|
p(\frak a, N,K)
$$ 
where $K$ runs over all imaginary quadratic field so that $\theta^2_K|N$ and $M$ all divisor of $N$ whose all prime 
divisors are inert. The last summation runs over all ideal $\frak a \subset \mathcal{O}$ dividing $\theta(K,M,N)$  
where 
$$\theta(K,M,N):=\prod_{\frak p\subset K \atop {\rm non-split}}{\frak p}^{2[{\rm ord}_{\frak p}(NM)/2]-{\rm ord}_{\frak p}(M)}.$$    
\end{prop}
By using this we have 
 \begin{prop}
\label{cap-k-order}
The notation being as above. Then for any $\ve>0$,
 $${\rm dim} S_{(k,1)}(\G(N))=O(N^{6+\ve})  {\rm \ as\ }N\to \infty.
$$ 
 \end{prop}
 \begin{proof}
The trivial bound is $w(K)\leq 6$, $|E_{K, \frak f}|\leq 6$.
By definition, $N_1(K)\leq \frac N{D_K}$ and then  
$$\prod_{p| N_1(K)} (1-p^{-4})\leq \prod_p (1-p^{-4})=\zeta(4)^{-1}$$ 
If $\frak p$ is inert, $N(\frak p)=p^2$, where $\frak p=p$ is a rational prime. If $\frak p$ is ramified, $N(\frak p)=p$, where $p=\frak p^2$. Since $1-p^{-2}\leq 1-p^{-4}$, one has 
$$\prod_{\frak p|\frak a} (1-N(\frak a)^{-2})\leq \prod_p (1-p^{-4})=\zeta(4)^{-1}.$$ 
Hence $$p(\frak a, N,K)\leq \zeta(4)^{-2}N^4 D_K^{-4} N(\frak a)^2.$$
If $\frak p$ is inert and $\frak p| N$, then $2[\ord_\frak p(NM)/2]-\ord_\frak p(M)\leq \ord_\frak p(N)$.
If $\frak p$ is ramified, then $\frak p| N$ and $\ord_\frak p(M)=0$ by definition.
Hence 
\begin{eqnarray*}
&&\sum_{\frak a} \delta(K,k,\frac a N_1(K)) |E_{K,\frak a N_1(K)}|\ p(\frak a,N,K)
\ll N^4D_K^{-4} \left(\sum_{d| D_K^2} d^2\right)\left(\sum_{d| \frac N{D_K}} d^2\right) \\
&&=N^4D_K^{-4} \sigma_2(D_K^2)\sigma_2(\tfrac N{D_K})\ll N^{6+\epsilon}D_K^{-2+\epsilon}.
\end{eqnarray*}

The sum over $M$ is majorized by $d(N)$. Since $h(K)\ll D_K^{\frac 12+\epsilon}$, 
$$
\dim S_{(k,1)}(\Gamma(N))\ll N^{6+\epsilon}d(N)\sum_{D_K | N} D_K^{-\frac 32+\epsilon}\ll N^{6+\epsilon}.
$$
 \end{proof}

\subsection{Automorphic counting measures}\label{pm}
Let $S'$ be a finite set of rational primes. Let $\xi=\xi_{\underline{k}}$ be an irreducible algebraic representation of $G(\R)$ with the 
highest weight $(k_1,k_2)$ satisfying $k_1\ge k_2\ge 3$ as in (\ref{algebraic-char}), and $D^{{\rm hol}}_{l_1,l_2}$ be 
the holomorphic discrete series of $G(\R)$ with the Harish-Chandra parameter $(l_1,l_2)=(k_1-1,k_2-2)$ and whose central character equals $\chi_{\xi^\vee}$ on $A_{G,\infty}$. Here $\chi_{\xi^\vee}$ is the central character of $\xi^\vee$. 
We fix a pseudo-coefficient $f_\xi\in C^\infty_c(G(\R),\xi^\vee)$ (cf. \cite{CD}, \cite[p.266]{Arthur1}, or \cite[p.161]{Arthur2}), 
so that  
$$\tr(\pi_\infty(f_\xi))=(-1)^{q(G(\R))}= -1,\ \pi_\infty=D^{{\rm hol}}_{l_1,l_2},
$$
where $q(G(\R))=\dim G(\R)/A_{G,\infty}U(2)=3$. 
By \cite{Hiraga}, we have,
\begin{eqnarray}\label{pseudo}
&& \text{if $l_2>1$},\quad \tr(\pi_\infty(f_\xi))=\begin{cases} -1, &\text{if $\pi_\infty= D^{{\rm hol}}_{l_1,l_2}$,}\\
0, &\text{otherwise};
\end{cases} \\
&& \text{if $l_1>2$ and $l_2=1$},\quad
\tr(\pi_\infty(f_\xi))=\begin{cases} -1, & \text{if $\pi_\inf= D^{{\rm hol}}_{l_1,l_2}$,} \\ 
1, &\text{if $\pi_\inf=\omega_{l_1}$,}\\
0, &\text{otherwise};
\end{cases} \nonumber \\
&& \text{if $(l_1,l_2)=(2,1)$}, \quad 
\tr(\pi_\infty(f_\xi))=\begin{cases} -1, & \text{if $\pi_\infty=D^{{\rm hol}}_{l_1,l_2}$,}\\
1, &\text{if $\pi_\inf=\omega_{l_1}$ or $\trep$,} \\ 
0, &\text{otherwise}.
\end{cases} \nonumber
\end{eqnarray}

Let $\pi_{S'}^0$ be a given unitary representation of $G(\Q_{S'})$.
Let $\delta_{\pi^0_{S'}}$ be the Dirac delta measure supported on $\pi^0_{S'}$ with respect to the Plancherel measure $\widehat{\mu}^{{\rm pl}}_{S'}$ on $\widehat{G(\Q_{S'})}$. 
Then we define a normalized Dirac delta measure supported on $\pi^0_{S'}$ by 
\begin{equation}\label{n-dirac}
\delta_{\pi^0_{S'},\xi}(\widehat{f}_{S'}):=\sum_{i}a_i|\nu(\alpha_i)|^{-\frac{k_1+k_2}{2}}_{S'} \ve_{S'}(\alpha_i)^{-1}
\delta_{\pi^0_{S'}}(\widehat{f}_{S'})
\end{equation}
for $f_{S'}=\sum_ia_i[G(\Z_{S'})\alpha_iG(\Z_{S'})]\in C^\infty_c(G(\Q_{S'}))$ 
with respect to the normalization of (\ref{local-global}). 
The factor $\ve_{S'}(\alpha_i)$ is defined by 
\begin{equation}\label{e-factor0}
\ve_{S'}(\alpha_i)=
\begin{cases}
2,  & \text{if $\nu(\alpha_i)\in (\Q^\times_{S'})^2, -1\in \nu(U)$, and $Sp_4(\Z)\cap U$ has non-trivial center}, \\
1, & \text{ otherwise}.
 \end{cases}
\end{equation}

We do not need the property of the Plancherel measure except for the following Plancherel formula by Harish-Chandra:
\begin{equation}
\widehat{\mu}^{\rm pl}_{S'}(\widehat{f}_{S'})=f_{S'}(1).
\end{equation}

For the above $\xi$ and any compact open subgroup $U$ of $G(\A^{S',\infty})$,  
we define a counting measure on $\widehat{G(\Q_{S'})}$ by 
\begin{equation}\label{measures}
\widehat{\mu}_{U,\xi_{\underline{k}}, D_{l_1,l_2}^{\rm hol}}:= \frac{1}{{\rm vol}(G(\Q)A_{G,\infty}\bs G(\A))\cdot {\rm dim}\, \xi_{\underline{k}}}\sum_{\pi^0_{S'}\in \widehat{G(\Q_{S'})}}
\mu^{S',\infty}(U)m_{\rm cusp}(\pi^0_{S'};U,\xi_{\underline{k}},D_{l_1,l_2}^{\rm hol})\delta_{\pi^0_{S'},\xi}
\end{equation}
where for a given unitary representation $\pi_{S'}^0$ of $G(\Q_{S'})$, 
the normalized multiplicity $m_{\rm cusp}(\pi^0_{S'};U,\xi_{\underline{k}}, D_{l_1,l_2}^{\rm hol})$ 
is given by 
\begin{equation}\label{mult}
m_{\rm cusp}(\pi^0_{S'};U,\xi_{\underline{k}}, D_{l_1,l_2}^{\rm hol})= \sum_{\pi\in \Pi(G(\A))\atop \pi_{S'}\simeq \pi^0_{S'},\, \pi_\infty\simeq D_{l_1,l_2}^{\rm hol}}m_{\rm cusp}(\pi){\rm tr}(\pi^{S',\infty}(f_U))\cdot \tr(\pi_\infty(f_\xi)),
\end{equation}
where $\Pi(G(\A))$ stands for the set of all isomorphism classes of automorphic representations of $G(\A)$.

Since $\omega_{l_1}$ occurs in both cuspidal spectrum and residual spectrum, we define, for $\ast\in \{{\rm cusp,res}\}$, 

\begin{equation*}
\widehat{\mu}_{U,\xi_{\underline{k}}, \omega_{l_1},\ast}:=\frac{1}{{\rm vol}(G(\Q)A_{G,\infty}\bs G(\A))\cdot {\rm dim}\, \xi_{\underline{k}}}\sum_{\pi^0_S\in \widehat{G(\Q_S)}}
\mu^{S,\infty}(U) m_{\ast}(\pi^0_S;U,\xi_{\underline{k}},\omega_{l_1})\delta_{\pi^0_S,\xi}
\end{equation*}
where 
\begin{equation*}
m_{\ast}(\pi^0_S;U,\xi_{\underline{k}},\omega_{l_1})=\sum_{\pi\in \Pi(G(\A))\atop \pi_{S'}\simeq \pi^0_{S'},\, \pi_\infty\simeq \omega_{l_1}}m_\ast(\pi){\rm tr}(\pi^{S',\infty}(f_U))\cdot \tr(\pi_\infty(f_\xi)),\ \ast\in \{{\rm cusp,res}\}.
\end{equation*} 

Since $\bold{1}$ occurs only in the residual spectrum, we define

\begin{equation*}
\widehat{\mu}_{U,\xi_{\underline{k}}, \bold{1}}:=\frac{1}{{\rm vol}(G(\Q)A_{G,\infty}\bs G(\A))\cdot {\rm dim}\, \xi_{\underline{k}}}\sum_{\pi^0_{S'}\in \widehat{G(\Q_{S'})}}
\mu^{S',\infty}(U) m_{\rm res}(\pi^0_{S'};U,\xi_{\underline{k}},\bold{1})\delta_{\pi^0_{S'},\xi}
\end{equation*}
where 
\begin{equation*}
m_{\rm res}(\pi^0_{S'};U,\xi_{\underline{k}},\bold{1})=\sum_{\pi\in \Pi(G(\A))\atop \pi_{S'}\simeq \pi^0_{S'},\, \pi_\infty\simeq \bold{1}}
m_{\rm res}(\pi){\rm tr}(\pi^{S',\infty}(f_U))\cdot \tr(\pi_\infty(f_\xi)).
\end{equation*} 

If $U\cap Sp_4(\Z)$ has a non-trivial center, then it contains $-E_4$. 
In the case when $U$ contains $-E_4$ and $-1\in \nu(U)$, there are two ways to extend a classical form $F$ to an adelic form $\phi$ of 
level $U$. 
A way is explained in (\ref{const-auto}). Since $d=\diag(-1,1,1,-1)\in U,\ \nu(d)=-1$, we have a holomorphic form $F_d$ defined by 
$$F_d(Z):=\overline{F(dZ)}=
\overline{F(
\begin{pmatrix}
-z_1 & z_2 \\
z_2 & -z_3 
\end{pmatrix}
)}.$$
Then $F_d$ can be also extended to an adelic form of 
level $U$.  
Both of them generate the same automorphic representation. This explains a meaning of (\ref{e-factor0}).   

\section{Arthur's invariant trace formula and some calculations}\label{s3}

In this section we make use of Arthur's invariant trace formula, and as in \cite{Shin}, \cite{ST}, we relate the Plancherel measure with the spectral expansion of the trace formula.
Then this leads to the calculation of the geometric side.
In our setting the pseudo-coefficient is chosen in a single discrete series. This causes the contribution from unipotent elements.
This contribution should be understood in terms of endoscopic representations which appear in the spectral side.
In general we do not know how to control the unipotent contribution, but in our case we compute every terms very explicitly.
We will give several estimates for invariants which appear in the trace formula.

Recall the notations $G$, $B$, $T$, and $M_j$ $(j=0,1,2)$ given in Section \ref{smf}.
Set $P_0=B$ and $M_0=T$.
We denote by $N_0$ the unipotent radical of $P_0$.
The Weyl group $W_0^G(=N_G(T)/T)$ for $M_0$ in $G$ is generated by two elements $s_0$ and $s_1$ which satisfies the relations $s_0^2=s_1^2=1$ and $s_0s_1s_0s_1=s_1s_0s_1s_0$.
We put $s_2=s_0s_1s_0$.
Then, we have
\[
W_0^G=\{ 1,\;\; s_0,\;\; s_1,\;\; s_2,\;\; s_0s_1,\;\; s_0s_2,\;\; s_1s_2,\;\; s_0s_1s_2    \}.
\]
For $s_0$ and $s_2$ in $W_0^G$, their representatives $w_{s_0}$ and $w_{s_2}$ in $G(\Q)\cap K$ can be chosen as
\[
w_{s_0}=\left(
\begin{array}{cccc}
0&1&0& 0\\
1&0&0& 0\\
0&0&0& 1\\
0&0&1& 0
\end{array}\right),\ 
w_{s_2}=\left(
\begin{array}{cccc}
1&0&0& 0\\
0&0&0& 1\\
0&0&1& 0\\
0&-1&0& 0
\end{array}\right).
\]
For all elements $s$ in $W_0^G$, we fix their representatives $w_s$ by $w_{s_0}$, $w_{s_2}$, and some products like $w_{s_1}=w_{s_0}w_{s_2}w_{s_0}$.
We also find
\[
W_0^{M_0}=1, \quad W_0^{M_1}=\{ 1, \;\; s_0\}, \quad  W_0^{M_2}=\{ 1, \;\; s_2 \} .
\]
For $s\in W_0^G$ and $H\subset G$, we set $sH=w_sHw_s^{-1}$.
The set of all Levi subgroups containing $M_0$ is given by
\[
\cL=\{ M_0 , \; \; M_1, \; \; s_1M_1, \;\; M_2,  \;\; s_0M_2, \;\; G  \}.
\]

Set
\[
\bK_\inf=\{ \begin{pmatrix}A& B \\ -B&A \end{pmatrix}\in G(\R)  \}, \quad \bK_v=G(\Z_v) \quad (\forall v<\inf).
\]
Then $\bK_v$ is a maximal compact subgroup of $G(\Q_v)$ and $\bK=\prod_v \bK_v$ is also a maximal compact subgroup of $G(\A)$.
We normalize Haar measures $\d k_v$ on $\bK_v$ as $\int_{\bK_v}\d k_v=1$.
A Haar measure $\d k$ on $\bK$ is defined by $\d k=\prod_v \d k_v$.
We also choose Haar measures $\d x_v$ on $\Q_v$ as $\int_{\Z_v}\d x_v=1$ $(v<\inf)$ and the Lebesgue measure $\d x_\inf$ on $\R$.
A Haar measure $\d x$ on $\A$ is defined by $\d x=\prod_v \d x_v$.
For each $M$ in $\cL$, we fix Haar measures on $A_M(\R)^0$ as in \cite[Condition 5.1]{HW}.
Moreover, we fix a Haar measure on $G(\A)$.
By the same manner as in \cite[p.32]{Arthur2} we normalize Haar measures on $M(\A)^1$.

\subsection{Characters of holomorphic discrete series of $Sp_4(\R)$}

We recall character formulas for holomorphic discrete series of $Sp_4(\R)$. 
These are necessary to control the geometric side  $I_{{\rm geom}}(f)$ of Arthur's invariant trace formula. 

Let
\[
t_4(\theta_1,\theta_2)=\begin{pmatrix}\cos\theta_1&0&\sin\theta_1&0 \\ 0&\cos\theta_2&0&\sin\theta_2 \\ -\sin\theta_1&0&\cos\theta_1&0 \\ 0&-\sin\theta_2&0&\cos\theta_2  \end{pmatrix}\quad (\theta_1,\theta_2\in\R).
\]
We define a compact Cartan subgroup $T_4$ of $Sp_4(\R)$ as
\[
T_4=\{ t_4(\theta_1,\theta_2) \mid \theta_1,\theta_2\in\R\}.
\]
We write $T_4^\reg$ the subset of regular elements of $T_4$.
For each $(l_1,l_2)$ in $\Z\oplus\Z$, a function $\Theta_{l_1,l_2}$ on $T_4^\reg$ is defined by
\begin{equation}\label{ec1}
\Theta_{l_1,l_2}(t_4(\theta_1,\theta_2))=\frac{-e^{il_1\theta_1+il_2\theta_2}+e^{il_2\theta_1+il_1\theta_2}}{(e^{i\theta_1}-e^{-i\theta_1})(e^{i\theta_2}-e^{-i\theta_2})( 1 - e^{i\theta_1+i\theta_2})(e^{-i\theta_1}-e^{-i\theta_2})}.
\end{equation}
Assume that $(l_1,l_2)$ satisfies $l_1>l_2>0$.
Then, there exists a unique holomorphic discrete series $D_{l_1,l_2}$ of $Sp_4(\R)$ whose character equals $\Theta_{l_1,l_2}$ (cf. \cite[Theorem 12.7]{Knapp}).
The parameter $(l_1,l_2)$ is called the Harish-Chandra parameter for discrete series representations.

Throughout this section, we use the Harish-Chandra parameter $(l_1,l_2)$ to describe holomorphic discrete series representation 
instead of the minimal $\bK_\inf$-type $(k_1,k_2)$. Since $(k_1,k_2)=(l_1+1,l_2+2)$, one can easily convert the results from one to the other. 

For $a$, $a_1$, $a_2$, $\theta$ in $\R$, we set
\[
t_0(a_1,a_2)=\diag(e^{a_1},e^{a_2},e^{-a_1},e^{-a_2}),
\]
\[
t_1(a,\theta)=\begin{pmatrix}e^a \cos\theta&e^a \sin\theta&0&0 \\ -e^a \sin\theta&e^a \cos\theta&0&0 \\ 0&0&e^{-a} \cos\theta&e^{-a} \sin\theta \\ 0&0& -e^{-a} \sin\theta&e^{-a} \cos\theta \end{pmatrix},\]
\[
t_2(a,\theta)=\begin{pmatrix}e^a&0&0&0 \\ 0&\cos\theta&0&\sin\theta \\ 0&0&e^{-a}&0 \\ 0&-\sin\theta&0&\cos\theta \end{pmatrix}.
\]
It is known that the group $Sp_4(\R)$ has the four Cartan subgroups $T_0$, $T_1$, $T_2$, $T_4$ up to conjugation, where
\[
T_0=\{t_0(a_1,a_2) \mid a_1,a_2\in\R_{>0}\}\subset M_0(\R),
\]
\[
T_j=\{ t_j(a,\theta) \mid a\in\R_{>0},\;\; \theta\in\R\} \subset M_j(\R)\quad (j=1,2).
\]
Let $T_j^\reg$ denote the set of regular elements of $T_j$.
A character formula of $\Theta_{l_1,l_2}$ on $T_j^\reg$ is known (cf. \cite{Martens,Hecht,Hirai}).
Here we summarize it briefly.
For $t_0(a_1,a_2)$ in $T_0^\reg$, we have
\begin{equation}\label{ec2}
\Theta_{l_1,l_2}(t_0(a_1,a_2))=\frac{ \{-e^{-l_1|a_1| -l_2|a_2|} + e^{-l_2|a_1| -l_1|a_2|}\} \times \sgn(a_1a_2) }{(e^{a_1}-e^{-a_1})(e^{a_2}-e^{-a_2})(1-e^{a_1+a_2})(e^{-a_1}-e^{-a_2})}
\end{equation}
if $|a_1|>|a_2|>0$.
For $t_0(a_1,a_2)$ in $T_0^\reg$ and $\delta_1=\diag(1,-1,1,-1)$, we get
\begin{equation}\label{ec3}
\Theta_{l_1,l_2}(\delta_1 t_0(a_1,a_2))=\frac{ \{ -(-1)^{l_2} e^{-l_1|a_1| -l_2|a_2|} + (-1)^{l_1}e^{-l_2|a_1| -l_1|a_2|} \} \times \sgn(a_1a_2) }{(e^{a_1}-e^{-a_1})(-e^{a_2}+e^{-a_2})(1+e^{a_1+a_2})(e^{-a_1}+e^{-a_2}) }.
\end{equation}
For $t_1(a,\theta)$ in $T_1^\reg$, we have
\begin{equation}\label{ec4}
\Theta_{l_1,l_2}(t_1(a,\theta))=\frac{\{ - e^{-l_1(|a|+i\theta) -l_2(|a|-i\theta)}+ e^{-l_2(|a|+i\theta) -l_1(|a|-i\theta)} \}\times \sgn(a)}{(e^{a+i\theta}-e^{-a-i\theta})(e^{a-i\theta}-e^{-a+i\theta})(1-e^{2a})(e^{-a-i\theta}-e^{-a+i\theta})}.
\end{equation}
For $t_2(a,\theta)$ in $T_2^\reg$, we have
\begin{equation}\label{ec5}
\Theta_{l_1,l_2}(t_2(a,\theta))=\frac{ \{  e^{-l_1|a|+il_2\theta } - e^{-l_2|a| + il_1\theta } \}\times \sgn(a) }{(e^{i\theta}-e^{-i\theta})(e^{a}-e^{-a})(1-e^{i\theta+a})(e^{-a}-e^{-i\theta })}.
\end{equation}
Since $\Theta_{l_1,l_2}(-\gamma)=(-1)^{l_1+l_2}\Theta_{l_1,l_2}(\gamma)$ and $\Theta_{l_1,l_2}(g^{-1}\gamma g)=\Theta_{l_1,l_2}(\gamma)$ $(g$, $\gamma\in Sp_4(\R))$, the formulas \eqref{ec1}$\sim$\eqref{ec5} cover all cases for regular semisimple elements in $Sp_4(\R)$.

A closed subgroup $Sp_4^\pm(\R)$ of $G(\R)=GSp_4(\R)$ is defined by
\[
Sp_4^\pm(\R)=\left\{g\in GL_4(\R)\ \Bigg|\ g \begin{pmatrix} O_2&E_2 \\ -E_2&O_2 \end{pmatrix} \,^tg = \pm\begin{pmatrix} O_2&E_2 \\ -E_2&O_2 \end{pmatrix}  \right\}.
\]
Note that an isomorphism $G(\R)\cong A_{G,\infty}\times Sp_4^\pm(\R)$ holds.
We write $\Theta_{l_1,l_2}^{{\rm hol}}$ for the holomorphic discrete series $D^{{\rm hol}}_{l_1,l_2}$ of $G(\R)$ $(l_1>l_2>0)$.
For the algebraic representation $\xi$ of $G(\R)$ corresponding to $(l_1,l_2)$, we denote by $\chi_\xi$ the central character of $\xi$. 
For $\delta=\diag(1,1,-1,-1)\in Sp_4^\pm(\R)$, it is obvious that
\[
Sp_4^\pm(\R)=Sp_4(\R)\sqcup Sp_4(\R)\delta.
\]
Hence, considering the action of $\delta$, one finds
\begin{equation}\label{charGSp}
\Theta_{l_1,l_2}^{{\rm hol}}(zg)=\chi_\xi(z)^{-1} \times\{\Theta_{l_1,l_2}(g)+\overline{\Theta_{l_1,l_2}(g)}\} \quad (z\in A_{G,\infty}, \;\; g\in Sp_4(\R)).
\end{equation}
\begin{lem}\label{clv}
We get $\Theta_{l_1,l_2}^{{\rm hol}}(\gamma)=0$ for any regular semisimple element $\gamma$ in $A_{G,\inf}Sp_4(\R)\delta$.
\end{lem}
\begin{proof}
Let $H_{l_1,l_2}$ denote a representation space of $D_{l_1,l_2}$.
There exists an anti-holomorphic discrete series $\overline{D_{l_1,l_2}}$ with the same infinitesimal character as $D_{l_1,l_2}$.
The Hilbert space $H_{l_1,l_2}$ is also regarded as a representation space of $\overline{D_{l_1,l_2}}$, because $\overline{D_{l_1,l_2}}$ can be defined by $\overline{D_{l_1,l_2}}(g)v=D_{l_1,l_2}(\delta g\delta)v$ $(v\in H_{l_1,l_2})$.
Therefore, the space $H_{l_1,l_2}\oplus H_{l_1,l_2}$ becomes a representation space of $D^{{\rm hol}}_{l_1,l_2}$.
Namely, we have
\[
D^{{\rm hol}}_{l_1,l_2}(g)(v_1,v_2)=(D_{l_1,l_2}(g)v_1,D_{l_1,l_2}(\delta g\delta)v_2) , \quad D^{{\rm hol}}_{l_1,l_2}(\delta)(v_1,v_2)=(v_2,v_1)
\]
for each vector $(v_1,v_2)$ in $H_{l_1,l_2}\oplus H_{l_1,l_2}$ and each element $g\in Sp_4(\R)$.
%This is well-defined, because $D^{{\rm hol}}_{l_1,l_2}(g)D^{{\rm hol}}_{l_1,l_2}(\delta)=D^{{\rm hol}}_{l_1,l_2}(\delta)D^{{\rm hol}}_{l_1,l_2}(\delta g \delta)$.
By an orthonormal basis $\{v_j\}_{j=1}^\inf$ of $H_{l_1,l_2}$, we can choose an orthonormal basis $\{ (v_j,0) , \;\; (0,v_k) \mid j,k=1,2,\dots  \}$ of $H_{l_1,l_2}\oplus H_{l_1,l_2}$.
Hence, for any function $f$ in $C_c^\inf(GSp_4(\R))$ whose support is contained in $A_{G,\infty}Sp_4(\R)\delta$, it follows that
\[
\langle D_{l_1,l_2}^{{\rm hol}}(f)(v_j,0)\, , \, (v_j,0) \rangle = \langle D_{l_1,l_2}^{{\rm hol}}(f)(0,v_k)\, , \, (0,v_k) \rangle=0 \quad (j,k=1,2,\dots).
\]
This is obvious if one sees the action of $D_{l_1,l_2}^{{\rm hol}}(\delta)$.
Thus, this lemma is proved.
\end{proof}

Let $D^M(\gamma)$ denote the Weyl denominator of $\gamma$ in $M(\R)$ and let $W(M,T)$ denote the Weyl group with respect to a torus $T$ in $M$ over $\R$.
For each $0\leq j\leq 2$, we set $M=M_j$ and $T$ is a torus in $M$ over $\R$ such that $T(\R)=T_j \sqcup (-T_j)$ when $j=1$ or $2$, and $T(\R)=M_0(\R)= T_0\sqcup (-T_0)\sqcup \delta_1 T_0\sqcup (-\delta_1)T_0$ when $j=0$.
In case of $M=G$, there exists a torus $T$ over $\R$ such that $T(\R)=T_4$.
We say that $\Theta_{l_1,l_2}^{{\rm hol}}$ is stable for $M$ if $|D^{M}(\gamma)|^{-1/2}|D^G(\gamma)|^{1/2}\Theta^{{\rm hol}}_{l_1,l_2}(\gamma)$ on regular elements $\gamma$ of $T(\R)$ is a finite, $W(M,T)$-invariant linear combination of quasi-characters.
This condition is the same as the assumption for $\Phi(\gamma)$ in \cite[Lemma 4.1 in p.271]{Arthur1}.
\begin{lem}\label{cst}
The character $\Theta_{l_1,l_2}^{{\rm hol}}$ is stable for $M_0$, $M_1$, $M_2$, and is not stable for $G$.
\end{lem}
\begin{proof}
This can be proved by using \eqref{ec1}$\sim$\eqref{ec5}, \eqref{charGSp} and Lemma \ref{clv}.
\end{proof}

\subsection{Spectral side}\label{spectral}

Let $\xi$ be an irreducible algebraic representation of $G(\R)$ with the highest weight $(k_1,k_2)$ with $k_1\ge k_2\ge 3$, and
$D^{{\rm hol}}_{l_1,l_2}$ denote the holomorphic discrete series representation of $G(\R)$ with 
the Harish-Chandra parameter $(l_1,l_2)$ so that the central character is same as $\xi^{\vee}$ on $A_{G,\infty}$, where $(l_1,l_2)=(k_1-1,k_2-2)$.
Choose a test function $h$ in $C_c^\inf(G(\A_\fin))$ and we write $f_\xi$ for a pseudo-coefficient of $D_{l_1,l_2}^{\rm hol}$.
If we set
\begin{equation}\label{test-funct}
f=f_\xi h 
\end{equation}
then $f$ belongs to the Hecke algebra of $\bK$-finite functions in $C_c^\inf(G(\A)^1)$.

By \eqref{pseudo}, for $\underline{l}=(l_1,l_2)$, we set
\begin{equation}\label{hol-coeff}
\Pi(\underline{l},\xi)= \begin{cases} \{ \,  D^{{\rm hol}}_{l_1,l_2} \, \} & \text{if $l_2>1$,}\\
 \{ \,  D^{{\rm hol}}_{l_1,l_2}, \; \omega_{l_1} \, \} & \text{if $l_1>2$ and $l_2=1$,} \\
  \{ \,  D^{{\rm hol}}_{l_1,l_2}, \; \omega_2, \; \trep \, \} & \text{if $(l_1,l_2)=(2,1)$.} 
\end{cases}
\end{equation}
They correspond to 
(1) $k_1\ge k_2\ge 4$, (2) $k_1>k_2=3$, and (3) $k_1=k_2=3$, respectively in terms of the classical weight for Siegel modular forms.
Then the spectral side of Arthur's invariant trace formula for $f$ is
\begin{equation}\label{spectral}
I_{{\rm spec}}(f)=\sum_{\pi=\pi_\inf\otimes\pi_\fin\in \widehat{G(\A)}, \; \pi_\inf\in\Pi(\underline{l},\xi) } m_{\rm disc}(\pi) \tr(\pi_\infty(f_\xi))
 \,  \tr(\pi_\fin(h))
\end{equation}
where $m_{{\rm disc}}(\pi)$ denotes the multiplicity of $\pi$ in the discrete spectrum of $L^2_{{\rm disc}}(G(\Q)\bsl G(\A),\chi_{\xi^\vee})$ and the unramified Hecke action inside $\tr(\pi_\fin(h))$ is normalized as \eqref{HLL}.  
For the proof of this expansion, we refer to \cite[Section 3]{Arthur1}.
It is fortunate that cohomological, non-holomorphic Saito-Kurokawa representations do not appear in this case.
If $f_\xi$ is a pseudo-coefficient of a large discrete series whose parameter satisfies $|l_1-l_2|=1$, then it appears on the spectral side.

We are now ready to relate the above measures to the spectral side $I_{{\rm spec}}(f)$ and also to the geometric side in 
Arthur's trace formula, as in \cite[Proposition 4.2]{Shin}.

\begin{prop}\label{trace} 
Let $S'$ be a finite set of finite places of $\Q$. 
For any compact subgroup $U$ of $G(\A^{S',\infty})$ and $f_{S'}=\ds\sum_ia_if_{S',\alpha_i}\in C^\infty_c(G(\Q_{S'}))$,  
$f_{S',\alpha_i}=[G(\Z_{S'})\alpha_iG(\Z_{S'})],\ \alpha_i\in T(\Q_{S'})$, 

\begin{eqnarray*}
&&\ds\sum_ia_i \frac{I_{{\rm geom}}(f_{U}f_{S',\alpha_i}f_{\xi})}{\ve_{S'}(\alpha_i)\overline{\mu}(G(\Q)A_{G,\infty}\bs G(\A))\dim \xi}
=\ds\sum_ia_i \frac{I_{{\rm spec}}(f_{U}f_{S',\alpha_i}f_{\xi})}{\ve_{S'}(\alpha_i)\overline{\mu}(G(\Q)A_{G,\infty}\bs G(\A))\dim \xi}\\
&&=\begin{cases} 
\widehat{\mu}_{U,\xi_{\underline{k}},D^{{\rm hol}}_{l_1,l_2}}(\widehat{f}_{S'})), &\text{if $l_2>1$,}\\
\widehat{\mu}_{U,\xi_{\underline{k}},D^{{\rm hol}}_{l_1,l_2}}(\widehat{f}_{S'}))+\ds\sum_{\ast\in\{{\rm cusp},{\rm res}\}} \widehat{\mu}_{U,\xi_{\underline{k}},\omega_{l_1},\ast}(\widehat{f}_{S'}), &\text{if $l_1>2$ and $l_2=1$,}\\
\widehat{\mu}_{U,\xi_{\underline{k}},D^{{\rm hol}}_{l_1,l_2}}(\widehat{f}_{S'}))+\widehat{\mu}_{U,\xi_{\underline{k}},\bold{1}}(\widehat{f}_{S'})+\ds\sum_{\ast\in\{{\rm cusp},{\rm res}\}} \widehat{\mu}_{U,\xi_{\underline{k}},\omega_{l_1},\ast}(\widehat{f}_{S'}), 
&\text{if $(l_1,l_2)=(2,1)$},
\end{cases}
\end{eqnarray*}
where $f_U$ is the characteristic function of $U$ and 
the factor $\ve_{S'}(\alpha_i)$ is defined by (\ref{e-factor0}).
\end{prop}
\begin{proof} The claim follows from the definition and 
$\delta_{\pi^0_{S'}}(\widehat{f}_{S'})=\ds\int_{\widehat{G(\Q_{S'})}}\widehat{f}_{S'} d\delta_{\pi^0_{S'}}=
\tr(\pi^0_{S'}(\widehat{f}_{S'}))$. 
\end{proof}

\subsection{Geometric side}\label{gs1}
Fix a finite set $S'$ of finite places.  
Let $S_0$ be a finite set of finite places containing $S'$ and put  $S=S_0\cup\{\inf\}$. 
We will consider Hecke operators for $G(\Q_{S'})$ while we will vary $S_0$ and hence $S$ in the geometric side of the 
Arthur's trace formula. Put $q(G(\R))=\frac{1}{2}\dim G(\R)/K_\infty A_{G,\infty}=3$. 
Assume that $S$ is sufficiently large.
Then, the geometric side of Arthur's invariant trace formula is, for $f=f_\xi h$ as in (\ref{test-funct}),
\begin{equation}\label{geom}
I_{{\rm geom}}(f)=\sum_{M\in\cL}(-1)^{q(G(\R))+\dim(A_M/A_G)}\frac{|W_0^M|}{|W_0^G|}\sum_{\gamma\in (M(\Q))_{M,S}}a^M(S,\gamma)\, I_M^G(\gamma,f_\xi)\, J_M^M(\gamma,h_P),
\end{equation}
where $(M(\Q))_{M,S}$ denotes the set of $(M,S)$-equivalence classes in $M(\Q)$ (cf. \cite[p.113]{Arthur2}) 
which turns out to be a finite set, for each $M$ in $\cL$ we choose a parabolic subgroup $P$ such that $M$ is a Levi subgroup of $P$, and we set
\begin{equation}\label{hp}
h_P(m)=\delta_P(m)^{1/2}\int_{\bK_{S_0}}\int_{N_P(\Q_{S_0})} h(k^{-1}mnk)\, \d n \, \d k \qquad (m\in M(\Q_{S_0}))
\end{equation}
($\bK_{S_0}=\prod_{v\in S_0}\bK_v$ and $\delta_P$ is the modular function of $P(\A)$). 
Regarding $ J_M^M(\gamma,h_P)$ it follows from (\ref{hp}) that  $h_P(m)=0$ unless  
 $k^{-1}m^{-1}\gamma mnk\in {\rm Supp}(h)$ where $m\in M(\Q_{S_0}),\ k\in K_{S_0}$, and $n\in N_P(\Q_{S_0})$. 
This implies 
\begin{equation}\label{n-gamma}
\nu(\gamma)\in \nu({\rm Supp}(h)). 
\end{equation}
For the definitions of the invariant weighted orbital integral $I_M^G(\gamma,f_\xi)$ and the orbital integral $J_M^M(\gamma,h_P)$, we refer to \cite[Sections 18 and 23]{Arthur2}.
The factor $a^M(S,\gamma)$ is called the global coefficient (see \cite[Sections 19]{Arthur2}).
We will later give their details for some cases which are necessary to explain our estimations.

Let $Z_G$ denote the center of $G$.
For each element $\gamma$ in $G(\A)$, we write $\{\gamma\}_G$ for the $G(\Q)$-conjugacy class of $\gamma$.
For convenience, we set
\[
u_\min(x)=\begin{pmatrix}1&0&x&0 \\ 0&1&0&0 \\ 0&0&1&0 \\ 0&0&0&1 \end{pmatrix}, \quad  \delta_1(x,y)=\begin{pmatrix}1&x&0&y \\ 0&-1&-y&0 \\ 0&0&1&0 \\ 0&0&x&-1 \end{pmatrix},
\]
\[
u_\min=u_\min(1), \quad \delta_1=\delta_1(0,0).
\]
It is clear that all minimal unipotent elements belong to $\{ u_\min\}_G$.
Let $\gamma$ be an element of $G(\Q)$.
If $\gamma$ is a semisimple element whose diagonalization is $\delta_1$, then $\gamma$ is $G(\Q)$-conjugate to $\delta_1$.
This fact can be proved by using the Galois cohomology.

To study concretely the geometric side, we separate the sum into the following seven types:
\[
I_{{\rm geom}}(f)=I_1(f)+I_2(f)+I_3(f)+I_4(f)+I_5(f)+I_6(f)+I_7(f)
\]
where 
\begin{itemize}
\item $I_1(f)$: $M=G$ and $\gamma\in Z_G(\Q)$,
\item $I_2(f)$: $M=G$ and $\gamma\in Z_G(\Q)\{u_\min\}_G$,
\item $I_3(f)$: $M=G$ and $\gamma\in Z_G(\Q)\{\delta_1\}_G$,
\item $I_4(f)$: $M=G$ and $\gamma$ is semisimple and $\gamma\not\in Z_G(\Q)\sqcup Z_G(\Q)\{\delta_1\}_G$, 
\item $I_5(f)$: $M=G$ and $\gamma$ is not semisimple and $\gamma\not\in Z_G(\Q)\{u_\min\}_G$,
\item $I_6(f)$: $M\not=G$ and $\gamma$ is semisimple,
\item $I_7(f)$: $M\not=G$ and $\gamma$ is non-semisimple.
\end{itemize}
The main term will be $I_1(f)$ and the second main term will be $I_2(f)$ in general, but also $I_3(f)$ in weight aspect.
The terms $I_4(f)$, $I_5(f)$, $I_6(f)$ and $I_7(f)$ never contribute to both the main term and the second main term in any aspect.
We will estimate each terms after the detailed studies of invariants $a^M(S,\gamma)$, $I_M^G(\gamma,f_\xi)$, and $J_M^M(\gamma,h_P)$. 
Since we clearly know
\[
I_1(f)=(-1)^{q(G(\R))}\sum_{z\in Z_G(\Q)}\mathrm{vol}(G(\Q)\bsl G(\A)^1)\, f(z), 
\]
we do not discuss it throughout this section.

If $M=G$, then $I_G^G(\gamma)=J_G^G(\gamma)$ and $h_Q=h$.
For simplicity, we set
\[
J_G(\gamma,f_\xi)=I_G^G(\gamma,f_\xi), \quad J_G(\gamma,h)=J_G^G(\gamma,h).
\]

\subsection{Some measures concerning $I_2(f)$ and $I_3(f)$}\label{I2I3}

We choose some measures on centralizers to calculate explicitly the orbital integrals $J_G(zu_\min,f_\xi)$ and $J_G(z\delta_1,f_\xi)$.
The centralizer $G_{u_\min}$ of $u_\min$ in $G$ is given by
\[
G_{u_\min}=\{z\begin{pmatrix}1&*&*&* \\ 0&1&*&0 \\ 0&0&1&0 \\ 0&0&*&1 \end{pmatrix} \begin{pmatrix}1&0&0&0 \\ 0&*&0&* \\ 0&0&1&0 \\ 0&*&0&* \end{pmatrix}\in G\} .
\]
The centralizer $G_{\delta_1}$ of $\delta_1$ satisfies
\[
G_{\delta_1}=\{  \begin{pmatrix}*&0&*&0 \\ 0&*&0&* \\ *&0&*&0 \\ 0&*&0&* \end{pmatrix}  \in G  \} \cong \{ (g_1,g_2)\in GL_2\times GL_2 \mid \det(g_1)=\det(g_2)   \}.
\]
If we want to determine $J_G(zu_\min,f_\xi)$ and $J_G(z\delta_1,f_\xi)$ precisely, we should choose measures on the centralizers.
A Haar measure on $\R_{>0}$ is chosen by $x^{-1}\d x$ and a Haar measure on $SL_2(\R)$ is fixed by $(2\pi)^{-1}v^{-3}\d u \, \d v \, \d \theta$ for $\begin{pmatrix}1&u \\ 0&1 \end{pmatrix}\begin{pmatrix}v&0 \\ 0&v^{-1} \end{pmatrix}\begin{pmatrix}\cos\theta& \sin\theta \\ -\sin\theta &\cos\theta \end{pmatrix}$.
On the groups $\{\pm E_4\}$ and $\{E_4,\; \delta\}$, we take the counting measure.
By the isomorphism $G_{u_\min}(\R)\cong \{\pm E_4 \}\times \R_{>0}\times (\R^3 \ltimes SL_2(\R))$ (resp. $G_{\delta_1}(\R)\cong\{E_4,\; \delta\}\rtimes(\R_{>0}\times SL_2(\R)\times SL_2(\R))$), we obtain a Haar measure on $G_{u_\min}(\R)$ (resp. $G_{\delta_1}(\R)$).

To simplify the description for the global coefficient $a^G(S,u_\min)$, we choose measures on the orbits as below.
We define $J_G(zu_\min,f_\xi)$ and $J_G(zu_\min,h)$ as
\begin{equation}\label{min1}
J_G(zu_\min,f_\xi)=\int_\R \int_{\bK_\inf} f_\xi(zk^{-1}u_\min(x) k)\, |x|_\inf \, \d k \, \d x,
\end{equation}
\begin{equation}\label{min2}
J_G(zu_\min,h)=\prod_{p\in S_0}(1-p^{-1})^{-1}\times \int_{\Q_{S_0}} \int_{\bK_{S_0}} h(zk^{-1}u_\min(x) k)\, |x|_{S_0}\, \d k \, \d x.
\end{equation}
If we choose a suitable Haar measure on $G(\R)$, then the integral $J_G(zu_\min)$ coincides with the orbital integral of $zu_\min$ normalized by the above mentioned measure on $G_{u_\min}(\R)$.

We may define $J_G(z\delta_1,h)$ as
\begin{equation}\label{delta1}
J_G(z\delta_1,h)=\int_{\Q_{S_0}}\int_{\Q_{S_0}}\int_{\bK_{S_0}}h(zk^{-1}\delta_1(x,y)k)\, \d k\, \d x \, \d y.
\end{equation}
This definition is useful to compute spherical Hecke algebras.

We will determine the total contributions $I_2(f)$ and $I_3(f)$ up to constant multiples (cf. Lemmas \ref{I21}, \ref{I22}, \ref{I31} and \ref{I32}).
We do not explicitly calculate the constants, because it is unnecessary for our main purpose.
However, if one wants to know their numerical values, one can explicitly calculate them by using Lemmas \ref{id3}, \ref{id4}, \eqref{gcm}, and choosing some normalizations of measures.

\subsection{Estimations and vanishings for $I_M^G(\gamma,f_\xi)$}

By \cite{Arthur3} and \cite{Arthur4} one knows
%\begin{equation}\label{inv-dist}
\[
I_M^G(\gamma,f_\xi)=(-1)^{\dim\mathfrak{a}_M^G}|D^G(\gamma)|^{1/2} \, \Theta_{l_1,l_2}^{{\rm hol}}(\gamma)
\]
for any $G$-regular semisimple element $\gamma$ which is $\R$-elliptic in $M$.
If $\gamma$ is not $\R$-elliptic, then $I_M^G(\gamma,f_\xi)$ vanishes.
Hence, our remaining work is to study its behaviors at singular elements.

\begin{lem}\label{id1}
Let $M$ be a proper Levi subgroup in $\cL$.
Then, for any $M(\R)$-conjugacy class $\gamma$ in $M(\Q)$, there exists a positive constant $c(\gamma)$ such that the absolute value of $I_M^G(\gamma,f_\xi)$ is bounded by $c(\gamma)\times \chi_\xi(\gamma)^{-1}\times \{l_1+l_2\}$.
Furthermore, we have $I_M^G(\gamma,f_\xi)=0$ if the semisimple part of $\gamma$ is not $\R$-elliptic in $M$.
In addition, the term $I_7(f)$ vanishes. 
\end{lem}
\begin{proof}
By Lemma \ref{cst} we can apply the same argument as in \cite[Proof of Theorem 5.1]{Arthur1} to $I_M^G(\gamma,f_\xi)$ related to $M_0$, $M_1$, $s_1M_1$, $M_2$ and $s_0M_2$, because the character satisfies the same assumption as in \cite[Lemma 4.1]{Arthur1}.
Hence, one can explicitly compute $I_M^G(\gamma,f_\xi)$.
In particular, $I_M^G(\gamma,f_\xi)$ vanishes for all non-semisimple conjugacy classes $\gamma$.
Hence, we get $I_7(f)=0$.
\end{proof}
A required estimation for $I_6(f)$ can be proved by this lemma (cf. Section \ref{sec6}).
Hence, it is enough to consider the terms related to $M=G$, i.e., $I_2(f)$, $I_3(f)$, $I_4(f)$, and $I_5(f)$.

For each non-semisimple conjugacy class $\gamma$ in $G(\R)$, the distribution $J_G(\gamma)$ on $C_c^\inf(G(\R))$ is expressed by a linear combination of limits of regular semisimple orbital integrals (cf. \cite[Appendix]{Arthur1}), but its coefficients are still unknown in general.
Hence, we should consider them case by case.
\begin{lem}\label{id2}
Let $\gamma$ be an element in $G(\Q)$.
We assume that the semisimple part of $\gamma$ is not in $Z(\Q)$ and $\gamma$ does not belong to $Z_G(\Q)\{\delta_1\}_G$.
Then, there exists a positive constant $c(\gamma)$ such that the absolute value of $J_G(\gamma,f_\xi)$ is bounded by $c(\gamma)\times \chi_\xi(\gamma)^{-1}\times \{l_1+l_2\}$.
In particular, we have $J_G(\gamma,f_\xi)=0$ if the semisimple part of $\gamma$ is not $\R$-elliptic in $G$.
\end{lem}
\begin{proof}
As for semisimple singular elements, it was done by Langlands using Harish-Chandra's limit formula (cf. \cite{Langlands}).
In case of $SL_2(\R)$, one can study the limit formula for the unipotent elements in the book \cite[Section 3, Chapter XI]{Knapp}.
Hence, one can easily calculate them (all such explicit calculations were done in \cite{Wakatsuki}).
\end{proof}

\begin{lem}\label{id3}
If $\gamma=z\delta_1$ $(z\in Z_G(\Q))$, then there exists a positive constant $c(\gamma)$ such that
\[
J_G(\gamma,f_\xi)=c(\gamma)\times \chi_\xi(z)^{-1}\times (-1)^{l_2}l_1l_2\{1+(-1)^{l_1-l_2-1}\}.
\]
If we choose the Haar measure given in Section \ref{I2I3} on the centralizer, we have $c(\gamma)=2^{-4}\pi^{-2}$.
\end{lem}
\begin{proof}
This follows from the limit formula for $SL_2(\R)$ (cf. \cite{Knapp}).
\end{proof}

Now, the remaining conjugacy classes are only unipotent orbits.
The group $G$ has the four unipotent classes; (1) regular, (2) subregular, (3) minimal, (4) unit.
\begin{lem}\label{id4}
If $\gamma=z u_\min$ $(z\in Z_G(\Q))$, then there exists a positive constant $c(\gamma)$ such that
\[
J_G(\gamma,f_\xi)= c(\gamma) \times \chi_\xi(z)^{-1}\times (l_1-l_2)(l_1+l_2).
\]
If we choose the Haar measure given in Section \ref{I2I3} on the centralizer, we have $c(\gamma) = - 2^{-3}\pi^{-3}$.
\end{lem}
\begin{proof}
This lemma can be proved by the limit formula of \cite{Rossmann}.
As for a suitable chamber and the constant $c(\gamma)$, we refer to \cite[Lemma 4.9 and Lemma 4.11]{Wakatsuki}
\end{proof}
\begin{lem}\label{id5}
There is only one regular unipotent $G(\R)$-conjugacy class $u_\mathrm{reg}$ in $G(\Q)$.
For $\gamma=z u_\mathrm{reg}$ $(z\in Z_G(\Q))$, we find $J_G(\gamma,f_\xi)=0$.
\end{lem}
\begin{proof}
This obviously follows from the limit formulas of \cite{Bozicevic} and \cite{Rossmann}.
\end{proof}
\begin{lem}\label{id6}
There are two subregular unipotent $G(\R)$-conjugacy classes $u_{\mathrm{sub},1}$ and $u_{\mathrm{sub},2}$ in $G(\Q)$.
For any $z$ in $Z_G(\Q)$, we have $J_G(z u_{\mathrm{sub},1},f_\xi)=J_G(z u_{\mathrm{sub},2},f_\xi)=0$.
\end{lem}
\begin{proof}
For a real symmetric matrix $S$ of degree $2$, we set $u_\mathrm{sub}(S)=\begin{pmatrix}E_2&S \\ O_2&E_2 \end{pmatrix}$.
Let $S_{++}=\diag(1,1)$, $S_{+-}=\diag(1,-1)$, and $S_{--}=\diag(-1,-1)$.
Then, $u(S_{++})$, $u(S_{+-})$, $u(S_{--})$ are representatives for subregular unipotent orbits of $Sp_4(\R)$.
But, the sum of the orbits of $u(S_{++})$ and $u(S_{--})$ forms a $G(\R)$-conjugacy class.
Then, we denote it by $u_\mathrm{sub,1}$, and let $u_\mathrm{sub,2}$ denote the $G(\R)$-conjugacy class of $u(S_{+-})$.

Using the limit formulas \cite{Bozicevic,Rossmann} and some associated constants \cite[Lemma 4.11]{Wakatsuki}, one gets
\[
J_{Sp_4}(zu(S_{++}),f_\xi)=-J_{Sp_4}(zu(S_{--}),f_\xi).
\]
Hence, it follows that $J_G(zu_{\mathrm{sub},1},f_\xi)=J_{Sp_4}(zu(S_{++}),f_\xi)+J_{Sp_4}(zu(S_{--}),f_\xi)=0$.
Since $f_\xi$ is cuspidal, we deduce $\int_{\bK_\inf}\int_{N_1(\R)}f_\xi(zk^{-1}nk) \d n \, \d k=0$ from the Plancherel formula for $M_1(\R)$.
By normalizing Haar measures on the centralizers, one can see that
\[
J_G(zu_{\mathrm{sub,1}},f_\xi)+J_G(zu_{\mathrm{sub,2}},f_\xi)=\int_{\bK_\inf}\int_{N_1(\R)}f_\xi(zk^{-1}nk) \d n \, \d k=0.
\]
Hence, we get $J_G(zu_{\mathrm{sub},2},f_\xi)=0$.
\end{proof}

\subsection{Global coefficients $a^G(S,\gamma)$}

For details of global coefficients $a^G(S,\gamma)$, we refer to \cite[Section 19]{Arthur2} and \cite{HW}.
It is known that $a^G(S,1)=\vol(G(\Q)\bsl G(\A)^1)$ holds.
For non-trivial unipotent orbits for $G$, they are explicitly calculated in \cite{HW}.
By Lemmas \ref{id5} and \ref{id6}, we need only an information for $a^G(S,u_\min)$.
Let $\gamma$ be a $(G,S)$-conjugacy class in $G(\Q)$.
We shall consider the case $\gamma$ is not unipotent.
We may reduce to the centralizers of the semisimple part $\gamma_s$ of $\gamma$.
For each element $\gamma_1$ in $G(\Q)$, we denote by $G_{\gamma_1,+}$ the centralizer of $\gamma_1$ in $G$ over $\Q$ and by $G_{\gamma_1}$ the connected component of $1$ in $G_{\gamma_1,+}$.
(Note that $G_{u_\min}=G_{u_\min,+}$ and $G_{\delta_1}=G_{\delta_1,+}$.)
We set $\iota(\gamma_s)=G_{\gamma_s,+}(\Q)/G_{\gamma_s}(\Q)$.
If $S$ is sufficiently large, then we have
\[
a^G(S,\gamma)=\varepsilon^G(\gamma_s)\, |\iota^G(\gamma_s)|^{-1} \, \sum_{\{u:\gamma_s u\sim \gamma\}} a^{G_{\gamma_s}}(S,u)
\]
where  $u$ runs over $G_{\gamma_s}(\Q_S)$-unipotent orbits in $G_{\gamma_s}(\Q)$ such that $\gamma_s u$ are $(G,S)$-equivalent to $\gamma$ and we set
\[
\varepsilon^G(\gamma_s)=\begin{cases}  1 & \text{if $\gamma_s$ is $\Q$-elliptic in $G$,} \\ 0 & \text{otherwise}.\end{cases}
\]
Especially, if $\gamma$ is semisimple, then we have
\begin{equation}\label{ssgc}
a^G(S,\gamma)=\varepsilon^G(\gamma)\, |\iota^G(\gamma)|^{-1} \,\vol(G_\gamma(\Q)\bsl G_\gamma(\A)^1).
\end{equation}
Hence, from this and Lemma \ref{id2}, one finds that a needed estimation for $I_4(f)$ is obviously reduced to some known results (cf. Proof of Proposition \ref{weight-est}).
Note that we carefully see the growth of $a^G(S,\gamma)$ with respect to $S$ if $\gamma$ is not semisimple.
We also note that
\[
a^G(S,z\gamma)=a^G(S,\gamma)
\]
holds for any $z$ in $Z_G(\Q)$.

The following notations are necessary to describe $a^G(S,\gamma)$ explicitly.
Let $E$ be an algebraic number field and let $\chi=\prod_w\chi_w$ be a character on $E^\times\bsl \A_E^1\cong E^\times\R_{>0}\bsl \A_E^\times$.
We set
\[
S_E=\bigsqcup_{v\in S}\{ w\mid \text{$w$ is a place of $E$ such that $w$ divides $v$}\},
\]
\[
L_E^S(s,\chi)=\prod_{w\not\in S_E}L_{E,w}(s,\chi_w),
\]
\[
L_{E,w}(s,\chi_w)=\begin{cases} (1-\chi_w(\pi_w)\, q_w^{-s})^{-1} & \text{if $\chi_w$ is unramified,} \\ 1 & \text{if $\chi_w$ is ramified}, \end{cases}
\]
where $\pi_w$ is a prime element of $E$ and $q_w$ denotes the cardinality of the residue field of $E_w$.
For the trivial representation $\mathbf{1}_E$ on $E^\times\bsl \A_E^1$, we set
\[
\zeta_E^S(s)=L^S(s,\mathbf{1}_E).
\]
We write $c_E$ for the residue of $\zeta^{\Sigma_\inf}_E(s)$ where $\Sigma_\inf=\{w|\inf\}$.
We denote $\mathfrak{c}_E(S)$ by the constant term of the Laurent expansion of $\zeta_E^S(s)$ at $s=1$, that is,
\[
\zeta_E^S(s)=\frac{c_{E,S}^{-1}c_E}{s-1}+\mathfrak{c}_E(S)+*(s-1)+\cdots
\]
where we set $c_{E,S}=\prod_{w\in S_E-\Sigma_\inf}(1-q_w^{-1})^{-1}$.
If $E=\Q$, then we set
\[
L^S(s,\chi)=L^S_\Q(s,\chi), \quad \zeta^S(s)=\zeta^S_\Q(s), \quad  \mathfrak{c}(S)=\mathfrak{c}_\Q(S)
\]
for simplicity.
We will later use the following estimates.
\begin{lem}\label{estquad}
Let $m\in \Bbb N$ be fixed.
For any positive real number $\varepsilon$, there exists a positive constant $c(\varepsilon,m,E)$ such that
\[
\sum_{\chi}|\, L^S_E(1,\chi) \, |^m < c(\varepsilon,m,E) \times \prod_{p\in S_0} p^\varepsilon
\]
where $\chi=\prod_w \chi_w$ runs over all non-trivial quadratic characters on $E^\times\bsl \A_E^1$ such that $\chi_w$ is unramified for any $w\not\in S_E$.
\end{lem}
\begin{proof} Let $N(\chi)$ be the norm of the conductor of $\chi$. Then
by Lemma 1.4 of \cite{wang}, 
$$N(\chi)\leq 2^{3 n_E} \prod_{\frak p_w\nmid 2\atop w\in S_E-\Sigma_\infty} N(\frak p_w),
$$
where $n_E=[E:\Bbb Q]$.
Now $N(\frak p_w)\leq p^{n_E}$. Hence $N(\chi)\leq 2^{2n_E}\prod_{p\in S_0} p^{n_E}$.
By \cite{Li}, $L_E(1,\chi)\ll_{E,\epsilon'} exp\left(C \frac {\log N(\chi)}{\log\log N(\chi)}\right)\ll N(\chi)^{\epsilon'}$ for some constant $\epsilon'$. Here
$$L_E^S(1,\chi)=L_E(1,\chi) \prod_{w\in S_E-\Sigma_\infty} L_{E,w}(1,\chi_w)^{-1}.
$$
Hence 
$$\left|\prod_{w\in S_E-\Sigma_\infty} L_{E,w}(1,\chi_w)^{-1}\right|\leq \left|\prod_{w\in S_E-\Sigma_\infty} \left(1+\frak p_w^{-1}\right)\right|\ll_E \prod_{p\in S_0} (1+p^{-1})^{n_E}.
$$
Now for $M=\prod_{p\in S_0} p$, 
$$\log \prod_{p\in S_0} (1+p^{-1})=\sum_{p\in S_0} \log(1+p^{-1})\ll \sum_{p\in S_0} \frac 1p\ll \sum_{p\leq M} \frac 1p\ll \log\log M.
$$
Since $N(\chi)\ll_E M^{n_E}$,
$$|L_E^S(1,\chi)|\ll_E exp\left(C \frac {\log M}{\log\log M}\right) \log M\ll_{E,\epsilon'} M^{\epsilon'}\log M.
$$
Hence for each $m\in \Bbb N$,
$$\sum_{\chi} |L_E^S(1,\chi)|^m\ll M^{\epsilon' m}(\log M)^m\sum_{d|M} 1=M^{\epsilon' m} \phi(M)(\log M)^m \ll_{E,m,\epsilon} M^{\epsilon}.
$$
\end{proof}

For the $G(\Q_S)$-orbit of $u_\min$ and the chosen measures \eqref{min1} and \eqref{min2}, we have
\begin{equation}\label{gcm}
a^G(S,u_\min)=2^{-1}\, \vol(M_2(\Q)\bsl M_2(\A)^1)\, \zeta^S(2)
\end{equation}
by \cite[Theorem 6.1]{HW}.
Next, we explain $a^{G_{\gamma_s}}(S,u)$ for non-semisimple and non-unipotent elements $\gamma=\gamma_s u$, which we call a mixed element.
For all mixed elements of $G$, their global coefficients were studied in \cite[Section 3]{HW}.
We mention only cases required for our estimation.
Furthermore, we will choose the same measures as in \cite[Section 3.4]{HW} on the unipotent orbits over $\Q_S$.
However, we do not explain details for normalizations of measures, because they are unnecessary for estimations.
Namely, the following equalities for the global coefficients hold under suitable normalizations of measures.
By a classification of mixed elements in \cite[Section 5.3]{HW} it is enough to consider the following groups (as the centralizers of semisimple elements),
\[
G_1=\{ (g_1,g_2)\in GL_2\times GL_2\mid \det(g_1)=\det(g_2)  \}(\cong G_{\delta_1}),
\]
\[
G_2=\{g\in R_{E/\Q}(GL_2)\mid \det(g)\in GL_1\},
\]
\[
G_3=\{ (x,g)\in R_{E/\Q}(GL_1)\times GL_2 \mid N_{E/\Q}(x)=\det(g) \}
\]
where $E$ is a quadratic extension of $\Q$ and $R_{E/\Q}$ means the restriction of scalars.
A unipotent element in $G_1(\Q)$ can be written as
\[
u_1(x,y)=(\begin{pmatrix}1&x \\ 0&1\end{pmatrix},\; \begin{pmatrix}1&y \\ 0&1\end{pmatrix}).
\]
For the group $G_1$, representative elements of nontrivial unipotent orbits over $\Q_S$ are as follows:
\[
u_1(1,0) ,\quad  u_1(0,1),\quad  u_1(\alpha,1) \;\; (\alpha\in \Q^\times/((\Q^\times_S)^2\cap\Q^\times)).
\]
For $\alpha$ in $\Q^\times/((\Q^\times_S)^2\cap\Q^\times)$, there exist constants $c_\min(u_1)$ and $c_\reg(u_1)$ (which do not depend on $S$) such that
\begin{align}
& a^{G_1}(S,u_1(1,0))=a^{G_1}(S,u_1(0,1))=c_\min(u_1)\times \mathfrak{c}(S), \nonumber \\
& a^{G_1}(S,u_1(\alpha,1))=  c_\reg(u_1)\times \big\{ \mathfrak{c}(S)^2 + \sum_{\chi}\chi_S(\alpha) \, L^S(1,\chi)^2 \big\} , \label{gc1}
\end{align}
where $\chi=\prod_v\chi_v$ runs over all nontrivial quadratic characters such that $\chi_v$ is unramified for any $v\not\in S$, and we set $\chi_S=\prod_{v\in S}\chi_v$ (see \cite[Example 3.9]{HW}).
Representative elements of non-trivial unipotent $G_2(\Q_S)$-orbits in $G_2(\Q)$ are
\[
u_2(\alpha)=\begin{pmatrix}1&\alpha \\ 0&1 \end{pmatrix} \quad (\alpha\in E^\times/(((E_{S_E}^\times)^2\Q_S^\times)\cap E^\times)).
\]
Then, there exists a constant $c(u_2)$ (which does not depend on $S$) such that
\begin{equation}\label{gc2}
a^{G_2}(S,u_2(\alpha))=c(u_2) \times \big\{ \mathfrak{c}_E(S)+\sum_\chi \chi_{S_E}(\alpha)\, L_E^S(1,\chi) \big\}
\end{equation}
where $\chi=\prod_w\chi_w$ runs over all nontrivial quadratic characters such that $\chi|_{\A_\Q^1}=1$ and $\chi_w$ is unramified for any $w\not\in S_E$ (see \cite[Example 3.8]{HW}).
Representative elements of non-trivial unipotent $G_3(\Q_S)$-orbits in $G_3(\Q)$ are
\[
u_3(\alpha)=(1,\begin{pmatrix}1&\alpha \\ 0&1 \end{pmatrix}) \quad (\alpha\in \Q^\times/(N_{E/\Q}(E^\times_{S_E})\cap \Q^\times)).
\]
where $N_{E/\Q}$ is the norm of $E/\Q$.
Then, there exists a constant $c(u_3)$ (which does not depend on $S$) such that
\begin{equation}\label{gc3}
a^{G_3}(S,u_3(\alpha))= c(u_3) \times \big\{ \mathfrak{c}(S)+\chi_S(\alpha)\, L^S(1,\chi) \big\}
\end{equation}
where $\chi$ is the nontrivial quadratic character on $\Q^\times\bsl \A_\Q^1$ corresponding to $E$ via the class field theory (see \cite[Section 3]{HW}).

\subsection{Explicit calculations for $I_2(f)$}

Recall $J_G(zu_\min,f_\xi)$ and $J_G(zu_\min,h)$ defined as \eqref{min1} and \eqref{min2}.
\begin{lem}\label{I21}
Let $c(zu_\min)$ denote the constant given in Lemma \ref{id4}.
Then, we have
\[
I_2(f)=\sum_{z\in Z(\Q)} 2^{-1}\, \vol(M_2(\Q)\bsl M_2(\A)^1)\, \zeta^S(2)\times c(zu_\min) \times \chi_\xi(z)^{-1}\times(l_1-l_2)(l_1+l_2)\times J_G(zu_\min,h).
\]
In particular, $\zeta^S(2)$ is bounded by a positive constant for any $S$.
\end{lem}
\begin{proof}
This lemma follows from Lemma \ref{id4} and \eqref{gcm}.
\end{proof}
\begin{lem}\label{I22}
Assume that $h_{a_1,a_2,a_3}$ is the characteristic function of the open compact set $\bK_p\diag(p^{-a_1},p^{-a_2},p^{a_1-a_3},p^{a_2-a_3})\bK_p$ on $G(\Q_p)$ $(a_3\geq a_1\geq a_2\geq 0)$.
If $a_3$ is odd, then we have $J_G(zu_\min,h_{a_1,a_2,a_3})=0$.
If $a_3$ is even, then we may set $a_3=2m$ and assume $m\geq a_1\geq a_2$ by the action of $W_0^G$, and we get
\[
J_G(zu_\min,h_{a_1,a_2,a_3})=\begin{cases}
(1-p^{-2})^{-1} & \text{if $a_3=2m$, $a_1=a_2=m$, and $|z|_p=p^{m}$,} \\
p^{a_3-2a_2} & \text{if $a_3=2m$, $m=a_1>a_2$, and $|z|_p=p^{m}$,} \\
0 & \text{otherwise.}
\end{cases}
\]
If $h_p$ is the characteristic function of $\{x\in \bK_p \mid x\equiv E_4 \mod p^l \Z_p  \}$ on $G(\Q_p)$, then
\[
J_G(zu_\min, h_p)=\begin{cases}  p^{-2l}(1-p^{-2})^{-1} & \text{if $z\equiv 1 \mod p^l\Z_p$}, \\ 0 & \text{otherwise.} \end{cases}
\]
\end{lem}
\begin{proof}
The first assertion is stated in \cite[Theorem 2.4.1]{Assem}.
The second assertion is trivial.
He assumed that the residual characteristic is not $2$ in the paper.
However, since \cite[Lemma 2.1.1]{Assem} can be applied for $\Q_2$, one can easily compute the above integral.
\end{proof}

\subsection{Explicit calculations for $I_3(f)$}

\begin{lem}\label{I31}
Let $c(z\delta_1)$ denote the constant given in Lemma \ref{id3}.
Then, we have
\[
I_3(f)=\sum_{z\in Z(\Q)}  \vol(G_{\delta_1}(\Q)\bsl G_{\delta_1}(\A)^1)\times c(z\delta_1) \times \chi_\xi(z)^{-1}\times(l_1-l_2)(l_1+l_2)\times J_G(z\delta_1,h).
\]
\end{lem}
\begin{proof}
This lemma follows from Lemma \ref{id3} and \eqref{ssgc}.
\end{proof}
\begin{lem}\label{I32}
Assume that $J_G(z\delta_1,h)$ is defined as \eqref{delta1} and $h_{a_1,a_2,a_3}$ is the characteristic function of $\bK_p\diag(p^{-a_1},p^{-a_2},p^{a_1-a_3},p^{a_2-a_3})\bK_p$ on $G(\Q_p)$ $(a_3\geq a_1\geq a_2\geq 0)$.
If $a_3$ is odd, then we have $J_G(z\delta_1,h_{a_1,a_2,a_3})=0$.
If $a_3$ is even, then we may set $a_3=2m$ and assume $m\geq a_1\geq a_2$ by the action of $W_0^G$, and we get
\[
J_G(z\delta_1,h_{a_1,a_2,a_3})=\begin{cases}
1 & \text{if $a_3=2m$, $a_1=a_2=m$, and $|z|_p=p^{m}$,} \\
p^{a_3-a_1-a_2}(1-p^{-2}) & \text{if $a_3=2m$, $m>a_1=a_2$, and $|z|_p=p^{m}$,} \\
0 & \text{otherwise.}
\end{cases}
\]
\end{lem}
\begin{proof}
This lemma can be proved by \cite[Lemma 2.1.1]{Assem}.
%Note that \cite[Lemma 2.1.1]{Assem} can be applied for $\Q_2$.
\end{proof}

\subsection{Estimations for $I_5(f)$}

From now on, we assume that $\gamma\in G(\Q)$ is not semisimple and $\gamma\not\in Z_G(\Q)\{ u_\min\}_G$.
By Lemmas \ref{id2}, we may also assume that the $\gamma_s$ is $\R$-elliptic in $G$.
\begin{lem}\label{ei51}
If the centralizer of $\gamma$ is not isomorphic to $G_1$, $G_2$, and $G_3$, then we get $J_G(\gamma,f_\xi)=0$.
\end{lem}
\begin{proof}
For unipotent elements $\gamma$, it was proved in Lemmas \ref{id5} and \ref{id6}.
All mixed elements in $G(\Q)$ are classified in \cite[Section 5.3]{HW}.
Using the classification and the limit formula for $SL_2(\R)$ (cf. \cite{Knapp}) one can show $J_G(\gamma,f_\xi)=0$.
\end{proof}

\begin{lem}\label{ei52}
Let $z\in Z_G(\Q)$.
Then, we get
\[
J_G(z\delta_1 u_1(1,0),f_\xi)=J_G(z\delta_1 u_1(0,1),f_\xi)=0.
\]
There exists a constant $c(z\delta_1 u_1)$ such that
\[
J_G(z\delta_1u_1(1,1),f_\xi)=-J_G(z\delta_1u_1(1,-1),f_\xi)=c(z\delta_1 u_1)\times \chi_\xi(z)^{-1}\times\{ (-1)^{l_2}-(-1)^{l_1}  \}.
\]
\end{lem}
\begin{proof}
This can be proved by the limit formula for $SL_2(\R)$ (cf. \cite{Knapp}).
\end{proof}
\begin{lem}\label{ei53}
Let $z\in Z_G(\Q)$.
The contribution of $z\delta_1u_1(1,0)$ and $z\delta_1u_1(0,1)$ to $I_5(f)$ is zero.
For any positive real number $\varepsilon$, there exists a constant $c(z\delta_1u_1,\varepsilon)>0$ such that the contribution of $(G,S)$-equivalence classes of elements $z\delta_1u_1(1,\alpha)$ $(\alpha\in\Q^\times/(\Q^\times_S)^2\cap\Q^\times)$ is bounded by
\[
c(z\delta_1u_1,\varepsilon) \times \chi_\xi(z)^{-1}\times J_{M_0}^{M_0}(z\delta_1,|h_{P_0}|)\times \prod_{p\in S_0}p^\varepsilon \times \prod_{p\in S_0}(1-p^{-1})^{-2} .
\]
\end{lem}
\begin{proof}
The first assertion obviously follows from Lemma \ref{ei52}.
By \eqref{gc1} and Lemma \ref{ei52}, the contribution equals
\begin{align*}
&\sum_{\alpha\in\Q^\times/(\Q^\times_S)^2\cap\Q^\times}a^G(S,z\delta_1(1,\alpha)) \, J_G(z\delta_1u_1(1,\alpha),f_\xi) \, J_G(z\delta_1u_1(1,\alpha),h) \\
&= 2c(z\delta_1 u_1)\{ (-1)^{l_2}-(-1)^{l_1}  \} \sum_\chi L^S(1,\chi)^2 \sum_{\alpha\in\Q^\times/(\Q^\times_{S_0})^2\cap\Q^\times} \chi_{S_0}(\alpha)\, J_G(z\delta_1u_1(1,\alpha),h)
\end{align*}
where $\chi=\prod_v \chi_v$ runs over all quadratic characters such that $\chi_\inf=\sgn$ and $\chi_v$ is unramified for any $v\not\in S$.
By calculating $G_{z\delta_1u_1(1,\alpha)}$, we find that $G_{z\delta_1u_1(1,\alpha)}$ is contained in $Z_GN_0$ and we may set
\[
x=\begin{pmatrix}1&0&0&0 \\ 0&a^{-1}&0&0 \\ 0&0& b &0 \\ 0&0&0&ab  \end{pmatrix} \begin{pmatrix}1&0&0&y \\ 0&1&y&0 \\ 0&0&1&0 \\ 0&0&0&1 \end{pmatrix} \begin{pmatrix}1&r&0&0 \\ 0&1&0&0 \\ 0&0&1&0 \\ 0&0&-r&1 \end{pmatrix}
\]
for each element $x$ in $G_{z\delta_1u_1(1,\alpha)}(\Q_{S_0})\bsl P_0(\Q_{S_0})$.
Then, we get
\[
x^{-1} z\delta_1u_1(1,\alpha)x=z\delta_1
\begin{pmatrix}
\begin{array}{cc} 1&2r \\ 0&1 \end{array} & \begin{pmatrix}1&r \\ 0&1 \end{pmatrix}\begin{pmatrix}b& 2y \\ 2y & \alpha a^2b \end{pmatrix}\begin{pmatrix}1&0 \\ -r &1 \end{pmatrix} \\
\begin{array}{cc}0&0 \\ 0&0 \end{array} & \begin{array}{cc} 1&0 \\ -2r&1 \end{array}
\end{pmatrix}.
\]
From this, we deduce
\begin{align*}
& \Big| \sum_{\alpha\in\Q^\times/(\Q^\times_{S_0})^2\cap\Q^\times} \chi_{S_0}(\alpha)\, J_G(z\delta_1u_1(1,\alpha),h) \Big| \leq \sum_{\alpha\in\Q^\times/(\Q^\times_{S_0})^2\cap\Q^\times} J_G(z\delta_1u_1(1,\alpha),|h|) \\
& = 4\times \prod_{p\in S_0}(1-p^{-1})^{-2}\times \int_{\bK_{S_0}}  \int_{N_0(\Q_{S_0})} \big| \, h(k^{-1}z\delta_1n \, k)\, \big| \, \d n\,  \d k \\
& = 4\times \prod_{p\in S_0}(1-p^{-1})^{-2}\times J_{M_0}^{M_0}(z\delta_1,|h_{P_0}|).
\end{align*}
Hence, the inequality follows from this estimation and Lemma \ref{estquad}.
\end{proof}

\begin{lem}\label{ei54}
Let $\delta_2$ be a semisimple element in $G(\Q)$ such that $G_{\delta_2}\cong G_2$ for a quadratic extension $E/\Q$.
Each unipotent element $u_2(\alpha)$ in $G_2(\Q)$ is identified with an element in $G_{\delta_2}(\Q)\subset G(\Q)$.
The $G(\Q)$-conjugacy class of $\delta_2 u_2(\alpha)$ has an intersection with $P_1(\Q)$ and we may assume that $\delta_2$ belongs to $M_1(\Q)$ as a representative element of the conjugacy class.
For any positive real number $\varepsilon$, there exists a constant $c(\delta_2u_2,\varepsilon)>0$ such that the contribution of $(G,S)$-equivalence classes of elements $\delta_2u_2(\alpha)$ $(\alpha\in E^\times/((E^\times_{S_E})^2\Q_S^\times)\cap E^\times)$ is bounded by
\[
c(\delta_2u_2,\varepsilon) \times \chi_\xi(\delta_2)^{-1}\times J_{M_1}^{M_1}(\delta_2,|h_{P_1}|)\times \prod_{p\in S_0}p^\varepsilon \times \prod_{p\in S_0}(1-p^{-1})^{-2}.
\]
\end{lem}
\begin{proof}
By \cite[Section 5.3]{HW}, we may choose the semisimple element $\delta_2$ as
\[
\delta_2=z\begin{pmatrix}h_\beta&O_2 \\ O_2& {}^t\!h_\beta \end{pmatrix},\quad h_\beta=\begin{pmatrix}0&1 \\ \beta &0 \end{pmatrix}
\]
for an element $\beta$ in $\Q^\times -(\Q^\times)^2$ and an element $z$ in $Z_G(\Q)$.
In particular, we have $E=\Q(\sqrt{\beta})$.
If $\beta$ is negative, then $\nu(\delta_2)$ is also negative.
Hence, by Lemma \ref{clv}, the contribution vanishes if $\beta$ is negative.
We may assume that $\beta$ is positive, i.e., $E$ is a real quadratic field.
Then, there exists an element $z'$ in $Z_G(\R)$ such that $\delta_2$ is $G(\R)$-conjugate to $z'\delta_1$.
Thus, it follows from \eqref{gc2} and Lemma \ref{ei52} that the contribution equals
\begin{align*}
&\sum_{\alpha\in E^\times/((E^\times_{S_E})^2\Q_S^\times)\cap E^\times)}a^G(S,\delta_2 u_2(\alpha)) \, J_G(\delta_2u_2(\alpha),f_\xi) \, J_G(\delta_2u_2(\alpha),h) \\
&= 2c(z\delta_1 u_1)\{ (-1)^{l_2}-(-1)^{l_1}  \} \sum_\chi L_E^S(1,\chi) \sum_{\alpha\in E^\times/((E^\times_{(S_0)_E})^2\Q_{S_0}^\times)\cap E^\times)} \chi_{(S_0)_E}(\alpha)\, J_G(\delta_2u_2(\alpha),h)
\end{align*}
where $\chi=\prod_w\chi_w$ runs over all nontrivial quadratic characters such that $\chi|_{\A_\Q^1}=1$, $\chi_w=\sgn$ for any $w|\inf$, and $\chi_w$ is unramified for any $w\not\in S_E$.
Therefore, we finish the proof by using Lemma \ref{estquad} and an argument similar to the proof of Lemma \ref{ei53}.
\end{proof}

\begin{lem}\label{ei55}
Let $\delta_3$ be a semisimple element in $G(\Q)$ such that $G_{\delta_3}\cong G_3$ for a quadratic extension $E/\Q$.
Each unipotent element $u_3(\alpha)$ in $G_3(\Q)$ is identified with an element in $G_{\delta_3}(\Q)\subset G(\Q)$.
The $G(\Q)$-conjugacy class of $\delta_3 u_3(\alpha)$ has an intersection with $P_2(\Q)$ and we may assume that $\delta_3$ belongs to $M_2(\Q)$ as a representative element of the conjugacy class.
There exists a positive constant $c(\delta_3u_3)$ such that the contribution of $(G,S)$-equivalence classes of elements $\delta_3u_3(\alpha)$ $(\alpha\in \Q^\times/(N_{E/\Q}(E^\times_{S_E})\cap \Q^\times)$ is bounded by
\[
c(\delta_3u_3)\times \chi_\xi(\delta_3)^{-1}\times J_{M_2}^{M_2}(\delta_3,|h_{P_2}|)\times \prod_{p\in S_0}(1-p^{-1})^{-1}.
\]
\end{lem}
\begin{proof}
By \cite[Section 5.3]{HW}, the semisimple element $\delta_3$ can be written as
\[
\delta_3=z\begin{pmatrix}1&0&0&0 \\ 0&x&0&y \\ 0&0&1&0 \\ 0&\beta y & 0 & x \end{pmatrix}
\]
for an element $z$ in $Z_G(\Q)$, an element $\beta$ in $\Q^\times -(\Q^\times)^2$, and elements $x$, $y$ in $\Q$ such that $x^2-\beta y^2=1$.
Note that $E=\Q(\sqrt{\beta})$.
If $\beta$ is positive, then $\delta_3$ is not $\R$-elliptic in $G$.
Hence, we have $J_G(\delta_3u_3(\alpha))=0$ by Lemma \ref{id2}.
So, we assume that $\beta$ is negative, i.e., $E$ is an imaginary quadratic field.
Since $u_3(\alpha)$ is $G(\R)$-conjugate to $u_3(1)$ or $u_3(-1)$, using the limit formula for $SL_2(\R)$, we have
\[
J_G(\delta_3u_3(1),f_\xi)=-J_G(\delta_3u_3(-1),f_\xi)=c_3'\times ( -e^{il_2\theta}+e^{il_1\theta}  )
\]
for a positive constant $c_3'$.
Hence, by \eqref{gc3}, the contribution is equal to
\begin{align*}
&\sum_{\alpha\in \Q^\times/(N_{E/\Q}(E^\times_{S_E})\cap \Q^\times)}a^G(S,\delta_3 u_3(\alpha)) \, J_G(\delta_3u_3(\alpha),f_\xi) \, J_G(\delta_3u_3(\alpha),h) \\
&= 2c_3'\times ( -e^{il_2\theta}+e^{il_1\theta}  )  L^S(1,\chi) \sum_{\alpha\in \Q^\times/(N_{E/\Q}(E^\times_{(S_0)_E})\cap \Q^\times)} \chi_{S_0}(\alpha)\, J_G(\delta_3u_3(\alpha),h)
\end{align*}
where $\chi$ denotes the nontrivial quadratic character on $\Q^\times\bsl\A^1$ corresponding to $E$.
We can derive this lemma from this equality and an argument similar to the proof of Lemma \ref{ei53}, 
\end{proof}

\section{An estimation of the geometric side}\label{sec6}
Let us recall our setting. 
Let $S'$ be a (non-empty) finite set of finite primes and $S_0$ be a finite set of finite primes containing $S'$. 
We choose $S_0$ sufficiently large so that Arthur's geometric expansion works. Put $S=S_0\cup\{\infty\}$. 
For a positive integer $N$ whose prime divisors do not belong to $S'$, put $f_{K(N)}={\rm char}_{K(N)}$. 
When we study the level aspect, we always choose the level $N$ for the fixed $S'$ as above. 
We denote by $\underline{k}=(k_1,k_2)$ the highest weight of $\xi=\xi_{\underline{k}}$ for $k_1\ge k_2\ge 3$ as in Section \ref{alg-rep}. 
Let us put 
$$f_{S',\alpha}:=[G(\Z_{S'})\alpha G(\Z_{S'})]\in C^\infty_c(G(\Q_{S'}))$$
for $\alpha\in T(\Q)$. Note that usually we would choose $\alpha$ from $T(\Q_{S'})$, but due to 
the comparison with spectral side, intentionally we choose $\Q$-rational elements and clearly this never changes anything. 
If $\nu(\alpha)\in (\Q^\times)^2$, then  
such an $\alpha$  can be uniquely written as 
$$\alpha=(z_\alpha \cdot E_4) k_\alpha,\ z_\alpha\cdot E_4\in Z^+_G(\Q),\ k_\alpha\in (T\cap Sp_4)(\Q).$$
Put 
\begin{equation}\label{zalpha}
z'_\alpha=\left\{
\begin{array}{cc}
z_\alpha\cdot E_4  &\ {\rm if}\ \nu(\alpha)\in (\Q^\times)^2\\
E_4  &\ {\rm otherwise}
\end{array}\right..
\end{equation}
To prove Theorem \ref{main} we have only to check it for 
$$f=f_{\xi}h,\quad h=f_{z'_\alpha K(N)}f_{S',\alpha}\Big(\bigotimes_{p\in S\setminus (S'\cup\{v|N\infty\})}{\rm char}_{ K_p}\Big)
\in C^\infty_c(G(\Q_{S_0}))$$
where $f_{z'_\alpha K(N)}$ stands for the characteristic function of $z'_\alpha K(N)$.   

%We now estimate each term of the geometric side for such a $f$. 
%In this section the set $S_0$ here plays a role in ``$S$" of Section \ref{s3}. 
%Note that Arthur's $S$ is different from ``$S$" in Section \ref{s3}.  

Since $I_7(f)=0$ (see Lemma \ref{id1}), it is unnecessary to consider it.
Fix a test function $f$ as above.
Let $\gamma$ be a $(M,S)$-equivalence class  whose contribution to $I_\mathrm{geom}(f)$ is not zero.
%We know that finitely many $\mathcal{O}$-equivalence classes contribute to the geometric side $I_\mathrm{geom}(f)$ and its number does not change even if $S$ is moved.
By \eqref{n-gamma} we see that 
\[
\nu(\alpha)=\nu(\gamma).
\]
Furthermore its semisimple part $\gamma_s$ is $\R$-elliptic in $M$ (cf. Lemmas \ref{id1} and \ref{id2}).
If we set $\gamma_s=z_{\gamma}\gamma_1$ where $z_\gamma\in  \R^\times\simeq  A_M(\R)$ and $\gamma_1\in M(\R)$ with $\nu(\gamma_s)\in\{\pm1 \}$. 
%Then the absolute value of any eigenvalue for $\gamma_1$ is 1.
Therefore we have
\[
\chi_\xi(z_\gamma)=z^{k_1+k_2}_\gamma=\sgn(z_\gamma)|\nu(z_\gamma)|^{\frac{k_1+k_2}{2}}_\R=\sgn(z_\gamma)|z_\gamma|^{k_1+k_2}_\R=
\sgn(z_\gamma)|\nu(\alpha)|^{-\frac{k_1+k_2}{2}}_{S'}.
\]
This observation will be implicitly used in the proof of Proposition \ref{semisimple-est} below. 

Each orbital integral $J_M^M(\gamma,h_P)$ is bounded by that of $PGSp_4(\Q_S)$ using the projection $G\to PGSp_4$.
Therefore, we can use the same arguments as in \cite{Shin} and \cite{ST} to estimate it for any semisimple element $\gamma$.
\begin{prop}\label{semisimple-est}
Fix a finite set $S'$ and a function $f_{S',\alpha}=[G(\Z_{S'})\alpha G(\Z_{S'})]\in C^\infty_c(G(\Q_{S'}))$.
Then, we have
\[
I_2(f)\times \vol(K(N))^{-1}\times |\nu(\alpha)|^{-\frac{k_1+k_2}{2}}_{S'} =O((k_1-k_2+1)(k_1+k_2-3)\varphi(N) N^8),
\]
\[
I_3(f)\times \vol(K(N))^{-1}\times |\nu(\alpha)|^{-\frac{k_1+k_2}{2}}_{S'} =O((k_1-1)(k_2-2)),
\]
\[
\{I_4(f)+I_5(f)\}\times \vol(K(N))^{-1}\times |\nu(\alpha)|^{-\frac{k_1+k_2}{2}}_{S'} =O(k_1+k_2-3),
\]
\[
 I_6(f)\times \vol(K(N))^{-1}\times |\nu(\alpha)|^{-\frac{k_1+k_2}{2}}_{S'} =O((k_1+k_2-3)\varphi(N) N^7)
\]
for any weight $(k_1,k_2)$ $(k_1\geq k_2\geq 3)$ and any level $N>0$ prime to $\prod_{p\in S'}p$. Note that $\vol(K(N))^{-1}=[\G(1):\G(N)]=N^{10}\prod_{p|N} (1-p^{-2})(1-p^{-4})$.
\end{prop}
\begin{proof}
If $zu_\min$ contributes to $I_\mathrm{geom}(f)$, then we have $z=z_\alpha$.
Hence, only $z_\alpha u_\min$ contributes to $I_2(f)$ and the estimation for $I_2(f)$ obviously follows from Lemmas \ref{I21} and \ref{I22}. 

By Lemmas \ref{id2}, \ref{id5}, \ref{id6} and \cite[Lemma 5]{Clozel} (also see \cite[Lemma 4.3 and line 12 in p.100]{Shin}), there exists a sufficiently large natural number $\tilde N_0$ such that, if $N>\tilde N_0$, then we have $I_3(f)=I_4(f)=I_5(f)=0$ and only the central elements contribute to $I_6(f)$.
When $N$ moves between $1$ and $\tilde N_0$, there are finitely many $M(\Q)$-conjugacy classes $(M\in \cL)$ which contribute to $I_3(f)$, $I_4(f)$ and $I_6(f)$, and finitely many $(G,S)$-equivalence classes which contribute to $I_5(f)$ (cf. \cite[Proof of Theorem 4.11]{Shin}).
Thus, we get the estimation for $I_3(f)$ by Lemma \ref{I31} and the estimation for $I_4(f)+I_5(f)$ by Lemmas \ref{id2}, \ref{id5}, \ref{id6}.
By these facts and Lemma \ref{id1}, the remaining work is to find the growth of $I_6(f)$ with respect to $N$.
For each proper Levi subgroup $M$ in $\cL$ and each element $z\in Z_G(\Q)$, we have
\begin{equation}  \label{ga1}
|J^G_M(z,h_P)|=\mathrm{(constant)}\times \int_{\prod_{v|N} N_P(\Q_v)}f_{z'_\alpha K(N)}(n)dn\leq \mathrm{(constant)}\times N^{-3}.
\end{equation}
since $K(N)$ is a normal subgroup in $\prod_{v| N}\bK_v$.
Hence, this proposition is proved.
%Next we compute a bound of $I_4(f)$. 
%As remarked before the pair $(M,\gamma)$ contributing $I_4(f)$ is finite and independent of the level $N$. 
%Furthermore $\gamma$ should be $\R$-elliptic in $M(\R)$. 
%With this and the character formula in Section \ref{s2} we have 
%$$|I^G_M(\gamma,f_\xi)|\le 2^{\frac{\dim(G)-\dim(M)}{2}}\cdot \frac{4}{\Delta(\gamma)}\cdot | z_\gamma|^{k_1+k_2}_\R
%=C_4\cdot |\nu(\alpha)|^{-\frac{k_1+k_2}{2}}_{S'}$$
%where we put $C_4=2^{\frac{\dim(G)-\dim(M)}{2}}\cdot \ds\frac{4}{\Delta(\gamma)}$.
%Since now $\gamma$ is semisimple, the global coefficient $a^G_M(\gamma,S)$ is 
%independent of the weight and the level (hence Arthur's $S$). 
%Therefore we may put $a^G_M(\gamma):=a^G_M(\gamma,S)$.   
%What remains is to estimate $J^G_M(\gamma,h_P)$. 
%If $N$ is sufficiently large, as explained at line 12 in p.100 of \cite{Shin} 
%the orbital integral $J^G_M(\gamma,h_P)=0$ unless $\gamma\in Z_G(\Q)$. Then we have 
%\begin{equation}
%\begin{split}
%|J^G_M(\gamma,h_P)|&=\Big|\int_{K^{S_0,\infty}}\int_{N_P(\A^{S_0,\infty})}h(k^{-1}nk)dndk\Big| \\
%&\le \mu^{S_0,\infty}(K^{S_0,\infty})\mu_{N_P(\A^{S_0,\infty})}(K(N)\cap N_P(\A^{S_0,\infty})) \\
%&=\mu_{N_P(\A^{S_0,\infty})}(K(N)\cap N_P(\A^{S_0,\infty}))\le N^{-3}
%\end{split}
%\end{equation}
%since now $M\in \{M_1, \; \; s_1M_1, \;\; M_2,  \;\; s_0M_2  \}$. 
%Summing up  we have 
%\begin{equation}
%\frac{|\nu(\alpha)|^{\frac{k_1+k_2}{2}}_{S_0}I_4(f)}{\overline{\mu}(G(\Q)A_{G,\infty}\bs G(\A))\cdot\dim \xi}
%=O(N^{-3})\ {\rm as}\ N+k_1+k_2\lra \infty.  
%\end{equation}
\end{proof}

We set
\[
p_{S'}=\prod_{p\in S'}p , \quad H^\mathrm{ur}(G(\Q_{S'}))^\kappa=\bigoplus_{p\in S'}H^\mathrm{ur}(G(\Q_p))^\kappa.
\]
\begin{prop}\label{level-est}
{\bf (Level aspect)}
There exist positive constants $a$, $b$, and $N_0$ such that
\[
I_2(f)\times \vol(K(N))^{-1}\times |\nu(\alpha)|^{-\frac{k_1+k_2}{2}}_{S'} =O(p_{S'}^\kappa(k_1-k_2+1)(k_1+k_2-3)\varphi(N) N^8)
\]
\[
I_6(f)\times \vol(K(N))^{-1}\times |\nu(\alpha)|^{-\frac{k_1+k_2}{2}}_{S'} =O(p_{S'}^{a\kappa +b}(k_1+k_2-3)\varphi(N) N^7)
\]
\[
I_3(f)=I_4(f)=I_5(f)=I_7(f)=0.
\]
for any $(k_1,k_2)$, $N>0$, $\kappa\geq 1$, $S'$, and $f_{S',\alpha}$, which satisfy the conditions $k_1\geq k_2\geq 3$, $f_{S',\alpha}\in H^\mathrm{ur}(G(\Q_{S'}))^\kappa$, $N$ is prime to $\prod_{p\in S'}p$, and $N\geq N_0 \prod_{p\in S'} p^{10\kappa}$.
\end{prop}
\begin{proof}
According to \cite{ST}, a faithful algebraic representation $\Xi:PGSp_4\to GL_m$ is required for us (see \cite[Section 8]{ST}).
Here, we consider the adjoint action $\mathrm{Ad}:PGSp_4\to GL(\mathrm{Lie}{(Sp_4)})$ as it, i.e., $\Xi :PGSp_4\to GL_{10}$.
By \cite[Lemma 8.4]{ST}, there exists a natural number $N_0$ such that, if $N\geq N_0 \prod_{p\in S'} p^{10\kappa}$, then we have $I_3(f)=I_4(f)=I_5(f)=0$ and the contributions of the non-central elements vanish in $I_6(f)$.
Hence, the estimation for $I_6(f)$ can be proved by \cite[Lemma 2.14]{ST} and \eqref{ga1}.
The estimation for $I_2(f)$ follows from Lemmas \ref{I21} and \ref{I22}.
\end{proof}
\begin{prop}\label{weight-est}
(Weight aspect)
Fix a level $N>0$.
There exist positive constants $a'$ and $b'$ such that
\[
I_2(f)\times \vol(K(N))^{-1}\times |\nu(\alpha)|^{-\frac{k_1+k_2}{2}}_{S'} =O(p_{S'}^\kappa(k_1-k_2+1)(k_1+k_2-3))
\]
\[
I_3(f)\times \vol(K(N))^{-1}\times |\nu(\alpha)|^{-\frac{k_1+k_2}{2}}_{S'} =O(p_{S'}^\kappa (k_1-1)(k_2-2)),
\]
\[
\{I_4(f)+I_6(f)\}\times \vol(K(N))^{-1}\times |\nu(\alpha)|^{-\frac{k_1+k_2}{2}}_{S'} =O(p_{S'}^{a'\kappa +b'}(k_1+k_2-3)),
\]
\[
I_5(f)\times \vol(K(N))^{-1}\times |\nu(\alpha)|^{-\frac{k_1+k_2}{2}}_{S'} =O(p_{S'}^{a'\kappa +b'}(k_1+k_2-3))
\]
for any $(k_1,k_2)$, $\kappa\geq 1$, $S'$, and $f_{S',\alpha}$, which satisfy the conditions $k_1\geq k_2\geq 3$, $f_{S',\alpha}\in H^\mathrm{ur}(G(\Q_{S'}))^\kappa$, and $N$ is prime to $\prod_{p\in S'}p$.
\end{prop}
\begin{proof}
For each $M$ in $\cL$, let $Y_M$ denote the set of $M(\A)$-conjugacy classes of semisimple $\R$-elliptic elements of $M(\Q)$ whose contributions to $I_\mathrm{geom}(f)$ are non-zero.
By \cite[Proposition 8.7]{ST}, one finds $|Y_M|=O(p_{S'}^{a_1\kappa+b_1})$ where $a_1$ and $b_1$ are certain positive numbers.
Furthermore, there exist positive numbers $a_2$ and $b_2$ such that $a^M(S,\gamma)\, J_M(\gamma,h_P)=O(p_{S'}^{a_2\kappa+b_2})$ holds for each semisimple $\R$-elliptic element $\gamma$ in $Y_M$.
This fact is due to \cite[Proof of Theorem 9.19]{ST}.
Hence, the estimation for $I_4(f)$ and $I_6(f)$ follows from these results of \cite{ST} and Lemmas \ref{id1} and \ref{id2}.

Here $I_3(f)$ is the total contribution of $z\delta_1$ $(z\in Z_G(\Q))$ and the center $z$ must satisfy $\nu(\alpha)=z^2$.
Hence, only the $G(\Q)$-conjugacy class of $z_\alpha\delta_1$ can contribute to $I_3(f)$.
Therefore, the estimation for $I_3(f)$ is deduced from Lemmas \ref{I31} and \ref{I32}.
By the same argument, only the $G(\Q)$-conjugacy class of $z_\alpha u_\min$ contributes to $I_2(f)$.
Hence, the estimation for $I_2(f)$ follows from Lemmas \ref{I21} and \ref{I22}.

Next we shall consider the term $I_5(f)$.
By Lemma \ref{ei51}, it is enough to treat the $(G,S)$-conjugacy class $\gamma$ such that $G_{\gamma_s}$ is isomorphic to $G_1$, $G_2$, or $G_3$.
By Lemmas \ref{ei53}, \ref{ei54}, \ref{ei55}, each the contribution is bounded by the product of a constant, $\prod_{p\in S'\; \text{or} \; p|N} p^\epsilon$ and a semisimple $\R$-elliptic orbital integral of $h_P$ for a proper standard parabolic subgroup $P$.
Therefore, we can reduce the estimation for $I_5(f)$ to the semisimple case as above.
Hence, the proof is completed.
\end{proof}

Finally we treat $I_1(f)$. 
Suppose $\gamma=z_\gamma \cdot I_4\in Z_G(\Q)$. By (\ref{n-gamma}) it satisfies 
$$\gamma\in {\rm Supp}(h)=(z'_\alpha K(N))\times G(\Z_{S'})\alpha G(\Z_{S'})\times 
\prod_{p\in S\setminus S'\cup\{v|N\infty\}}K_p.$$ 
Further $\gamma$ can happen exactly when $\nu(\alpha)\in (\Q^\times_{S'})^2$. 
In this case, it follows that 
$$\gamma=\left\{
\begin{array}{cc}
z_\alpha \cdot E_4 &\ {\rm if}\ N\ge 3\\
\pm z_\alpha\cdot E_4  &\ {\rm otherwise}
\end{array}\right., 
$$
since $S$ is sufficiently large (see (\ref{zalpha}) for $z'_\alpha\in Z^+_G(\Q)$).  
Furthermore it follows from (\ref{n-gamma}) again that 
$$\nu(\gamma)=|z_{\gamma}|^2_\R=|\nu(\alpha)|^{-1}_{S'}.$$
We define 
$$\ve(\alpha)=\left\{
\begin{array}{cl}
2  &\ {\rm if}\ N=1,2\ {\rm and}\ \nu(\alpha)\in (\Q^\times_{S'})^2 \\
1  &\ {\rm otherwise}
\end{array}\right. 
$$
which is nothing but $\ve_{S'}(\alpha)$ in (\ref{e-factor0}) for $U=K(N)$. 
By Plancherel formula and the limit formula we have 
\begin{equation}
\begin{split}
(-1)^{q(G(\R))}I_1(f)&=\overline{\mu}(G(\Q)A_{G,\infty}\bs G(\A))\dim \xi\sum_{\gamma} 
\gamma^{-(k_1+k_2)} \cdot f_{S',\alpha}(\gamma)\\
&=\overline{\mu}(G(\Q)A_{G,\infty}\bs G(\A))\dim \xi\sum_{\gamma} 
|\gamma|^{-(k_1+k_2)}_\R \cdot f_{S',\alpha}(|\gamma|_\R)\ \ ({\rm by\ } (\ref{parity1}))\\
&=\overline{\mu}(G(\Q)A_{G,\infty}\bs G(\A))\cdot\dim \xi\sum_{\gamma}  \widehat{\mu}^{{\rm pl}}_{S'}(\widehat{f}^{\gamma}_{S',\alpha})
\cdot |\nu(\alpha)|^{\frac{k_1+k_2}{2}}_{S'}\\
&=\ve(\alpha)\cdot \overline{\mu}(G(\Q)A_{G,\infty}\bs G(\A))\cdot\dim \xi\cdot  
\widehat{\mu}^{{\rm pl}}_{S'}(\widehat{f}^{z'_{\alpha}}_{S',\alpha})
\cdot |\nu(z'_{\alpha})|^{\frac{k_1+k_2}{2}}_{S'}
\end{split}
\end{equation}
where  $f^{\gamma}_{S',\alpha}$ (resp. $f^{z'_{\alpha}}_{S',\alpha}$) is the translation by $|\gamma|_\R$ 
(resp. $|z'_{\alpha}|_\R$) of $f_{S',\alpha}$ and we used 
$$\widehat{\mu}^{{\rm pl}}_{S'}(\widehat{f}^{\gamma}_{S',\alpha})=\ds\int_{\widehat{G(\Q_{S'})}}\widehat{f}^\gamma_{S',\alpha}(\pi)
\widehat{\mu}^{{\rm pl}}_{S'}(\pi)=\ds\int_{\widehat{G(\Q_{S'})}}{\rm tr}(\pi(f^\gamma_{S',\alpha}))
\widehat{\mu}^{{\rm pl}}_{S'}(\pi)$$
$$=f^\gamma_{S',\alpha}(1)=f_{S',\alpha}(|\gamma|_\R)=f_{S',\alpha}(|z'_{\alpha}|_\R)=f^{z'_{\alpha}}_{S',\alpha}(1)
=\widehat{\mu}^{{\rm pl}}_{S'}(\widehat{f}^{z'_{\alpha}}_{S',\alpha}).$$ 
Hence we have 
\begin{equation}\label{main-term}
\frac{|\nu(\alpha)|^{-\frac{k_1+k_2}{2}}_{S'}I_1(f)}{\ve(\alpha)\cdot\overline{\mu}(G(\Q)A_{G,\infty}\bs G(\A))\cdot\dim \xi}
=(-1)^{q(G(\R))}\widehat{\mu}^{{\rm pl}}_{S'}(\widehat{f}^{z'_{\alpha}}_{S',\alpha})=
-\widehat{\mu}^{{\rm pl}}_{S'}(\widehat{f}^{z'_{\alpha}}_{S',\alpha}).
\end{equation}
Note that if $\nu(\alpha)\not\in (\Q^\times_{S'})^2$, then both sides of (\ref{main-term}) are zero giving the trivial identity. 

Summing up we have obtained the following results which follow from Propositions \ref{level-est}, \ref{weight-est}, and (\ref{main-term}). 
\begin{thm}\label{level-aspect}(Level-aspect)
Keep the notation as in Proposition \ref{level-est}. Then 
$$\frac{|\nu(\alpha)|^{-\frac{k_1+k_2}{2}}_{S'}I_{{\rm geom}}(f)}{\ve(\alpha)\cdot\overline{\mu}(G(\Q)A_{G,\infty}\bs G(\A))\cdot\dim \xi}
=-\widehat{\mu}^{{\rm pl}}_{S'}(\widehat{f}^{z'_{\alpha}}_{S',\alpha})+A+O(p^{a\kappa+b}_{S'}\varphi(N)N^{-3}),$$
where $A=O(p^\kappa_{S'}\varphi(N)N^{-2})$, $(N,p_{S'})=1$, and $N\geq N_0 p_{S'}^{10\kappa}$.
\end{thm}

\begin{thm}\label{weight-aspect}(Weight-aspect)
Keep the notation as in Proposition \ref{weight-est}. Then 
$$
\frac{|\nu(\alpha)|^{-\frac{k_1+k_2}{2}}_{S'}I_{{\rm geom}}(f)}{\ve(\alpha)\cdot\overline{\mu}(G(\Q)A_{G,\infty}\bs G(\A))\cdot\dim \xi}
=-\widehat{\mu}^{{\rm pl}}_{S'}(\widehat{f}^{z'_{\alpha}}_{S',\alpha})+ B_1+B_2+
O(\frac{p^{a'\kappa+b'}_{S'}}{(k_1-k_2+1)(k_1-1)(k_2-2)}),
$$
where $B_1=O(\frac{p^\kappa_{S'}}{(k_1-1)(k_2-2)})$ and $B_2=O(\frac{p^\kappa_{S'}}{(k_1-k_2+1)(k_1+k_2-3)})$, and $(N,p_{S'})=1$.
\end{thm}
\begin{remark}
Shin's condition in the weight aspect in \cite{Shin} becomes: 
$$k_1-k_2\lra \infty,\ k_2\lra \infty.$$
Theorem \ref{weight-aspect} can also treat in the case where $k_1-k_2$ is constant while $k_2$ tends to infinity. 
This is a new direction of the weight aspect which has not been studied. 
\end{remark}
\begin{remark}\label{level-fixed-char} In Theorem \ref{level-aspect}, if we restrict automorphic forms to those  
with a fixed central character $\chi\in \widehat{(\Z/N\Z)^\times}$, then we have the better result without $\varphi(N)$ on the right hand side. Accordingly the dimension of Siegel cusp forms with a fixed central character is smaller by a factor of $\varphi(N)$. (See Proposition \ref{dimension}.)
\end{remark}

\section{Proof of the main theorem}
In this section we give a proof of Theorem \ref{main} in the introduction. 

Fix a prime $p\nmid N$. Let $U=K(N)$ and $S'=\{p\}$. For any $\kappa\in \Bbb Z_{\geq 0}$, let
$f_p$ be the characteristic function of $K_p\diag(p^{-a_1},p^{-a_2},p^{a_1-\kappa},p^{a_2-\kappa})K_p,\ 
0\le a_2\le a_2\le \kappa$. 

If $k_1\geq k_2\geq 4$, it is immediate by Proposition \ref{trace}, Theorems \ref{level-aspect} and \ref{weight-aspect}.

If $k_2=3$, we need to estimate the traces of Hecke operators on the residual spectrum and a 
part of cuspidal space related to non-tempered representations $\omega_{l_1}$.
 
If $k_1> k_2=3$, by Propositions \ref{K-Eisen} and \ref{cap-k-order},
\begin{eqnarray*}
\ds\sum_{\ast\in\{{\rm cusp},{\rm res}\}}\widehat{\mu}_{K(N),\xi_{\underline{k}},\omega_{l_1}, \ast}(\widehat{f}_{p}) =({\rm dim} \xi_{\underline{k}})^{-1}(O(p^{\frac{\kappa}{2}}N^{-4+\ve})+O(p^{\frac{\kappa}{2}}k_1N^{-6})).
\end{eqnarray*}

If $k_1=k_2=3$, by (\ref{residual}), Propositions \ref{K-Eisen} and \ref{cap-k-order},
\begin{eqnarray*}
&& \widehat{\mu}_{K(N),\xi_{\underline{k}},\bold{1}}(\widehat{f}_{p})+\ds\sum_{\ast\in\{{\rm cusp},{\rm res}\}}\widehat{\mu}_{K(N),\xi_{\underline{k}},\omega_{l_1}, \ast}(\widehat{f}_{p}) \\
&& =({\rm dim} \xi_{\underline{k}})^{-1} (O(p^{\frac{\kappa}{2}}N^{-4+\ve})+O(p^{\frac{\kappa}{2}}k_1N^{-6})+
O(p^{\frac{3\kappa}{2}}N^{-10})).
\end{eqnarray*}

Notice that the above error terms are subsumed in the error terms in Theorems \ref{level-aspect} and \ref{weight-aspect}. 
Therefore Theorem \ref{main} follows from Proposition \ref{trace}, Theorems \ref{level-aspect} and \ref{weight-aspect}. 
This completes a proof of the main theorem.  

\section{Applications}

\subsection{The classical formulation}
In this subsection we reformulate Theorem \ref{main} in terms of classical Siegel modular forms. 
As we will see later, this will be used in the study of Hecke fields and low-lying zeros. 

Let  $\chi:(\Z/N\Z)^\times\lra \C^\times$ be a Dirichlet character. Fix a square root $\chi(p)^{\frac{1}{2}}$ 
for each fixed $p\nmid N$. We write $\chi(p)^{\frac{i}{2}}=(\chi(p)^{\frac{1}{2}})^i$.  
Put $V_{\underline{k},N}=S_{\underline{k}}(\G(N))$ or $S_{\underline{k}}(\G(N),\chi)$  for $k_1\ge k_2\ge 3$. 
Let us recall the Hecke operators $T_m$ or $T(p^i)$ on $V_{\underline{k},N}$ for $m\in \Delta_n(N)$ and $p\nmid N$. 
We normalize them as $T'_m=T_m/\nu(m)^{\frac{k_1+k_2-3}{2}}$ on $S_{\underline{k}}(\G(N))$ and 
 $T'(p^i)=T(p^i)/(p^{\frac{k_1+k_2-3}{2}}\chi(p)^{\frac{i}{2}})$ on $S_{\underline{k}}(\G(N),\chi)$.  
Put $d_{\underline{k},N}=\dim_\C V_{\underline{k},N}$. Clearly 
if $m=p^\kappa E_4$, then 
$$\frac{1}{d_{\underline{k},N}}{\rm tr}(p^{-3\kappa}T'_{p^\kappa E_4}|V_{\underline{k},N})=1.$$
Then by Theorem \ref{main}, we have the following theorem:  

\begin{thm}\label{claim1}There exist constants $a_1,b_1,a'_1,b'_1$ depending only on $G$ such that for a prime $p\nmid N$ and $m=\diag(p^{a_1},p^{a_2},p^{-a_1+\kappa},p^{-a_2+\kappa}),\ a_1,a_2,\kappa\in \Z$ satisfying 
$0\le a_2\le a_1\le \kappa$ and $m\not\in Z_G(\Q)$,  
\begin{enumerate}
\item (Level-aspect) 
$$\frac{1}{d_{\underline{k},N}}{\rm tr}(T'_m|V_{\underline{k},N})=A+O(p^{a_1\kappa+b_1}N^{-3}),\, A=O(p^{-\frac{\kappa}{2}}N^{-2}), \quad N\gg p^{10\kappa}.
$$
\item (Weight-aspect) 
$$\frac{1}{d_{\underline{k},N}}{\rm tr}(T'_m|V_{\underline{k},N})=B_1+B_2+
O(\frac{p^{a'_1\kappa+b'_1}}{(k_1-k_2+1)(k_1-1)(k_2-2)})\ \  (k_1+k_2\to \infty),$$
$$B_1=O(\frac{p^{-\frac{\kappa}{2}}}{(k_1-1)(k_2-2)}),\ B_2=O(\frac{p^{-\frac{\kappa}{2}}}{(k_1-k_2+1)(k_1+k_2-3)})$$
\end{enumerate}
\begin{proof}Since $\nu(m)=p^\kappa$, by (\ref{local-global}) the classical Hecke operator $T'_m$ is interpreted as the action of 
$$f=f_{K(N)}(p^{-\frac{3}{2}\kappa}[K_p m^{-1}K_p])\Big(\bigotimes_{\ell\in S\setminus\{v|pN\infty\}}{\rm char}_{K_\ell} \Big)$$
on the spectral side. Then the LHS of the main theorem is exactly 
$\frac{1}{d_{\underline{k},N}}{\rm tr}(T'_m|V_{\underline{k},N})$.

Notice that 
$${\rm dim}S_{\underline{k}}(\G(N))\sim \frac {({\rm dim}\xi_{\underline{k}})\varphi(N)}{{\rm vol}(K(N))},\ 
{\rm dim}S_{\underline{k}}(\G(N),\chi)\sim \frac {{\rm dim}\xi_{\underline{k}}}{{\rm vol}(K(N))}.
$$
Then multiplying the RHS of the equations in Theorem \ref{level-aspect}, \ref{weight-aspect}  
by $p^{-\frac{3}{2}\kappa}$ we obtain the results  with Remark \ref{level-fixed-char}. 
\end{proof}

\end{thm}   
\begin{remark} The weight aspect depends on 
how we increase the weight. For instance, if $k_1=k_2$ goes to infinity, then $B_2$ becomes the second main term and $B_1$ is 
subsumed into the error term.  
On the other hand, if $k_2$ is fixed and $k_1$ goes to infinity, then $B_1$ 
becomes the second main term and $B_2$ is 
subsumed into the error term. This aspect would be a new case which has not been studied before.    
\end{remark}

\subsection{The vertical Sato-Tate theorem; proof of Theorem \ref{Sato-Tate}}
Let $K$ be the maximal open compact subgroup of $G(\Q_p)=GSp_4(\Q_p)$. 
Let us first recall the Plancherel measure $\widehat{\mu}^{{\rm pl}}_p$ for the unitary dual of $G(\Q_p)$. 
For our purpose it suffices to consider its restriction to the unramified tempered classes $\widehat{G(\Q_p)}^{{\rm ur, temp}}_\chi$ with a fixed unitary central character $\chi:\Q^\times_p\lra \C^1$. 
We denote it by $\widehat{\mu}^{{\rm pl,temp}}_{p,\chi}$. 
Then by Lemma 3.2 of \cite{ST}, we have a natural bijection 
\begin{equation}\label{omega}
\widehat{G(\Q_p)}^{{\rm ur, temp}}_\chi\simeq [0,\pi]^2
\end{equation}
which is in fact a topological isomorphism. 
By Proposition 3.3 of \cite{ST}, for a usual parameter $(\theta_1,\theta_2)$ of $[0,\pi]^2$,  
we have 
\begin{eqnarray*}
&& \widehat{\mu}^{{\rm pl,temp}}_{p,\chi}(\theta_1,\theta_2)=\frac{(p+1)^4}{p^4\pi^2} \\
&&\cdot \Bigg|\frac{(1-e^{2\sqrt{-1}\theta_1})(1-e^{2\sqrt{-1} \theta_2})(1-e^{\sqrt{-1} (\theta_1+\theta_2)})(1-e^{\sqrt{-1} (\theta_1-\theta_2)})}{(1-p^{-1}e^{2\sqrt{-1}\theta_1})(1-p^{-1}e^{2\sqrt{-1} \theta_2})(1-p^{-1}e^{\sqrt{-1} (\theta_1+\theta_2)})(1-p^{-1}e^{\sqrt{-1} (\theta_1-\theta_2)})}\Bigg|^2 \, d\theta_1d\theta_2.
\end{eqnarray*}
Note that $\ds\lim_{p\to\infty}\widehat{\mu}^{{\rm pl,temp}}_{p,1}=\mu^{{\rm ST}}_\infty$. 
By transforming  $(\theta_1,\theta_2)$ into $(x,y)=(2\cos\theta_1,2\cos\theta_2)$, one has the measure $\mu_p$ on 
$\Omega=[-2,2]^2$ in the introduction. 

By Stone-Weierstrass theorem, the natural map 
$H^{{\rm ur}}(G(\Q_p))\hookrightarrow C^\infty_c(G(\Q_p))\lra C^0(\Omega,\R)$ has the dense image where 
the second map is given by the restriction of the correspondence 
$f\mapsto \widehat{f}$ to $\Omega$ via (\ref{omega}). 

Now apply Theorem \ref{claim1} with $m=p$.
Then Theorem \ref{Sato-Tate} follows from Theorem \ref{main}.  

\subsection{Hecke fields; Proof of Corollary \ref{hecke2}}\label{Hecke-field}
Corollary \ref{hecke1} is a direct consequence of Theorem \ref{Sato-Tate} and Theorem \ref{dim-en1}. 
To see the subsequent corollary we need to estimate the dimension of the endoscopic lifts to 
$S_{\underline{k}}(\G(N))$. It reveals the contribution of the second main term of the geometric side. 
(See the second term of the RHS in Theorem \ref{main}.) 
Let $S_{\underline{k}}(\G(N))^{{\rm en}}$ be the subspace of $S_{\underline{k}}(\G(N))^{{\rm tm}}$ generated by 
Hecke eigen form $F$ so that $\pi_F$ is endoscopic. 
By Theorem \ref{dim-en1} we see that 
$$\ds\frac {\dim S_{\underline{k}}(\G(N))^{{\rm en}}}{\dim S_{\underline{k}}(\G(N))}=O(((k_1-1)(k_1-2))^{-1}N^{-2+\epsilon}), \ {\rm as}\ k_1+k_2+N\to \infty,\ (N,11!)=1. 
$$  
Corollary \ref{hecke2} now  follows from this with Theorem \ref{Sato-Tate}. 

\section{Properties of $L$-functions of Siegel cusp forms on $GSp_4$}\label{L-function}

Put $S_{\underline{k}}(N)=S_{\underline{k}}(\G(N),1)$ and $HE_{\underline{k}}(N)=HE_{\underline{k}}(\Gamma(N),1)$ 
as in the introduction. 
Given a Siegel cusp form $F\in HE_{\underline{k}}(N)$, let $\pi_F$ be the associated cuspidal representation of $GSp_4(\A)$. 
Throughout this section, we assume that the central character of $\pi_F$ is trivial and the level of $F$ satisfies $(N,11!)=1$ 
due to \cite{Fer} to control the conductor 
under the functorial lift from the endoscopic subgroup of $GSp_4$. 

\subsection{Degree 4 spinor $L$-functions}
Let us first assume that $\pi_F$ is a CAP representation. Since $k_2\geq 3$, by the classification of CAP representations,
we must have $k:=k_1=k_2\ge 3$ and 
it is associated to Siegel parabolic subgroup. 
As seen in Section \ref{sccf}, if $(N,11!)=1$ there exists a newform $f$ with trivial central character in 
$S_{2k-2}(\G^1(N))\subset S_{2k-2}(\G^1_1(N^2))$ so that $\pi_F$ comes from $\pi_f$ via a theta lift or Saito-Kurokawa lift if it is 
of level one. 
We define the spinor L-function for such $\pi_F$ by 
$$L(s,\pi_F, {\rm Spin}):=\zeta(s+\frac{1}{2})\zeta(s-\frac{1}{2})L(s,\pi_f).$$ 
Then the conductor of the L-function $L(s,{\rm spin},\pi_F)$ divides  
$N^2$. Our definition of the L-functions for CAP representations differs by local factors at bad places from Schmidt's definition of L-functions (see Theorem 5.2 of \cite{Schmidt1}), but it does not matter for the analytic properties we need below. 

Next we assume that $\pi_F$ is endoscopic. As seen in Section \ref{class-endo} it can be obtained by a theta lift from $H(\A)$ where 
$H=GSO(4)$ or $H=GSO(2,2)$. We may put $\pi_F=\theta(\tau)=\otimes'_v\theta_v(\tau_v)$ for some cuspidal representation of 
$\tau=\otimes'_v\tau_v$ on $H(\A)$. 
Since $\pi_F$ is non-generic, only the case $H=GSO(4)$ happens. 
Let $(\pi_1,\pi_2)$ be a pair of two cuspidal automorphic representations of $GL_2(\A)$ with the same central character 
obtained from $\tau$ via Jacquet-Langlands correspondence. As seen before, by Th\'eor\`eme 3.2.3 of \cite{Fer} 
it turns out that    
$\pi_i$ has a fixed vector under the action of $K^1(N)$ under the assumption $(N,11!)=1$. We define 
$$L(s,\pi_F, {\rm Spin}):=L(s,\pi_1)L(s,\pi_2)$$
and it has the conductor dividing $N^4$ since the conductor of each $L(s,\pi_i)$ divides $N^2$ by \cite{LR}.  

Finally we assume that $\pi_F$ is neither CAP nor endoscopic. Since $F$ has a cohomological weight, by \cite{Wei1},
$\pi_{F,p}$ is weakly equivalent to a generic cuspidal representation $\pi=\otimes'_p\pi_p$ of $GSp_4(\A)$ so that 
$\{\pi_{F,\infty},\pi_\infty\}$ makes up a L-packet of $\Pi(GSp_4(\R))$. 
Since $F$ is non-endoscopic, the 
$\ell$-adic Galois representation associated to $F$ is irreducible by Chebotarev density theorem and \cite{CG}. 
By the comparison theorem between de Rham cohomology and etale cohomology, $\pi_{F}$ and $\pi$ contribute simultaneously 
to the de Rham cohomology of the Siegel threefold $S_{K(N)}$ for $K(N)$ with the automorphic vector bundle 
corresponding to $\xi_{\underline{k}}$. Hence $\pi$ has a non-zero $K(N)$-fixed vector and unramified outside of $N$ provided that $F$ is of level $\G(N)$.  

We define 
$$L(s,\pi_F,{\rm Spin}):=L(s,\pi)=\prod_{p<\infty} L_p(s,\pi_p)=\sum_{n=1}^\infty \widetilde\lambda_F(n)n^{-s}. 
$$ 
This is independent of such a $\pi$ since the multiplicity one is known for generic cuspidal representations of $GSp_4(\A)$. 
Then for each prime $p\nmid N$, we may write 
\begin{eqnarray*}
L(s,\pi_F, {\rm Spin})_p^{-1}=(1-\alpha_{0p}p^{-s})(1-\alpha_{0p}\alpha_{1p}p^{-s})(1-\alpha_{0p}\alpha_{2p}p^{-s})(1-\alpha_{0p}\alpha_{1p}\alpha_{2p}p^{-s}),
\end{eqnarray*}
$$\widetilde\lambda_F(p)=\alpha_{0p}+\alpha_{0p}\alpha_{1p}+\alpha_{0p}\alpha_{2p}+\alpha_{0p}\alpha_{1p}\alpha_{2p}=\lambda_F(p)p^{-\frac{k_1+k_2-3}{2}}.
$$
We note that $\widetilde\lambda_F(p^2)=\lambda_F(p^2)p^{-(k_1+k_2+3)}+p^{-1}$.

Since the central character is trivial, one has a relation $\alpha_{0p}^2\alpha_{1p}\alpha_{2p}=1$.
Let $\Gamma_\Bbb R(s)=\pi^{-\frac s2}\Gamma(\frac s2)$ and $\Gamma_\Bbb C(s)=2(2\pi)^{-s}\Gamma(s)$. 

\begin{lem}\label{cond1} Let $\Lambda(s,\pi_F, {\rm Spin})=q(F)^{\frac s2} \Gamma_\Bbb C(s+\frac {k_1+k_2-3}2)\Gamma_\Bbb C(s+\frac {k_1-k_2+1}2)L(s,\pi_F, {\rm Spin})$. Then 
$$\Lambda(s,\pi_F, {\rm Spin})=\epsilon(\pi_F)\Lambda(1-s,\pi_F, {\rm Spin}),
$$
where $\epsilon(\pi_F)\in\{\pm 1\}$ and $N\le q(F)\le N^4$.
\end{lem}
\begin{proof} We first bound the conductor $q(F)$. 
When $\pi_F$ is either CAP or endoscopic, it was proved above.
Otherwise, let $\pi=\otimes'_{p} \pi_p$ of $GL_4(\A)$ be the strong transfer of $\pi_F$.
The conductor $q(F)$ can be written in terms of the local conductor of $\pi$. 
By Proposition 1 of \cite{Sor} (see also the last few lines of its proof) which is still true for not only supercuspidal representations but also 
square integrable representations, the depth is preserved under the above transfer. 
For a prime $p$ such that $\pi_{F,p}$ is square integrable we have 
$$\ord_p(q(F))=c(\pi_{F,p})=c(\pi_p)=4({\rm depth}(\pi_p)+1)=4({\rm depth}(\pi_{F,p})+1)$$
where we applied Proposition 2.2 of \cite{LR} (resp. Main Theorem of  \cite{Sor}) to get the third (resp. the second) equality.  
By definition of depth, $\pi^{K(N)}_F\not=0$ implies ${\rm depth}(\pi_p)\le \ord_p(N)-1$.   
This gives $\ord_p(q(F))\le 4\cdot\ord_p(N)$ provided if $\pi_{F,p}$ is square integrable or unramified. 
In the remaining cases we can directly compute the conductor by using the Table A.9 of \cite{RS} under 
the condition $\pi^{K(N)}_F\not=0$.  
\end{proof}

Let 
$$-\frac {L'}L(s,\pi_F, {\rm Spin})=\sum_{n=1}^\infty \Lambda(n) a_F(n)n^{-s},
$$
where $\Lambda(n)$ is the von-Mangoldt function, and 
$$a_F(p^d)=\alpha_{0p}^d+(\alpha_{0p}\alpha_{1p})^d+(\alpha_{0p}\alpha_{2p})^d+(\alpha_{0p}\alpha_{1p}\alpha_{2p})^d.
$$
For each $m\in T(\Q)$ we normalize the Hecke operator $T_m$ so that $T'_m:=T_m\nu(m)^{-\frac{k_1+k_2-3}{2}}$ and 
accordingly $T'(p^n)=T(p^n)p^{-\frac{n(k_1+k_2-3)}{2}}$. For a Hecke eigen form $F$,  
we denote by $\lambda'_{F,m}$ (resp. $\lambda'_F(p^n)$) the Hecke eigenvalue of $F$ for $T'_m$ (resp. $T'_F(p^n)$). 
By using the relations (\ref{hecke-operators}) it is easy to see 
that $$a_F(p)=\lambda'_F(p),\ a_F(p^2)=\lambda'_{F,t^2_1}-(p-1)\lambda'_{F,t_2}-\Big(1-\frac{1}{p}\Big)\Big(1+\frac{1}{p^2}\Big)$$
where $t_1=\diag(1,1,p,p)$ and $t_2=\diag(1,p,p^2,p)$.  
To apply Theorem \ref{claim1} we have to express these values in terms of linear combinations of the eigenvalues for 
the Hecke operators which take the shape of $T'_m$ as in Theorem \ref{claim1}.     
Then we find the corresponding operators 
$$T'(p)=T'_{t_1},$$
$$T'(p^2)= T'_{t^2_1}-(p-1)T'_{t_2}-\Big(1-\frac{1}{p}\Big)\Big(1+\frac{1}{p^2}\Big)(p^3T'_{p E_4})$$ 
respectively. Note that $p^3T'_{p E_4}$ acts on $S_{\underline{k}}(\G(N))$ as the identity map. 
Therefore we have 
\begin{prop}\label{spin} Assume $(N,11!)=1$. 
Put $\underline{k}=(k_1,k_2),\ k_1\ge k_2\ge 3$ and $d_{\underline{k},N}:=\dim S_{\underline{k}}(N)$.  
There exist constants $a''_1,a''_2,b_1'', b_2'', c_1'', c_2'',v_1,v_1',w_1,w_1'$ depending only on $G$ such that 
\begin{enumerate}
\item \label{est-p}
\begin{enumerate}
\item (level-aspect) Fix $k_1,k_2$. Then for $N\gg p^{10}$,
$$\frac{1}{d_{\underline{k},N}}\sum_{F\in HE_{\underline{k}}(N)} a_F(p) = O(p^{-\frac{1}{2}}N^{-2})
+O(p^{v_1}N^{-3});
$$ 
\item (weight-aspect) Fix $N$. Then as $k_1+k_2\to\infty$,
$$\frac{1}{d_{\underline{k},N}}\sum_{F\in HE_{\underline{k}}(N)} a_F(p) =B_1+B_2+
O(\frac{p^{v'_1}}{(k_1-k_2+1)(k_1-1)(k_2-2)}),$$
$$B_1=O(\frac{p^{-\frac{1}{2}}}{(k_1-1)(k_2-2)}),\ B_2=O(\frac{p^{-\frac{1}{2}}}{(k_1-k_2+1)(k_1+k_2-3)})$$
\end{enumerate}

\item \label{est-p^2}
\begin{enumerate}
\item (level-aspect) Fix $k_1,k_2$. Then for $N\gg p^{10}$,
$$\frac{1}{d_{\underline{k},N}}\sum_{F\in HE_{\underline{k}}(N) } a_F(p^2) = -\Big(1-\frac{1}{p}\Big)\Big(1+\frac{1}{p^2}\Big) +
 O\left((a''_1p^{-\frac{1}{2}}+a''_2p^{\frac{1}{2}})N^{-2}\right)\\
+O\left(p^{w_1}N^{-3}\right).
$$
\item (weight-aspect) Fix $N$. Then as $k_1+k_2\to\infty$,
$$\frac{1}{d_{\underline{k},N}}\sum_{F\in HE_{\underline{k}}(N) } a_F(p^2) = -\Big(1-\frac{1}{p}\Big)\Big(1+\frac{1}{p^2}\Big)+B_1+B_2+
O(\frac{p^{w'_1}}{(k_1-k_2+1)(k_1-1)(k_2-2)})\ ,$$
$$B_1=O(\frac{p^{-\frac{1}{2}}b''_1+p^{\frac{1}{2}}b''_2}{(k_1-1)(k_2-2)}),\ B_2=O(\frac{p^{-\frac{1}{2}}c''_1+p^{\frac{1}{2}}c''_2}{(k_1-k_2+1)(k_1+k_2-3)}).
$$
\end{enumerate}
\end{enumerate}
\end{prop}
\begin{proof}The claim follows from Theorem \ref{claim1}. 
\end{proof}

\subsection{Degree 5 standard $L$-functions}
Let us first recall the standard $L$-function for $\pi_F=\otimes'_p\pi_{F,p}$. 
Let $\tau=\pi_F|_{Sp_4}$ be the restriction of $\pi$ to $Sp_4(\Bbb A)$, and $\Pi$ be the transfer of $\tau$ corresponding to $\omega_2|_{Sp_4(\Bbb C)}$, where $\omega_2: GSp_4(\Bbb C)\longrightarrow GL_5(\Bbb C)$ is the homomorphism attached to the second fundamental weight. Note that if $\iota: GSp_4(\Bbb C)\hookrightarrow GL_4(\Bbb C)$, $\wedge^2\circ \iota=\omega_2\oplus 1$. Therefore $\wedge^2\pi=\Pi\boxplus 1$, and 
$L(s,\pi_F,{\rm St})=L(s,\tau,{\rm St})=L(s,\Pi)$.
For any unramified prime $p$,  
\begin{eqnarray*}
L(s,\pi_F,{\rm St})_p^{-1}=(1-p^{-s})(1-\alpha_{1p}p^{-s})(1-\alpha_{2p}p^{-s})(1-\alpha_{1p}^{-1}p^{-s})(1-\alpha_{2p}^{-1}p^{-s}).
\end{eqnarray*}

\begin{lem} Let $\Lambda(s,\pi_F,{\rm St})=q(F,{\rm St})^{\frac s2} \Gamma_\Bbb R(s)\Gamma_\Bbb C(s+k_1-1)\Gamma_\Bbb C(s+k_2-2)L(s,\pi_F,{\rm St})$.
 Then
$$\Lambda(s,\pi_F,{\rm St})=\epsilon(\pi_F,{\rm St})\Lambda(1-s,\pi_F,{\rm St}),
$$
where $\epsilon(\pi_F,{\rm St})\in\{\pm 1\}$, and $N\le q(F,{\rm St})\le N^{28}$.
\end{lem}
(G. Henniart noted in a private communication that we would have $q(F,{\rm St})\ll (N^4)^{\frac 32}=N^6$.)
\begin{proof} We bound the conductor $q(F,{\rm St})$ since others are well-known. By Lemma \ref{cond1} we have know that 
$q(F)\le N^4$. Let $\pi$ be the strong transfer of $\pi_F$ to $GL_4(\A)$. Then the global conductor 
$q(\pi)$ coincides with $q(F)$. Then the conductor is roughly estimated by the main theorem of \cite{BH} as follows:  
$$q(F,{\rm St})\le q(\pi\otimes\pi)\le N^{4(2\cdot 4-1)} =N^{28}.
$$ 
This gives us the claim. 
\end{proof}
Let 
$$L(s,\pi_F,{\rm St})=\sum_{n=1}^\infty \mu_F(n) n^{-s}.
$$ 
Then
$$\mu_F(p)=1+\alpha_{1p}+\alpha_{2p}+\alpha_{1p}^{-1}+\alpha_{2p}^{-1},\quad
\lambda'_F(p)^2-\lambda'_F(p^2)-p^{-1}=\mu_F(p)+1.
$$
Let 
$$-\frac {L'}L(s,\pi_F,{\rm St})=\sum_{n=1}^\infty \Lambda(n) b_F(n)n^{-s},
$$
where 
$$b_F(p^d)=1+\alpha_{1p}^d+\alpha_{2p}^d+\alpha_{1p}^{-d}+\alpha_{2p}^{-d}.
$$
Note that $a_F(p^2)=\lambda_F'(p)^2-b_F(p)-1=\lambda_F'(p^2)+p^{-1}$, and 
\begin{equation}\label{bp1}
b_F(p)=\mu_F(p)=p^{-1}\lambda'_{F,t_2}+p^{-2}.
\end{equation}  
By using the relations (\ref{hecke-operators}), we see that

$$(\lambda_{F,t_2}')^2=\lambda'_{F,\diag(1,p^2,p^4,p^2)}+(p+1)\lambda'_{F,\diag(p,p,p^3,p^3)}+(p^2-1)\lambda'_{F,\diag(p,p^2,p^3,p^2)}
+p^{-6}(1+p+p^3+p^4).
$$

Therefore,
\begin{eqnarray} \label{bp2}
&&b_F(p^2)=b_F(p)^2-2 a_F(p^2)-2b_F(p)-2 \\
&& =(p\lambda_{F,t_2}'+p^{-2})^2-2 (\lambda_F'(p^2)+p^{-1})-2a_F(p^2)-2 \nonumber \\
&& =p^2 (\lambda_{F,t_2}')^2+2 p^{-1} \lambda_{F,t_2}'-2\lambda_{F,t_2}'-2a_F(p^2)+p^{-4}-2p^{-1}-2\nonumber \\
&& =p^2 \lambda'_{F,\diag(1,p^2,p^4,p^2)}+p^2(p+1)\lambda'_{F,\diag(p,p,p^3,p^3)}+p^2(p^2-1)\lambda'_{F,\diag(p,p^2,p^3,p^2)} \nonumber\\
&& +(2 p^{-1}-2) \lambda_{F,t_2}'-2\Big(\lambda'_{F,t^2_1}-(p-1)\lambda'_{F,t_2}-\Big(1-\frac{1}{p}\Big)\Big(1+\frac{1}{p^2}\Big)\Big)-1-p^{-1}+p^{-3}+2p^{-4}, \nonumber \\
&& =p^2 \lambda'_{F,\diag(1,p^2,p^4,p^2)}+p^2(p+1)\lambda'_{F,\diag(p,p,p^3,p^3)}+p^2(p^2-1)\lambda'_{F,\diag(p,p^2,p^3,p^2)} \nonumber\\
&& +2 (p+p^{-1}-2) \lambda_{F,t_2}'-2\lambda'_{F,t^2_1}+1-3p^{-1}+2p^{-2}-p^{-3}+2p^{-4}, \nonumber
\end{eqnarray}
where $t_1=\diag(1,1,p,p)$ and $t_2=\diag(1,p,p^2,p)$. 
To obtain an estimation for the average of $b_F(p^2)$, according to (\ref{bp2}), we apply Theorem \ref{claim1} to 
$\lambda'_{F,\diag(1,p^2,p^4,p^2)}$, $\lambda'_{F,\diag(p,p,p^3,p^3)}$, $\lambda'_{F,\diag(p,p^2,p^3,p^2)}$, $\lambda_{F,t_2}$, and
$\lambda_{F,t^2_1}$.

\begin{remark}
We can see easily that 
$L(s,\Pi,\wedge^2)=L(s,\pi,{\rm Sym}^2)$. Under the Langlands functoriality conjecture, we expect ${\rm Sym}^2(\pi)$ to be 
an automorphic representation of $GL_{10}$.
Since $L(s, \pi,\wedge^2)$ has a pole at $s=1$, $L(s,\pi, {\rm Sym}^2)$ has no pole at $s=1$.
Let $L(s,\pi,{\rm Sym}^2)=\sum_{n=1}^\infty \lambda_{{\rm Sym}^2\pi}(n)n^{-s}$. Then 
$\lambda_{{\rm Sym}^2\pi}(p)=\widetilde\lambda_F(p^2)=\lambda_F'(p^2)+p^{-1}$.
\end{remark}
Recall that $S_{\underline{k}}(N)=S_{\underline{k}}(\G(N),1)$ and $HE_{\underline{k}}(N)=HE_{\underline{k}}(\Gamma(N),1)$. 
Then by Theorem \ref{claim1} we have the following: 

\begin{prop}\label{standard}
Put $\underline{k}=(k_1,k_2)$ and $d_{\underline{k},N}:=\dim S_{\underline{k}}(N)$. 
There exist constants $v_1,v'_1,w_1,w'_1$ depending only on $G$ such that 
\begin{enumerate}
\item \label{est-p-s}
\begin{enumerate}
\item (level-aspect) Fix $k_1,k_2$. Then for $N\gg p^{30}$, 
$$\frac 1{d_{\underline{k},N}}\sum_{F\in HE_{\underline{k}}(N)} b_F(p) =-p^{-2} +   
O(p^{-\frac{3}{2}}N^{-2})
+O(p^{v_1}N^{-3}).$$ 

\item (weight-aspect) Fix $N$. Then as $k_1+k_2\to\infty$,
$$\frac 1{d_{\underline{k},N}}\sum_{F\in HE_{\underline{k}}(N)} b_F(p) =-p^{-2} +B_1+B_2+
O(\frac{p^{v'_1}}{(k_1-k_2+1)(k_1-1)(k_2-2)}),$$
$$B_1=O(\frac{p^{-\frac{3}{2}}}{(k_1-1)(k_2-2)}),\ B_2=O(\frac{p^{-\frac{3}{2}}}{(k_1-k_2+1)(k_1+k_2-3)}).
$$
\end{enumerate}

\item \label{est-p^2-s}
\begin{enumerate}
\item (level-aspect) Fix $k_1,k_2$. Then for $N\gg p^{10}$, 
$$\frac{1}{d_{\underline{k},N}} \sum_{F\in HE_{\underline{k}}(N)}  b_F(p^2) =1-3p^{-1}+2p^{-2}-p^{-3}+2p^{-4}  +
 O\left(p^2f_A(p^{-1})N^{-2}\right)\\
+O\left(p^{w_1}N^{-3}\right).$$

\item (weight-aspect) Fix $N$. Then as $k_1+k_2\to\infty$,
$$\frac{1}{d_{\underline{k},N}}\sum_{F\in HE_{\underline{k}}(N) } a_F(p^2) =1-3p^{-1}+2p^{-2}-p^{-3}+2p^{-4}  +B_1+B_2+
O(\frac{p^{w'_1}}{(k_1-k_2+1)(k_1-1)(k_2-2)})\ ,$$
$$B_1=O(\frac{p^2f_{B_1}(p^{-1})}{(k_1-1)(k_2-2)}),\ B_2=O(\frac{p^2f_{B_1}(p^{-1})}{(k_1-k_2+1)(k_1+k_2-3)}),
$$

\end{enumerate}
\end{enumerate}
where $f_A(X), f_{B_1}(X),f_{B_2}(X)$ are the polynomials over $\Q$ of the degree four in $X$ whose 
coefficients are independent of $p,k_1,k_2,$ and $N$. 
\end{prop}

\section{One level-density}

We follow the exposition in \cite{CK}.
Katz and Sarnak \cite{KS1} proposed a conjecture on low-lying zeros of $L$-functions in natural families $\frak{F}$, which says that the distributions of the low-lying zeros of $L$-functions in a family $\frak{F}$ is predicted by a symmetry group $G(\frak{F})$ attached to 
$\frak{F}$: For a given entire $L$-function $L(s,\pi)$, we denote the non-trivial zeros of $L(s,\pi)$ by $\frac 12 +\gamma_j\sqrt{-1}$. Since we don't assume GRH for $L(s,\pi)$, $\gamma_j$ can be a complex number. Let $\phi(x)$ be a Schwartz function which is even and whose Fourier transform
$$
\hat{\phi}(y)=\int_{-\infty}^{\infty} \phi(x)e^{-2\pi x y\sqrt{-1} }dx
$$
has a compact support. We define
$$
D(\pi,\phi)=\sum_{\gamma_j}\phi\left( \frac{\gamma_j}{2\pi} \log c_{\pi} \right)
$$
where $c_{\pi}$ is the analytic conductor of $L(s,\pi)$. 

Let $\frak{F}(X)$ be the set of $L$-functions in $\frak F$ such that $X<c_{\pi}<2X$. The one-level density conjecture says that, for a Schwartz $\phi(x)$ which is even and whose Fourier transform $\hat{\phi}(y)$ is compactly supported,
\begin{equation*}
\lim_{ X\rightarrow \infty} \frac{1}{\# \frak{F}(X)} \sum_{\pi \in \frak{F}(X)}D(\pi,\phi)= 
\int_{-\infty}^{\infty} \phi(x)W(G(\frak{F}))\, dx,
\end{equation*}
where $W(G(\frak{F}))$ is the one-level density function.

There are five possible symmetry type of families of $L$-functions: U, SO({\rm even}), SO({\rm odd}), O, and Sp. The corresponding density functions $W(G)$ are determined in \cite{KS1}. They are

\begin{eqnarray*}
&& W({\rm U})(x)=1,\quad W({\rm SO({even})})(x)= 1+\frac{\sin 2\pi x}{ 2\pi x},\quad W({\rm O})(x)= 1+\frac 12\delta_0(x), \\
&&  W({\rm SO({odd})})(x)=1-\frac{\sin 2\pi x}{ 2\pi x} + \delta_0(x), \quad
W({\rm Sp})(x)=1-\frac{\sin 2\pi x}{ 2\pi x}.
\end{eqnarray*}
By Plancherel's formula (and because $\phi$ is even),
$$\int_{-\infty}^\infty \phi(x)W(G)(x) dx = \int_{-\infty}^\infty \hat\phi(x) \widehat W(G)(x) dx.
$$
It is useful to record that
\begin{eqnarray*}
&{}& \widehat W({\rm U})(x) = \delta_0(x),\quad \widehat W({\rm SO({even})})(x) = \delta_0(x) +\frac 12 \chi_{[-1,1]}(x),\quad \widehat W({\rm O})(x) = \delta_0(x) +\frac 12 \\
&& \widehat W({\rm SO({odd})})(x) = \delta_0(x) -\chi_{[-1,1]}(x) + 1, \quad
\widehat W({\rm Sp})(x) = \delta_0(x)-\frac 12 \chi_{[-1,1]}(x).
\end{eqnarray*}

\subsection{Degree 4 spinor $L$-functions}\label{spin-level}

We denote the non-trivial zeros of $L(s,\pi_F, {\rm Spin})$ by $\sigma_{F}=\frac{1}{2}+\sqrt{-1} \gamma_{F}$. We do not assume GRH, and hence 
$\gamma_F$ can be a complex number. Let $\phi$ be a Schwartz function which is even and whose Fourier transform has a compact support.
Define
\begin{equation*}
D(\pi_F, \phi, {\rm Spin}) = \sum_{\gamma_{F}} \phi\left( \frac{\gamma_{F}}{2\pi} \log c_{\underline{k},N}\right),
\end{equation*}
where $\log c_{\underline{k},N}=\ds\frac 1{d_{\underline{k},N}} \sum_{F\in HE_{\underline{k}}(N)} \log c(F,{\rm Spin})$ for $\underline{k}=(k_1,k_2)$, and $c(F,{\rm Spin})=(k_1+k_2)^2(k_1-k_2+1)^2 q(F)$ is the analytic conductor (cf. \cite{CK}).

\begin{prop} Assume $(N,11!)=1$. 
Let $\phi$ be a Schwartz function which is even and whose Fourier transform has a support sufficiently smaller than 
$(-1,1)$. 
$$\lim_{k_1+k_2+N\to\infty} \frac 1{d_{\underline{k},N}} \sum_{F\in HE_{\underline{k}}(N)}  D(\pi_F,\phi, {\rm Spin})=\hat\phi(0)+\frac 12 \phi(0)=
\int_\Bbb R \phi(x)W(G)(x)\, dx,
$$
where $G=$SO({even}), SO({odd}), or O type.
More precisely, let $v_1,w_1,v_1',w_1'$ be as in Proposition \ref{spin}.

\begin{enumerate} 
  \item (level aspect) Fix $k_1,k_2$. 
  Then for $\phi$ whose Fourier transform $\hat\phi$ has support in $(-u,u)$, where $u=\min\{\frac 3{4v_1+2}, \frac 3{4w_1}, \frac 1{40}\}$, as $N\to\infty$,
$$
 \frac 1{d_{\underline{k},N}} \sum_{F\in HE_{\underline{k}}(N)}  D(\pi_F,\phi, {\rm Spin})=\hat\phi(0)+\frac 12 \phi(0)+O(\frac 1{\log N}).
$$
  \item (weight aspect) Fix $N$. Then for $\phi$ whose Fourier transform $\hat\phi$ has support in $(-u,u)$, where $u=\min\{\frac 1{2v_1'+1}, \frac 1{2w_1'}\}$, as $k_1+k_2\to\infty$,
  
$$\frac 1{d_{\underline{k},N}} \sum_{F\in HE_{\underline{k}}(N)}  D(\pi_F,\phi, {\rm Spin})=\hat\phi(0)+\frac 12 \phi(0)+
O(\frac 1{\log ((k_1-k_2+2)k_1k_2)}).
$$
\end{enumerate}
\end{prop}

\begin{proof} 
For $G(s)=\phi \left( (s-\frac 12)\frac{\log c_{k,N}}{2\pi \sqrt{-1}}\right)$, by Cauchy's theorem,
\begin{eqnarray*}
D(\pi_F, \phi, {\rm Spin}) = \sum_{\gamma_F} G(\sigma_{F})=\frac{1}{2\pi \sqrt{-1}} \int_{(2)}2 G(s)\frac{\Lambda'(s,\pi_F, {\rm Spin})}{\Lambda(s,\pi_F, {\rm Spin})}ds.
\end{eqnarray*}
We have
\begin{eqnarray*}
\frac{\Lambda'(s,\pi_F, {\rm Spin})}{\Lambda(s,\pi_F, {\rm Spin})} = \frac{1}{2} \log q(F) +\psi(s+\tfrac {k_1+k_2-3}2)+\psi(s+\tfrac {k_1-k_2+1}2)
-\sum_{n=1}^{\infty}\frac{\Lambda(n)a_K(n)}{n^s}
\end{eqnarray*}
where $\psi(s)=\frac{\Gamma_\Bbb C'(s)}{\Gamma_\Bbb C(s)}$.

The contribution coming from the logarithmic derivative of $L(s,\pi_F, {\rm Spin})$ is
\begin{eqnarray}
&& \frac{1}{2\pi \sqrt{-1}} \int_{(2)} 2G(s) \left(-\sum_{n=1}^{\infty}\frac{\Lambda(n)a_F(n)}{n^s} \right)\, ds \nonumber \\
&&= -\frac{2}{\log c_{\underline{k},N}} \sum_{n=1}^{\infty}\Lambda(n)a_F(n)\frac{1}{2\pi \sqrt{-1}} \int_{(2)}\phi \left( \left(s-\frac 12\right) {2\pi \sqrt{-1}} \right) n^{-s}ds  \nonumber \\
&&= - \frac{2}{\log c_{\underline{k},N}} \sum_{n=1}^{\infty}\frac{\Lambda(n)a_F(n)}{\sqrt{n}}\int_{-\infty}^{\infty}\phi(y)e^{-y\frac{ 2\pi \log n}{\log c_{\underline{k},N}}\sqrt{-1}}dy \nonumber\\
&& = - \frac{2}{\log c_{\underline{k},N}} \sum_{n=1}^{\infty}\frac{\Lambda(n)a_F(n)}{\sqrt{n}} \widehat{\phi}\left( \frac{\log n}{\log c_{\underline{k},N}}\right).\label{p^k}
\end{eqnarray}

The contribution of the constant term $A=\frac{1}{2} \log q(F)$ is
\begin{eqnarray*}
\frac{1}{2\pi i} \int_{(2)} 2 G(s) A\, ds
= \frac{\log q(F)}{2\pi} \int_{-\infty}^{\infty} \phi \left( \frac{\log c_{\underline{k},N}}{2\pi} y\right) dy
= \frac{\log q(F)}{2\log c_{\underline{k},N}} \int_{-\infty}^{\infty} \phi(y) dy = \frac{\log q(F)}{\log c_{\underline{k},N}} \widehat{\phi}(0).
\end{eqnarray*}

For the Gamma factors' contribution, we use, for $a,t\in \Bbb R$, $a>0$, (cf. \cite{CK})

$$\frac {\Gamma'}{\Gamma}(a+t\sqrt{-1})+\frac {\Gamma'}{\Gamma}(a-t\sqrt{-1})=2\frac {\Gamma'}{\Gamma}(a)+O(t^2a^{-2}).
$$
For $\alpha\geq \frac 14$, $\frac {\Gamma'}{\Gamma}(\alpha+\tfrac 14)=\log \alpha+O(1)$. 
Hence the Gamma factors contribute
$$
\frac{2\log(k_1+k_2)+2\log (k_1-k_2+1)}{\log c_{\underline{k},N}} \widehat{\phi}(0)+ O\left(\frac{1}{\log^3 c_{\underline{k},N}}\right).
$$

It is shown in \cite{CK2} that the prime powers $p^l$, $l\geq 3$ from (\ref{p^k}), contribute
$O\left(\frac 1{\log c_{\underline{k},N}}\right)$; If $\pi_F$ satisfies the Ramanujan conjecture, $|a_{F}(n)|\leq 4$, and it is obvious.
Hence

\begin{eqnarray*}
&& \frac 1{d_{\underline{k},N}} \sum_{F\in HE_{\underline{k}}(N)} D(\pi_F, \phi, {\rm Spin})=\widehat{\phi}(0) \\
&& - \frac{2}{(\log c_{\underline{k},N}) d_{\underline{k},N}}  \sum_{F\in HE_{\underline{k}}(N)} \sum_p \frac{a_F(p)\log p}{\sqrt{p}}\widehat{\phi}\left( \frac{\log p}{\log c_{k,N}}\right)\\
&& -\frac{2}{(\log c_{\underline{k},N}) d_{\underline{k},N}} \sum_{F\in HE_{\underline{k}}(N)}  \sum_p \frac{a_F(p^2)\log p}{p}\widehat{\phi}\left( \frac{2\log p}{\log c_{\underline{k},N}}\right)+O\left( \frac{1}{\log c_{\underline{k},N}} \right). \nonumber
\end{eqnarray*}

Let $\tilde a_F(p)=a_F(p^2)+1$. We note, from the prime number theorem,
\begin{eqnarray*}
\sum_p \widehat{\phi} \left(\frac{2 \log p}{\log c_{k,N}} \right) \frac{2 \log p}{p \log c_{\underline{k},N}} = \int_2^{\infty} \widehat{\phi} \left(\frac{2 \log t}{\log c_{\underline{k},N}} \right) \frac{2 \log t}{t \log c_{k,N}} d \pi (t) + O\left( \frac{1}{\log c_{\underline{k},N}} \right)=
 \frac 12 \phi(0) + O\left( \frac{1}{\log c_{\underline{k},N}}\right).
\end{eqnarray*}

Hence
\begin{eqnarray} 
&& \frac 1{d_{\underline{k},N}} \sum_{F\in HE_{\underline{k}}(N)}  D(\pi_F,\phi)=\widehat{\phi}(0)+\frac 12 \phi(0)\nonumber \\
&& -\frac{2}{(\log c_{\underline{k},N}) d_{\underline{k},N}} \sum_{F\in HE_{\underline{k}}(N)}  \sum_p \frac{a_F(p)\log p}{\sqrt{p}}\widehat{\phi}\left( \frac{\log p}{\log c_{\underline{k},N}}\right) \label{error-p}\\
&&-\frac{2}{(\log c_{\underline{k},N}) d_{\underline{k},N}} \sum_{F\in HE_{\underline{k}}(N)}  \sum_p \frac{\tilde a_F(p)\log p}{p}\widehat{\phi}\left( \frac{2\log p}{\log c_{\underline{k},N}}\right)+O\left( \frac{1}{\log c_{\underline{k},N}} \right).\label{error-p^2}
\end{eqnarray}

Now we exchange the two sums and use Proposition \ref{spin}. Here in order to use Proposition \ref{spin}, we need to assume that 
$p\nmid N$. But $\sum_{p|N} \frac {\log p}{\sqrt{p}}\ll \sqrt{N}$ and $\sum_{p|N} \frac {\log p}p\ll \log N$. Hence they contribute to the error term. So if the support of $\hat\phi$ is $(-u,u)$ for an appropriate $u<1$,
then 
$$
 \frac 1{d_{\underline{k},N}} \sum_{F\in HE_{\underline{k}}(N)}  D(\pi_F,\phi, {\rm Spin})=\widehat{\phi}(0)+\frac 12 \phi(0) +O\left( \frac{1}{\log c_{\underline{k},N}}\right).
$$
\end{proof}

\begin{remark} Since the support of $\hat\phi$ is smaller than $(-1,1)$, we cannot distinguish the symmetry type among 
SO({even}), SO({odd}), or O type. In order to distinguish them, we need to compute the $n$-level density (cf. \cite{CK2}). We will show in an upcoming paper that when the root number $\epsilon(\pi_F)=1$, the symmetry type is SO(even); when the root number $\epsilon(\pi_F)=-1$, the symmetry type is SO(odd).
\end{remark}

\subsection{Degree 5 standard $L$-functions}

As in the degree 4 spin $L$-function case, denote the non-trivial zeros of $L(s,\pi_F,{\rm St})$ by $\sigma_{F}=\frac{1}{2}+\sqrt{-1} \gamma_{F}$. We do not assume GRH, and hence 
$\gamma_F$ can be a complex number. Let $\phi$ be a Schwartz function which is even and whose Fourier transform has a compact support.
Define
\begin{equation*}
D(\pi_F,\phi, {\rm St}) = \sum_{\gamma_{F}}\phi\left( \frac{\gamma_{F}}{2\pi} \log c_{\underline{k},st,N}\right),
\end{equation*}
where $\log c_{\underline{k},st,N}=\frac 1{d_{\underline{k},N}} \sum_{F\in HE_{\underline{k}}(N)} \log c(F,{\rm St})$, and
$c(F,{\rm St})=(k_1 k_2)^2 q(F,{\rm St})$ is the analytic conductor. 

As in the degree 4 spinor $L$-function case, we can show that

\begin{eqnarray} 
&& \frac 1{d_{\underline{k},N}} \sum_{F\in HE_{\underline{k}}(N)}  D(\pi_F,\phi, {\rm St})=\widehat{\phi}(0)-\frac 12 \phi(0)\nonumber \\
&& - \frac{2}{(\log c_{\underline{k},st,N}) d_{\underline{k},N}} \sum_{F\in HE_{\underline{k}}(N)} \sum_p \frac{b_F(p)\log p}{\sqrt{p}}\widehat{\phi}\left( \frac{\log p}{\log c_{\underline{k},st,N}}\right) \label{error-p-s}\\
&& -\frac{2}{(\log c_{\underline{k},st,N}) d_{\underline{k},N}} \sum_{F\in HE_{\underline{k}}(N)} \sum_p \frac{\tilde b_F(p)\log p}{p}\widehat{\phi}\left( \frac{2\log p}{\log c_{\underline{k},st,N}}\right)+O\left( \frac{1}{\log c_{\underline{k},st,N}} \right)\label{error-p^2-s} 
\end{eqnarray}
where $\tilde b_F(p)=b_F(p^2)-1$. (If $\pi_F$ satisfies the Ramanujan conjecture, then $|b_F(p^l)|\leq 5$ and we can show easily that the prime powers $p^l$, $l\geq 3$, contribute to $O\left( \frac{1}{\log c_{\underline{k},st,N}} \right)$. In the appendix, we show it without the Ramanujan conjecture.)

By interchanging two sums and using Proposition \ref{standard} as in Section \ref{spin-level}, we see that if the support of $\hat\phi$ is $(-u,u)$ for an appropriate $u<1$,
$$\frac 1{d_{\underline{k},N}} \sum_{F\in HE_{\underline{k}}(N)}  D(\pi_F,\phi, {\rm St})=\widehat{\phi}(0)-\frac 12 \phi(0)
+O\left( \frac{1}{\log c_{\underline{k},st,N}} \right).
$$
Hence we have proved 

\begin{prop} Let $\phi$ be a Schwartz function which is even and which its Fourier transform has a support sufficiently smaller than 
$(-1,1)$. 
$$\lim_{k_1+k_2+N\to\infty} \frac 1{d_{\underline{k},N}} \sum_{F\in HE_{\underline{k}}(N)}  D(\pi_F,\phi, {\rm St})=\hat\phi(0)-\frac 12 \phi(0)=
\int_\Bbb R \phi(x)W({\rm Sp})(x)\, dx.
$$
More precisely, let $v_1,w_1,v_1',w_1'$ be as in Proposition \ref{standard}.

\begin{enumerate} 
  \item (level aspect) Fix $k_1,k_2$. 
  Then for $\phi$ whose Fourier transform $\hat\phi$ has support in $(-u,u)$, where 
$u=\min\{\frac 3{28v_1+14}, \frac 3{w_1},\frac 1{840}\}$, as $N\to\infty$,
$$
 \frac 1{d_{\underline{k},N}} \sum_{F\in HE_{\underline{k}}(N)}  D(\pi_F,\phi,{\rm St})=\hat\phi(0)-\frac 12 \phi(0)+O(\frac 1{\log N}).
$$
  \item (weight aspect) Fix $N$. Then for $\phi$ whose Fourier transform $\hat\phi$ has support in $(-u,u)$, where $u=\min\{\frac 1{2v_1'+1}, \frac 1{2w_1'}, \frac 18\}$, as $k_1+k_2\to\infty$,
  
$$\frac 1{d_{\underline{k},N}} \sum_{F\in HE_{\underline{k}}(N)}  D(\pi_F,\phi,{\rm St})=\hat\phi(0)-\frac 12 \phi(0)+O(\frac 1{\log ((k_1-k_2+2)k_1k_2)}).
$$
\end{enumerate}

\end{prop}

\section{Stable vs unstable pseudo-coefficients}\label{Shin}
In this section we compare Shin's results \cite{Shin} with ours. This would explain how 
using a single pseudo-coefficient violates a symmetry, and 
how the defect corresponds to the non-semisimple contributions on the geometric side and non-holomorphic endoscopic lifts on the spectral side. 

Let $(l_1,l_2)=(k_1-1,k_2-2)$ be the Harish-Chandra parameter and let $D^{{\rm large}}_{l_1,l_2}$ be the large discrete series of $G(\R)=GSp_4(\R)$ so that 
$\{D^{{\rm hol}}_{l_1,l_2}, D^{{\rm large}}_{l_1,l_2}\}$ makes up an L-packet of $\prod(G(\R))$ 
(see Section 2.3 of \cite{Wakatsuki1} for $D^{{\rm large}}_{l_1,l_2}$ and \cite{Oda} for an interpretation as $C^\infty$ classical forms).
Note that irreducible components of $D^{{\rm hol}}_{l_1,l_2}|_{Sp_4(\R)}$ and $D^{{\rm large}}_{l_1,l_2}|_{Sp_4(\R)}$ 
form an L-packet of $\prod(Sp_4(\R))$ which consists of four elements. 

Let $\widetilde S_{\underline{k}}(N)$ be the set introduced in Remark \ref{Shin-T}. 

Suppose $\pi\in \widetilde S_{\underline{k}}(N)$ is not a CAP form, and non-endoscopic. Let $\pi=\pi_\infty\otimes\pi_f$. 
By Weissauer \cite{Wei1}, if $\pi_\infty\simeq D_{l_1,l_2}^{\rm hol}$, there exists a cuspidal representatin $\pi'=\pi_\infty'\otimes\pi_f'$ such that $\pi_\infty'\simeq D_{l_1,l_2}^{\rm large}$ and $\pi_f'\simeq \pi_f$. Conversely, if $\pi_\infty\simeq D_{l_1,l_2}^{\rm large}$, there exists a cuspidal representatin $\pi'=\pi_\infty'\otimes\pi_f'$ such that $\pi_\infty'\simeq D_{l_1,l_2}^{\rm hol}$ and $\pi_f'\simeq \pi_f$. 

Now suppose $\pi$ is endoscopic and $\pi_\infty\simeq D_{l_1,l_2}^{\rm hol}$. Then by Roberts \cite{Roberts}, there exists a cuspidal representatin $\pi'=\pi_\infty'\otimes\pi_f'$ such that $\pi_\infty'\simeq D_{l_1,l_2}^{\rm large}$ and $\pi_f'\sim \pi_f$. (Here $\sim$ means weak equivalence, and in fact equivalent outside the ramification of $\pi$.)

However, if $\pi$ is endoscopic and $\pi_\infty\simeq D_{l_1,l_2}^{\rm large}$, there does not exist a cuspidal representation $\pi'$ such that $\pi_\infty'\simeq D_{l_1,l_2}^{\rm large}$ and $\pi_f'\sim \pi_f$. (For example, we cannot construct a holomorphic Siegel cusp form from a pair of two elliptic cusp forms of level 1, but we can construct a cuspidal representation with the infinity type $D_{l_1,l_2}^{\rm large}$.)

Therefore, holomorphic Siegel cusp forms always appear in pairs with cuspidal representations with the infinity type $D_{l_1,l_2}^{\rm large}$, but there are cuspidal representations with the infinity type $D_{l_1,l_2}^{\rm large}$, which do not appear in pairs.
Let $\widetilde S_{\underline{k}}(N)^{{\rm en,large}}$ be the subset of $\widetilde S_{\underline{k}}(N)$ consisting of $\Pi$ such that 
$\Pi$ is endoscopic and $\Pi_\infty$ is isomorphic to the large discrete series $D^{{\rm large}}_{l_1,l_2}$.    
Then the same argument in Section \ref{class-endo} works for the theta lift from $GSO(2,2)$ to $GSp_4$ and we have 
$$\dim \widetilde S_{\underline{k}}(N)^{{\rm en,large}}=O((k_1-k_2+1)(k_1+k_2-3)N^{8+\epsilon}),\quad {\rm as}\ k_1+k_2+N\to \infty.
$$     
Therefore   
$$\ds\frac {\dim \widetilde S_{\underline{k}}(N)^{{\rm en, large}}}{\dim S_{\underline{k}}(N)}=O(((k_1-1)(k_1-2))^{-1}N^{-2+\epsilon}), \ {\rm as}\ k_1+k_2+N\to \infty,
$$
which might be related to the second main term $A, B_1$ of Theorem \ref{main}. 

For $\ast \in \{{\rm hol},{\rm large}\}$ and each $D^{\ast}_{l_1,l_2}$,  
we choose a pseudo coefficient $f^\ast_{\xi_{\underline{k}}}\in C^\infty_c(G(\R))$. 
Put $f^{{\rm tot}}_{\xi_{\underline{k}}}:=f^{{\rm hol}}_{\xi_{\underline{k}}}+f^{{\rm large}}_{\xi_{\underline{k}}}$, where we may  
call it ``stable" pseudo-coefficient. (This is called Euler-Poincar\'e function in \cite{Shin}.) 
Note that if we work on $Sp_4$, we would consider  
$f^{{\rm tot}}_{\xi_{\underline{k}}}:=f^{{\rm hol}}_{\xi_{\underline{k}}}+f^{{\rm large}}_{\xi_{\underline{k}}}+ 
f^{{\rm anti-large}}_{\xi_{\underline{k}}}+f^{{\rm anti-hol}}_{\xi_{\underline{k}}}$. 

As Shin \cite{Shin} did, by using $f^{{\rm tot}}_{\xi_{\underline{k}}}$ in the Arthur-Selberg trace formula, we can avoid 
the non-semisimple contributions. However the trace ${\rm tr}(f^{{\rm tot}}_{\xi_{\underline{k}}})$ collects various 
automorphic forms both holomorphic cusp forms and non-holomorphic cusp forms. 
On the other hand in this paper we used a single pseudo-coefficient $f^{{\rm hol}}_{\xi_{\underline{k}}}$ 
to collect only holomorphic cusp forms. Such a pseudo-coefficient might be called ``unstable".  
Thereby we had to calculate non-semisimple contributions whose 
behaviors have not been understood well.    

Let $\widehat{\mu}_{{\rm Shin}}$ be the measure for $U=K(N)$ introduced in \cite{Shin}. 
As in Section \ref{pm}, we can define, by using a pseudo-coefficient of  
$D^{{\rm large}}_{l_1,l_2}$ (cf. Section 2.3 of \cite{Wakatsuki1}),
the counting measure $\widehat{\mu}^{{\rm en}}_{K(N),\xi_{\underline{k}}, D^{{\rm large}}_{l_1,l_2}}$  on 
$\widetilde S_{\underline{k}}(N)^{{\rm en,large}}$.

It is not difficult to estimate non-holomorphic residual spectrum part as in Section \ref{resi}. 
Then we would have the following: 
for any $f=f_{S}$ in Proposition \ref{semisimple-est}, the difference 
$$\widehat{\mu}_{{\rm Shin}}(\widehat{f})
-2\widehat{\mu}_{K(N),\xi_{\underline{k}},D^{{\rm hol}}_{l_1,l_2}}(\widehat{f})=
\widehat{\mu}^{en}_{K(N),\xi_{\underline{k}},D^{{\rm large}}_{l_1,l_2}}(\widehat{f})+({\rm remainder})
$$ 
would be 
\begin{enumerate}
\item (level-aspect) 
$A+O(p^{a\kappa+b}_{S'}\varphi(N)N^{-3})$, as $N\to\infty$;
\item  (weight-aspect) 
$B_1+B_2+O(\ds\frac{p^{a'\kappa+b'}_{S'}}{(k_1-k_2+1)(k_1-1)(k_2-2)})$, as $k_1+k_2\to\infty$,
\end{enumerate}
where $A, B_1, B_2$ are the second main terms in Theorem \ref{main}. Note also that the remainder in RHS comes from CAP 
representations and residual spectrum, and it would be subsumed into the error term.
 
Hence we expect that the second terms $A, B_1$ correspond to $S_{\underline{k}}(\G(N))^{{\rm en,large}}$. However, $B_2$ is still 
mysterious and it seems interesting to figure out what kind of representations contribute to $B_2$. 
We will confirm the above speculation elsewhere.  

\section{Appendix} 
We prove 
$$\sum_{p^l,\, l\geq 3} \frac{|b_F(p^l)|\log p}{p^{\frac l2}},
$$
converges, without the assumption of Ramanujan conjecture.

Recall the bound on the Satake parameters \cite{LRS}: 
Let $\pi$ be a cuspidal representation of $GL_m$, and let
 $\{\alpha_1(p),...,\alpha_m(p)\}$ be Satake parameters. Then
$|\alpha_i(p)|\leq p^{\frac 12-\frac 1{m^2+1}}.$

Hence $|b_F(p^l)|\leq 5 p^{\frac l2-\frac l{26}}$.
So
$$
\sum_p \sum_{l\geq 52} \frac {|b_F(p^l)|\log p}{p^\frac l2}\ll \sum_p \log p \sum_{l\geq 52} (p^{-\frac 1{26}})^l
\ll 
\sum_p \frac {\log p}{p^2}=O(1).
$$

Hence it is enough to prove that for each $l\geq 3$, the series $\sum_{p} \frac{|b_F(p^l)|\log p}{p^{\frac l2}}$ converges.

Recall the classification of spherical generic unitary representations of $GSp_4(\Bbb Q_p)$:

\begin{enumerate}
\item $L(\mu_1,\mu_2, \eta)$; $\mu_1, \mu_2, \eta$ are unitary characters;
\item $L(\nu^\beta \mu, \nu^{\beta}\mu^{-1}, \nu^{-\beta}\eta)$; $\mu, \eta$ are unitary characters and $\mu^2\ne 1$ and $0<\beta<\frac 12$;
\item $L(\nu^\beta, \mu, \nu^{-\frac {\beta}2}\eta)$; $\mu\ne 1, \nu$ are unitary characters and $0<\beta<1$;
\item $L(\nu^{\beta_1} \mu, \nu^{\beta_2}\mu, \nu^{-\frac {\beta_1+\beta_2}2}\eta)$; $\mu, \eta$ are unitary characters and $\chi^2=1$ and $0<\beta_2\leq \beta_1<1$, $\beta_1+\beta_2<1$.
\end{enumerate}

Hence Satake parameters of $\Pi_p$ are of the form

\begin{enumerate}
\item $S_1: 1, \alpha_{1p}, \alpha_{2p}, \alpha_{1p}^{-1}, \alpha_{2p}^{-1}$, where $|\alpha_{ip}|=1$;
\item $S_2: 1, p^\beta \alpha_p, p^\beta \alpha_p^{-1}, p^{-\beta} \alpha_p, p^{-\beta} \alpha_p^{-1}$, where $|\alpha_p|=1$;
\item $S_3: 1, p^\beta, p^{-\beta}, \alpha_{p}, \alpha_{p}^{-1}$, where $|\alpha_p|=1$;
\item $S_4: 1, p^{\beta_1} \alpha_p, p^{\beta_2} \alpha_p, p^{-\beta_1} \alpha_p^{-1}, p^{-\beta_2} \alpha_p^{-1}$, where $|\alpha_p|=1$.
\end{enumerate}

Clearly, $\sum_{p\in S_1} \frac {|b_F(p^l)|\log p}{p^{\frac l2}}$ converges.

For $S_2$, note that
$$|b_F(p^l)|=|1+(p^{l\beta}+p^{-l\beta})(\alpha_p^l+\alpha_p^{-l})|\leq 2p^{l\beta}+3.
$$
and $|a_F(p)|=|p^{\beta}+p^{-\beta}+\alpha_p+\alpha_p^{-1}|\geq p^\beta-3$. Hence
$$|b_F(p^l)|\ll |a_F(p)|^l\ll |a_F(p)|^2 p^{\frac {l-2}2-\frac {l-2}{17}}.
$$
Therefore,
$$\sum_{p\in S_2} \frac {|b_F(p^l)|\log p}{p^{\frac l2}}\ll \sum_p \frac {|a_F(p)|^2}{p^{1+\frac {l-2}{17}}}.
$$
Since $L(s,\pi\times\pi)$ converges absolutely for $Re(s)>1$, the above series converges.

For $S_3$, note that
$$1+p^{\beta}+p^{-\beta}+|\alpha_p|+|\alpha_p^{-1}|\leq |1+p^{\beta}+p^{-\beta}+\alpha_p+\alpha_p^{-1})|+6.
$$
Hence $|b_F(p^l)|\ll |b_F(p)|^l\ll |b_F(p)|^2 p^{\frac {l-2}2-\frac {l-2}{26}}$. So
$$\sum_{p\in S_3} \frac {|b_F(p^l)|\log p}{p^{\frac l2}}\ll \sum_p \frac {|b_F(p)|^2}{p^{1+\frac {l-2}{26}}}.
$$
Since $L(s,\Pi\times\Pi)$ converges absolutely for $Re(s)>1$, the above series converges.

For $S_4$, note that
$$1+p^{\beta_1}+p^{-\beta_1}+p^{\beta_2}+p^{-\beta_2}\leq |1+p^{\beta_1}\alpha_p+p^{\beta_2}\alpha_p+ p^{-\beta_1}\alpha_p^{-1} +p^{-\beta_2}\alpha_p^{-1}|+6.
$$
Hence $|b_F(p^l)|\ll |b_F(p)|^l$, and it is similar to $S_2$ case.

\end{document}